\documentclass[11pt,reqno]{amsart}
\usepackage{amsmath}
\usepackage{amssymb}
\usepackage{amsthm}
\usepackage{color}
\usepackage{eucal}
\usepackage{tikz}
\usepackage{gastex}
\usepackage{stmaryrd}
\usepackage{caption,subcaption}
\usepackage{latexsym}
\usepackage{indentfirst}
\usepackage{graphicx,accents}

\usetikzlibrary{decorations.pathmorphing}
\usetikzlibrary{shapes.geometric}

\DeclareSymbolFont{rsfscript}{OMS}{rsfs}{m}{n}
\DeclareSymbolFontAlphabet{\mathrsfs}{rsfscript}

\numberwithin{equation}{section}
\newtheorem{prop}{Proposition}[section]

\newtheorem{lem}[prop]{Lemma}
\newtheorem{cor}[prop]{Corollary}

\def\Jc{\mathrel{\mathrsfs{J}}}
\def\Dc{\mathrel{\mathrsfs{D}}}
\def\Hc{\mathrel{\mathrsfs{H}}}
\def\Lc{\mathrel{\mathrsfs{L}}}
\def\Rc{\mathrel{\mathrsfs{R}}}
\def\Kc{\mathrel{\mathrsfs{K}}}
\def\A{\mathfrak{A}}
\def\B{\mathfrak{B}}
\def\E{\mathrsfs{E}}
\def\lE{\overline{\mathrsfs{E}}}
\def\oE{\vec{\mathrsfs{E}}} 
\def\V{\mathrsfs{V}}
\def\Sl{\mathrsfs{S}_l}
\def\Sr{\mathrsfs{S}_r}
\def\Ga{\Gamma}
\def\l{\ell}
\def\la{\lambda}
\def\al{\alpha}
\def\be{\beta}

\def\cb{\mathbf{c}}
\def\sb{\mathbf{s}}
\def\lv{\mathfrak{l}}
\def\rv{\mathfrak{r}}

\def\av{\mathfrak{a}}

\def\bv{\mathfrak{b}}

\def\ev{\mathfrak{e}}
\def\vev{\vec{\ev}}
\def\oev{\ol{\ev}}
\def\fv{\mathfrak{f}}

\def\w{\mathsf{w}}
\def\Rcc{\mathrm{R}}
\def\Hcc{\mathrm{H}}
\def\Lcc{\mathrm{L}}
\def\Dcc{\mathrm{D}}

\def\Kcc{\mathrm{K}}
\def\rvec{\overrightarrow}
\def\lvec{\overleftarrow}

\def\ol{\overline}

\newcommand{\oX}{\ol{X}}
\newcommand{\wX}{\widehat{X}}

\def\ler{\mathrel{{\le}_{\Rc}}}
\def\lel{\mathrel{{\le}_{\Lc}}}
\def\leh{\mathrel{{\le}_{\Hc}}}
\def\lek{\mathrel{{\le}_{\Kc}}}
\def\wr{\mathrel{\omega^r}}
\def\wl{\mathrel{\omega^l}}

\DeclareMathOperator{\fg}{FG}
\DeclareMathOperator{\fim}{FIM}
\DeclareMathOperator{\bfli}{BFLI}
\DeclareMathOperator{\LI}{LI}
\DeclareMathOperator{\I}{Inv}
\DeclareMathOperator{\G}{G}
\DeclareMathOperator{\Aut}{Aut}

\renewcommand{\iff}{if and only if}

\title[]{A combinatorial approach to the structure of locally inverse semigroups}
\author{Lu\'\i s Oliveira}
\address{Departamento de Matem\'atica,
Faculdade de Ci\^encias da Universidade do Porto,
R. Campo Alegre, 687, 4169-007 Porto, Portugal}
\email{loliveir@fc.up.pt}

\begin{document}

\begin{abstract}
This paper introduces a notion of presentation for locally inverse semigroups and develops a graph structure to describe the elements of locally inverse semigroups given by these presentations. These graphs will have a role similar to the role that Cayley graphs have for group presentations or that Sch\"utzenberger graphs have for inverse monoid presentations. However, our graphs have considerable differences with the latter two, even though locally inverse semigroups generalize both groups and inverse semigroups. For example, the graphs introduced here are not `inverse word graphs'. Instead, they are bipartite graphs with both oriented and non-oriented edges, and with labels on the oriented edges only. A byproduct of the theory developed here is the introduction of a graphical method for dealing with general locally inverse semigroups. These graphs are able to describe, for a locally inverse semigroup given by a presentation, many of the usual concepts used to study the structure of semigroups, such as the idempotents, the inverses of an element, the Green's relations, and the natural partial order. Finally, the paper ends characterizing the semigroups belonging to some usual subclasses of locally inverse semigroups in terms of properties on these graphs.
\end{abstract} 

\subjclass[2010]{(Primary) 20M17, 20M10 (Secondary) 05C25, 20M18, 20M05}
\keywords{Locally inverse semigroup, Presentation, Bipartite graph, Language}

\maketitle

\section{Introduction}

Let $S$ be a semigroup. The set of idempotents of $S$ is denoted by $E(S)$. Contrarily to the Group Theory case, the study of $E(S)$ is an important tool in Semigroup Theory. Another important concept is the notion of inverse of an element $s\in S$. An element $s'\in S$ is an \emph{inverse} of $s$ if $ss's=s$ and $s'ss'=s'$. If $S$ is a group, this notion of inverse is equivalent to the usual notion of inverse in Group Theory. We denote the set of inverses of $s$ in $S$ by $V(s)$. The set $V(s)$ may be empty or have multiple elements for a given $s$. The semigroup $S$ is \emph{regular} if $V(s)$ is non-empty for all $s\in S$, and it is \emph{inverse} if $V(s)$ has exactly one element for each $s\in S$.

A common alternative characterization of inverse semigroups is as regular semigroups whose idempotents form a subsemilattice. An inverse monoid is an inverse semigroup with an identity element. The class of all inverse monoids constitutes a variety of algebras of type $\langle 2,1,0\rangle$ where the unary operation is the operation of taking the inverse and the nullary operation chooses the identity element.

Throughout this paper, $X$ will denote a non-empty set and $X'=\{x'\,:\; x\in X\}$ will denote a disjoint copy of $X$. Let $\oX=X\cup X'$. If $y=x'\in X'$, then $y'$ denotes $x$. Let $\oX^*$ and $\oX^+$ be respectively the free monoid and the free semigroup on $\oX$. Thus $\oX^+$ is the set of all non-empty `words' on the `alphabet' $\oX$ equipped with the concatenation operation, and $\oX^*$ is just $\oX^+$ with the empty word $\iota$ adjoined. Let $\sigma$ be the congruence generated by $\{(yy',\iota)\,:\;y\in\oX\}$ and let $\varrho$ be the Wagner congruence, the least inverse monoid congruence, both on $\oX^*$. Then $\oX^*/\sigma$ is the free group $\fg(X)$ while $\oX^*/\varrho$ is the free inverse monoid $\fim(X)$, both on $X$.

An  \emph{inverse monoid presentation} [\emph{group presentation}] is a pair $P=\langle X;R\rangle$ where $R$ is a relation on $\oX^*$. The inverse monoid presented by $P$ is the inverse monoid $S=\oX^*/\rho$ where $\rho$ is the congruence generated by $\varrho\cup R$; and the group presented by $P$ is the group $G=\oX^*/\theta$ where $\theta$ is the congruence generated by $\sigma\cup R$. We denote $S$ by $\I\langle X;R\rangle$ and $G$ by $\G\langle X;R\rangle$. The \emph{word problem} for an inverse monoid presentation [group presentation] $P$ consists of knowing if there exist an algorithm to decide, given two words $u$ and $v$ from $\oX^*$ as input, whether $(u,v)\in\rho$ [$(u,v)\in\theta$] or not. It is well known that the word problem is undecidable for many group and inverse monoid presentations, that is, no such algorithm exists for many cases.

The structure of $\fim(X)$ was first described by Scheiblich \cite{sch73a, sch73b}. Another important description was given by Munn \cite{munn74}. Munn's approach uses graphical methods, the so called \emph{Munn trees}. Stephen \cite{stephen90} used Munn's idea to develop a graphical method to attack the word problem for inverse monoid presentations. He uses \emph{inverse word graphs}, a natural generalization of Munn trees obtained by dropping the condition of being a tree, to describe the structure of any inverse monoid given by an inverse monoid presentation. These inverse word graphs, called \emph{Sch\"utzenberger graphs}, play for inverse monoid presentations a role similar to the role that Cayley graphs play for group presentations. Both these graphs have revealed themselves to be very important tools in many contexts.

Apart from groups, inverse semigroups and finite semigroups are the two most studied classes of semigroups. The interest in inverse semigroups comes from many different areas. For example, they appear naturally when dealing with partial one-to-one transformations, or they are the natural algebras to consider when generalizing the notion of group preserving the uniqueness of inverses. But, the existence of graphical methods, such as the ones developed by Munn and Stephen, has also boosted the study of these semigroups, specially in Combinatorial Inverse Semigroup Theory. The lack of similar good graphical methods for other classes of semigroups has been a hindrance in the study of those classes. 

In this paper we introduce a notion of presentation for locally inverse semigroups. The goal is to develop a graph structure to describe the elements of a locally inverse semigroup given by a locally inverse semigroup presentation. A byproduct of the theory developed here is the introduction of graphical methods to address general locally inverse semigroups.

A \emph{locally inverse semigroup} is a regular semigroup $S$ where every local submonoid, that is, a submonoid of the form $eSe$ with $e\in E(S)$, is an inverse semigroup. The class $\bf LI$ of all locally inverse semigroups is an \emph{e-variety} of regular semigroups, a class of regular semigroups closed for homomorphic images, direct products and regular subsemigroups \cite{ha1,ks1}, with bifree objects on every set $X$ (see \cite{yeh} for the definition and existence of bifree objects in $\bf LI$ on every set $X$). 

The structure of the bifree locally inverse semigroup $\bfli(X)$ on a set $X$ has been studied in \cite{auinger94,auinger95,billhardt,ol17}. In \cite{ol17}, we gave a graph description for the elements of $\bfli(X)$ that can be viewed as the analogue for $\bfli(X)$ of the Munn tree representation for the elements of $\fim(X)$. In the present paper, we adapt those graphs so that we can describe the elements of a locally inverse semigroup given by a locally inverse semigroup presentation. The work done here generalizes \cite{ol17} in a similar manner as Stephen's work \cite{stephen90} generalizes Munn's approach to the free inverse monoid. 

However, the graphs considered here have considerable differences when compared with both Cayley graphs and Sch\"utzenberger graphs, even though locally inverse semigroups generalize both groups and inverse semigroups. For example, they are not `inverse word graphs' as the latter two are. Instead, the graphs introduced here are bipartite graphs with both oriented and non-oriented edges, and with labels on the oriented edges only.

Bipartite graphs have already been used to study completely 0-simple semigroups. Graham \cite{graham} and Houghton \cite{houghton}, independently, associated bipartite graphs to each regular Rees matrix semigroup representing a completely 0-simple semigroup. The vertices of these bipartite graphs are in one-to-one correspondence with the set of $\Rc$-classes and $\Lc$-classes of the non-trivial $\Dc$-class of the completely 0-simple semigroups. More precisely, the vertices are partitioned into `left' and `right' vertices, where the left vertices represent the $\Rc$-classes and the right vertices represent the $\Lc$-classes. Graham and Houghton used these graphs to study completely 0-simple semigroups. 

Recently, Reilly \cite{reilly} introduced the notion of a \emph{fundamental semigroup} associated with a bipartite graph $\Ga$. If the vertices of $\Ga$ are partitioned into the sets $V_1$ of left vertices and $V_2$ of right vertices, then let $P$ be the set of all walks from a left vertex to a right vertex. If $p_1,p_2\in P$, then define $p_1\cdot p_2=p_1ep_2\in P$ if $e$ is the edge connecting the right vertex of $p_1$ with the left vertex of $p_2$; otherwise, set $p_1\cdot p_2=0$. Reilly used the previous operation on $P\cup\{0\}$ to introduce the fundamental semigroup of the bipartite graph. These semigroups are 0-direct unions of idempotent generated completely 0-simple semigroups. Reilly applied this concept on the graphs introduced by Graham and Houghton, and showed that the fundamental semigroups obtained have a sort of universal property with respect to completely 0-simple semigroups. The reader should consult \cite{reilly} for more details.

In this paper, the bipartite graphs we shall consider have a different provenance. The vertices are no longer determined by the set of $\Rc$-classes and $\Lc$-classes. Instead, the left vertices of the graphs considered here correspond to the elements of an $\Lc$-class, while the right vertices correspond to the elements of an $\Rc$-class that intersects the previous $\Lc$-class. Thus, our graphs shall have more vertices in general than the Graham and Houghton's graphs. Also, we will not follow Reilly's approach trying to describe the elements of the locally inverse semigroups as walks in a bipartite graph. Instead, the elements of the locally inverse semigroups will be described as concrete bipartite graphs with two distinguished vertices. We will not distinguish any particular walk between the two distinguished vertices. This approach will allow us to introduce operations on these bipartite graphs that capture the structure of the locally inverse semigroup. Finally, as mentioned already, our bipartite graphs will have two distinct types of edges: non-oriented edges and oriented edges, the latter with labels. The non-oriented edges will describe the inverses and the idempotents inside a $\Dc$-class, while the labeled oriented edges will gather information about the partial multiplication with the elements of a generating set inside a $\Dc$-class.

In the next section, we recall some concepts and notations used in Semigroup Theory that will be useful to us. In Section 3, we introduce the concept of presentation for locally inverse semigroups. Before we continue with the study of these presentations, we introduce the graphs we shall need and study some of their properties. This will be done in Sections 4 and 5. More precisely, in Section 4, we introduce the birooted locally inverse word graphs and define the language recognized by them; and in Section 5, we study the `reduced' birooted locally inverse word graphs.

We should point out that the notion of language recognized by birooted locally inverse word graphs is crucial for the theory developed here, not only for Sections 4 and 5, but also for the subsequent sections. As usual, this notion comes from associating words to walks inside the graphs, but it has its own peculiarities. Firstly, each walk can have several associated words. Secondly, the words are on a larger alphabet than $\oX$, namely
$$\wX=\oX\cup\{(x\wedge y)\,|\; x,y\in\oX\}\,.$$
Thus, the language recognized by a birooted locally inverse word graph will be a subset of the free monoid $\wX^*$. 

In Section 6, we associate a reduced locally inverse word graph to each idempotent $e$ of a locally inverse semigroup $S$ given by a presentation. We see that these graphs characterize the $\Dc$-classes of $S$ and that they are isomorphic if and only if the associated idempotents belong to the same $\Dc$-class. In Section 7, we introduce the reduced birooted locally inverse word graph associated with a given element $s\in S$. We just need to add two roots to the graphs described in Section 6. We show that no two distinct elements of $S$ have the associated reduced birooted locally inverse word graphs isomorphic, and we analyze the relation between these graphs and the structure of $S$. In particular, we describe the idempotents, the inverses of an element, the Green's relations, the natural partial order and the product on $S$ using these graphs. Finally, in the last section, we describe some special classes of locally inverse semigroups in terms of particular properties on the reduced locally inverse word graphs.

\section{Preliminaries}

For concepts and notations left undefined in this paper, the reader should consult \cite{howie}. In this paper, $S$ and $\leq$ will always denote a regular semigroup and its \emph{natural partial order}, respectively. Recall that, $s\leq t$ for $s,t\in S$ \iff\ there exist idempotents $e,f\in E(S)$ such that $s=te=ft$. Let $(s]_\leq$ and $[s)_\leq$ denote respectively the principal ideal and the principal filter generated by $s\in S$ for the natural partial order. Note that $E(S)$ is an ideal with respect to the natural partial order.

As usual, $\Rc$, $\Lc$, $\Hc$, $\Dc$ and $\Jc$ denote the five \emph{Green's relations} on $S$, and $\Kcc_s$ denotes the $\Kc$-class of $s\in S$ for $\Kc\in\{\Rc,\Lc,\Hc,\Dc,\Jc\}$. Let $\ler$ be the quasiorder defined by
$$s\ler t \quad\Leftrightarrow\quad sS \subseteq tS\,.$$
Let $\lel$ be its dual relation and $\leh\,=\,\ler\cap\lel$. Consider also the quasiorder $\leq_{\Jc}$ defined by
$$s\leq_{\Jc} t\quad\Leftrightarrow\quad SsS\subseteq StS.$$ 
If $\mathrel{{\geq}_{\Rc}}$, $\mathrel{{\geq}_{\Lc}}$, $\mathrel{{\geq}_{\Hc}}$ and $\mathrel{{\geq}_{\Jc}}$ denote the respective reverse relations, then 
$$\Rc\,=\,\ler\cap\mathrel{{\geq}_{\Rc}},\;\;\Lc\,=\,\lel\cap\mathrel{{\geq}_{\Lc}},\;\;\Hc\,=\,\leh\cap\mathrel{{\geq}_{\Hc}}\;\mbox{ and }\;\Jc\,=\, \mathrel{{\leq}_{\Jc}}\cap\mathrel{{\geq}_{\Jc}}.$$ Let $(s]_{\Kc}$ and $[s)_{\Kc}$ denote respectively the principal ideal and the principal filter generated by $s\in S$ for the relations $\lek$ where $\Kc\in\{\Rc,\Lc,\Hc,\Jc\}$.

If $st\Rc s$ for $s,t\in S$, then the mapping $\varphi_t:\Lcc_s\to \Lcc_{st},\; u\mapsto ut$ is a well-defined bijection preserving $\Rc$-related elements by the Green's Lemma (see \cite{howie}). The mapping $\varphi_t$ is called the right translation of $\Lcc_s$ associated with $t$. Dually, if $st\Lc t$, then the mapping $\psi_s:\Rcc_t\to \Rcc_{st},\; u\mapsto su$ is a well-defined bijection preserving $\Lc$-related elements, called the left translation of $\Rcc_t$ associated with $s$.  

Let $\wr$ and $\wl$ be the following relations on $E(S)$:
$$e\wr f \;\Leftrightarrow\; fe=e\qquad\mbox{ and }\qquad e\wl f \;\Leftrightarrow\; ef=e\,.$$
Consider also the relation $\omega = \wr\cap\wl$ on $E(S)$. Once more, for $e\in E(S)$, $(e]_r$, $(e]_l$ and $(e]\,$ [$\,[e)_r$, $[e)_l$ and $[e)\,$] denote the principal ideals [filters] generated by $e$ for the relations $\wr$, $\wl$ and $\omega$, respectively. Then $\wr$ [$\wl$] coincides with the restriction of $\ler\,$ [$\lel$] to $E(S)$, while $\omega$ coincides with the restriction of both $\leq$ and $\leh$ to $E(S)$.

The locally inverse semigroups can be described using the relations $\wr$, $\wl$ and $\omega$. A regular semigroup $S$ is locally inverse if for each pair $(e,f)\in E(S)\times E(S)$, there exists $g\in E(S)$ such that $(e]_r\cap (f]_l=(g]$. Since $\omega$ is a partial order, the idempotent $g$ is unique for each pair $(e,f)$. If we consider the binary operation $\wedge$ on $E(S)$ defined by $e\wedge f=g$, then the algebra $(E(S),\wedge)$ is called the \emph{pseudosemilattice of idempotents} of $S$. Pseudosemilattices were introduced by Nambooripad \cite{na1}, who also gave an abstract characterization for them.

The operation $\wedge$ can be extended naturally to the whole locally inverse semigroup $S$. It is well known that $e\wedge f=e_1\wedge f_1$ if $e\Rc e_1$ and $f\Lc f_1$. Thus, without any ambiguity, we can define $s\wedge t$ for $s,t\in S$ as follows: 
$$s\wedge t=ss'\wedge t't \quad\mbox{ where }\; s'\in V(s)\;\mbox{ and }\; t'\in V(t)\,.$$ 
We shall look at the locally inverse semigroups as binary semigroups, that is, algebras of type $\langle 2,2\rangle$ where the first operation is associative.

Let $S^1$ be the monoid obtained from the semigroup $S$ by adjoining an identity element $1$ if necessary. A locally inverse semigroup with an identity element must be an inverse semigroup. Thus, if $S$ is a non-inverse locally inverse semigroup, then $S^1$ is no longer locally inverse. However, for technical reasons, it will be often convenient to consider the locally inverse semigroups with an identity adjoined. Therefore, we will use the terminology locally inverse monoid with a different meaning than usual. In the context of this paper, a \emph{locally inverse monoid} is a regular monoid $T$ with identity $1$ such that $T\setminus\{1\}$ is a locally inverse semigroup. 

We end this section with two lemmas about locally inverse semigroups that will be useful later. Note that, in any semigroup $S$, if $e,f\in E(S)$, then $eS\cap Sf=eSf$. It is also well known that, if $S$ is locally inverse, then no two idempotents of $(f]_r$ are $\Lc$-related and no two idempotents of $(f]_l$ are $\Rc$-related.

\begin{lem}\label{inv_lis}
Let $S$ be a locally inverse semigroup and $a\in fSe$ for some $e,f\in E(S)$. Then $|V(a)\cap eSf|=1$. 
\end{lem}

\begin{proof}
Let $a'\in V(a)$. Then $aa'\in (f]_r$ and $f_1=aa'f\in (f]\cap \Rcc_a$. Similarly $e_1=ea'a\in (e]\cap \Lcc_a$. Hence $\Rcc_{e_1}\cap\Lcc_{f_1}$ contains an inverse of $a$, and so $V(a)\cap eSf\neq\emptyset$. Now, let $a',a_1'\in V(a)\cap eSf$. Then $aa'$ and $aa_1'$ are $\Rc$-related idempotents in $(f]_l$, whence $aa'=aa_1'$. In the same manner we conclude that $a'a=a_1'a$. Thus $a'=a_1'$ and $|V(a)\cap eSf|=1$.
\end{proof}

\begin{lem}\label{inv_yy}
Let $S$ be a locally inverse semigroup, $a\in S$, $e\in E(S)$ and $a'\in V(a)\cap Se$. Then $a'=a_1'(f\wedge e)$ for all $a_1'\in V(a)\cap\Rcc_{a'}$ and all $f\in E(S)$ such that $a\in fS$.
\end{lem}

\begin{proof}
Let $f\in E(S)$ such that $a\in fS$ and let $a_1'\in V(a)\cap\Rcc_{a'}$. Since $aa'\in (f]_r\cap (e]_l=(f\wedge e]$ and $aa_1'\in E(S)\cap \Rcc_{aa'}$, we must have $aa_1'(f\wedge e)=aa'$ because $S$ is locally inverse. Hence $a_1'(f\wedge e)=a'$ since $a_1'a=a'a$.  
\end{proof}

\section{Presentations for locally inverse semigroups}\label{sec3}

For a nonempty set $X$, let
$$\wX=\oX\cup\{(x\wedge y)\,:\;x,y\in\oX\}\,.$$
We consider $(x\wedge y)$ as new letters added to $\oX$, and when they appear isolated we may omit the parentheses. Thus $\wX^+$ and $\wX^*$ are the free semigroup and the free monoid on the alphabet $\wX$, respectively.
For $u\in\wX^+$, let $\hat{\la}_u$ [$\hat{\tau}_u$] be the first [last] letter of $\wX$ occurring in $u$, and let $\la_u$ [$\tau_u$] be the first [last] letter of $\oX$ occurring in $u$. To make things clearer, if $u=(x\wedge y)z$ for $x,y,z\in \oX$, then $\hat{\la}_u=(x\wedge y)$, $\la_u=x$, and $\hat{\tau}_u=\tau_u=z$.

Let $\varepsilon$ be the Auinger's congruence \cite{auinger95} on $\wX^+$ that turns $\wX^+/\varepsilon$ into a model for $\bfli(X)$. We continue to denote by $\varepsilon$ the congruence on $\wX^*$ obtained by adding the pair $(\iota,\iota)$ to $\varepsilon$. Then $\wX^*/\varepsilon$ is just the bifree locally inverse semigroup $\wX^+/\varepsilon$ with an identity adjoined.

Define $x^{-1}=x'$ and $(x\wedge y)^{-1}=(y'\wedge y)(x\wedge x')$ for $x,y\in\oX$. If $u=z_1\cdots z_n\in\wX^+$ with $z_i\in\wX$ for $i=1,\cdots,n$, then set
$$u^{-1}=z_n^{-1}(\la_{z_n}\wedge \tau_{z_{n-1}})z_{n-1}^{-1}\cdots z_2^{-1}(\la_{z_2}\wedge\tau_{z_1})z_1^{-1}\,.$$
Although $(u^{-1})^{-1}\neq u$ and $uu^{-1}u\neq u$, we call $u^{-1}$ the \emph{formal inverse} of $u$ due to the following lemma:

\begin{lem}\cite[Lemma 3.6]{auinger95}\label{au_fi}
$u^{-1}\varepsilon\in V(u\varepsilon)$ and $(u^{-1})^{-1}\varepsilon=u\varepsilon$.
\end{lem} 

A presentation for locally inverse semigroups is a pair $P=\langle X;R\rangle$ where $R$ is a relation on $\wX^+$. Let $\mu$ be the congruence on $\wX^+$ generated by $\varepsilon\cup R$. The semigroup presented by $P$ is $S=\wX^+/\mu$. Clearly $S$ is locally inverse since it is a homomorphic image of the bifree locally inverse semigroup $\wX^+/\varepsilon$. We write $S=\LI\langle X;R\rangle$. Note that $(xx')\mu=(x\wedge x')\mu$ for all $x\in\oX$. Note further that $\mu_m=\mu\cup\{(\iota,\iota)\}$ is a congruence on $\wX^*$ since $R\subseteq \wX^+$. It is convenient to consider also the locally inverse monoid $T=\wX^*/\mu_m$, which is precisely $S$ with a new identity element adjoined. We denote $T$ by $\LI^1\langle X;R\rangle$. Without further comments, throughout this paper, $\mu$ will always represent the congruence generated by $\varepsilon\cup R$ associated with the presentation $P=\langle X;R\rangle$. Also, we denote $\mu_m$ just by $\mu$ since no ambiguity will occur.

There is also another convention we introduce to make the notation used in this paper less cumbersome. Note that $\wX^*$ acts naturally on $T$, both on the left and on the right, as follows: 
$$\wX^*\times T\to T,\; (u,a)\mapsto ua \qquad\mbox{ and }\qquad T\times\wX^* \to T,\; (a,u)\mapsto au$$
where $ua$ and $au$ represent respectively the elements $(u\mu)a$ and $a(u\mu)$ of $T$. Thus, from now on, we will write only $ua$ or $au$ to refer to the elements $(u\mu)a$ and $a(u\mu)$ of $T$.

Next, we introduce three examples of presentations for locally inverse semigroups. The first two examples will be used later to illustrate and explain some of the concepts and results obtained. The third example is the `four-spiral semigroup'. This semigroup is an important example of an idempotent generated bisimple, non-completely simple, semigroup (see \cite{bmp}). Due to its peculiar structure, it is an interesting example to look at and to see how the technique developed in this paper applies to.\vspace*{.3cm}

\noindent {\bf Example 1}: Let $X_1=\{x,y\}$, 
$$A=(\ol{X_1}\times\ol{X_1})\setminus\{(x',x'),(y',y'),(x',y'),(z,z')\,:\; z\in\ol{X_1}\}\,,$$
and $\ol{R}_0=\{(x^2,x^2z),(x^2,zx^2),(x^2,z_1z_2)\,:\;z\in\widehat{X_1},\,(z_1,z_2)\in A\}$. Set
$$\ol{R}_1=\{(x',x'^2),(y',y'^2),(x',x'(y'\wedge x')),(y',(y'\wedge x')y')\}\cup \ol{R}_0\,,$$ 
and consider the presentation $\ol{P}_1=\langle X_1;\ol{R}_1\rangle$. Let $\ol{S}_1=\LI\langle X_1;\ol{R}_1\rangle$ be the locally inverse semigroup presented by $\ol{P}_1$ and let $\ol{\mu}_1$ be the congruence on $\widehat{X_1}^+$ such that $\ol{S}_1=\widehat{X_1}^+/\ol{\mu}_1$. 

Note that $x^2\ol{\mu}_1$ works as a zero element of $\ol{S}_1$ due to the two first pairs of $\ol{R}_0$. Thus we denote it just as $0$. Also by definition of $\ol{R}_0$, $z_1z_2\in 0$ for all $(z_1,z_2)\in A$. Hence, also 
$$\{z_2\wedge z_1\,:\;(z_1,z_2)\in A\}\subseteq 0\,.$$ Analyzing now $\ol{R}_1$, we conclude that $x'\ol{\mu}_1$ and $y'\ol{\mu}_1$ are idempotents, and $(x'\ol{\mu}_1)\Lc ((y'\wedge x')\ol{\mu}_1)\Rc (y'\ol{\mu}_1)$.  Now, with a few more easy computations, it is not hard to see that $\ol{S}_1$ is the combinatorial completely 0-simple semigroup with $\Dc$-class egg-box picture depicted in Figure \ref{egg-boxa}. 
\begin{figure}[ht]
\begin{tabular}{c}
\begin{tabular}{|c|c|c|c|}
\hline $x'x\;^*$ & $x'\;^*$ & $x'y'y$ & $x'y'$ \\ \hline
$x$ & $xx'\;^*$ & $xx'y'y$ & $xx'y'$ \\ \hline
$(y'\wedge x')x$ & $y'\wedge x'\;^*$ & $y'y\;^*$ & $y'\;^*$ \\ \hline
$y(y'\wedge x')x$ & $y(y'\wedge x')$ & $y$ & $yy'\;^*$ \\ \hline
\end{tabular} \\ \\
\begin{tabular}{|c|}
\hline $0\;^*$ \\ \hline
\end{tabular}
\end{tabular}
\caption{Egg-box picture of $\ol{S}_1$.}\label{egg-boxa}
\end{figure}
Note that, in the egg-box picture, we chose a representative for each $\ol{\mu}_1$-class in the non-trivial $\Dc$-class. From now on, when referring to $\ol{S}_1$, we identify each $\ol{\mu}_1$-class with its representative indicated in the egg-box picture. Also, the symbol $^*$ next to an element of $\ol{S}_1$ indicates that element is an idempotent.\hfill\qed\vspace*{.3cm}

\noindent {\bf Example 2}: Let $\ol{S}_2$ be the locally inverse semigroup given by the presentation $\ol{P}_2=\langle X_2;\ol{R}_2\rangle$ where $X_2=\{z\}$ and $\ol{R}_2=\{(z,z^3),(z',{z'}^2)\}$. Let $\ol{\mu}_2$ be the congruence on $\widehat{X_2}^+$ such that $\ol{S}_2=\widehat{X_2}^+/\ol{\mu}_2$. Then 
$$(z'\wedge z')\;\ol{\mu}_2\; z'\;\ol{\mu}_2\;{z'}^2\qquad\mbox{ and }\qquad (z\wedge z)\;\ol{\mu}_2 \;z^2\,.$$ 
Thus $(z_1\wedge z_2)\;\ol{\mu}_2\; z_1z_2$ for each $(z_1\wedge z_2)\in\widehat{X_2}$. It is not hard to see now that the set 
$$Z=\left\{u_1u_2u_3\,:\;u_1,u_3\in\{\iota,z'\}\mbox{ and }u_2\in\{z,z^2\}\right\} $$ 
is a transversal set for the congruence $\ol{\mu}_2$. We identify each $\ol{\mu}_2$-class with its representative in $Z$ except for $(z'zz')\ol{\mu}_2$ where we choose $z'$ instead of $z'zz'$, although sometimes it will be easier if we consider $z'zz'$ instead of $z'$. The semigroup $\ol{S}_2$ is then the completely simple semigroup with $\Dc$-class egg-box picture depicted in Figure \ref{egg-boxb}. 
\begin{figure}[ht]
\begin{tabular}{|c|c|}
\hline $z'z^2$ & $z'z^2z'$ \\ [.2cm]
$z'z\,^*$ & $z'\,^*$ \\ \hline
$z^2\,^*$ & $z^2z'$ \\ [.2cm]
$z$ & $zz'\,^*$  \\ \hline 
\end{tabular}
\caption{Egg-box picture of $\ol{S}_2$.}\label{egg-boxb}
\end{figure}

For each $u\in\widehat{X_2}^+$, let $n_1(u)$ be the number of occurrences of the letter $z$ in $u$. To make it clear, we consider $n_1(z\wedge z)=2$ and $n_1(z\wedge z')=1$. Also, let $n_2(u)$ be the number of maximal nonempty substrings of $u$ with no letter $z'$, and $n(u)=n_1(u)-n_2(u)$. For example, if
$$u=(z\wedge z')z^3(z\wedge z)(z'\wedge z)z'\,,$$
then $n_1(u)=7$, $n_2(u)=3$ and $n(u)=4$. Note now that
$$u\ol{\mu}_2=\{u'\in\widehat{X_2}^+:\; \la_{u'}=\la_u,\;\tau_{u'}=\tau_u\;\mbox{ and }\; n(u')-n(u) \mbox{ even}\}$$ 
for all $u\in \widehat{X_2}^+$.\hfill\qed\vspace*{.3cm}

\noindent {\bf Example 3}: The four-spiral semigroup $Sp_4$ (see \cite{bmp}) is a semigroup generated by four idempotents, say $a$, $b$, $c$ and $d$, such that
$$a\Rc b\Lc c\Rc d\qquad\mbox{ and }\qquad d\wl a\,.$$
This semigroup is bisimple but not completely simple. It can be decomposed into the disjoint union of four copies of the bicyclic semigroup $B$ and a copy of the free monogenic semigroup $N$. The unique $\Dc$-class of $Sp_4$ is depicted in Figure \ref{F4}.
\begin{figure}[ht]
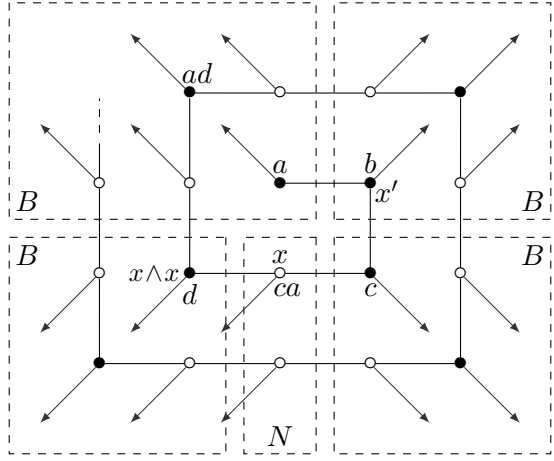

$$\tikz[scale=1.2, shorten <=2pt, shorten >=2pt, >=latex]{
\draw[dashed, shorten >=0pt, shorten <=0pt] (-.6,.4)--(-.6,-2)--(-3,-2)--(-3,.4)--(-.6,.4);
\draw[dashed, shorten >=0pt, shorten <=0pt](.4,.6)--(.4,3)--(-3,3)--(-3,.6)--(.4,.6);
\draw[dashed, shorten >=0pt, shorten <=0pt] (.6,.6)--(.6,3)--(3,3)--(3,.6)--(.6,.6);
\draw[dashed, shorten >=0pt, shorten <=0pt] (.6,.4)--(.6,-2)--(3,-2)--(3,.4)--(.6,.4);
\draw[dashed, shorten >=0pt, shorten <=0pt] (.4,.4)--(.4,-2)--(-.4,-2)--(-.4,.4)--(.4,.4);
\coordinate (1) at (0,1);
\coordinate (2) at (1,1);
\coordinate (3) at (1,0);
\coordinate (4) at (-1,0);
\coordinate (5) at (-1,2);
\coordinate (6) at (2,2);
\coordinate (7) at (2,-1);
\coordinate (8) at (-2,-1);
\draw (2.8,.8) node {$B$};
\draw (2.8,.2) node {$B$};
\draw (-2.8,.8) node {$B$};
\draw (-2.8,.2) node {$B$};
\draw (0,-1.8) node {$N$};
\draw (1) node {$\bullet$};
\draw (2) node {$\bullet$};
\draw (3) node {$\bullet$};
\draw (4) node {$\bullet$};
\draw (5) node {$\bullet$};
\draw (6) node {$\bullet$};
\draw (7) node {$\bullet$};
\draw (8) node {$\bullet$};
\draw (1)--(2);
\draw (2)--(3);
\draw (3)--(0,0);
\draw (0,0)--(4);
\draw (4)--(-1,1);
\draw (-1,1)--(5);
\draw (5)--(0,2);
\draw (0,2)--(1,2);
\draw (1,2)--(6);
\draw (6)--(2,1);
\draw (2,1)--(2,0);
\draw (2,0)--(2,-1);
\draw (2,-1)--(7);
\draw (7)--(1,-1);
\draw (1,-1)--(0,-1);
\draw (0,-1)--(-1,-1);
\draw (-1,-1)--(8);
\draw (-2,0)--(8);
\draw (-2,0)--(-2,1);
\draw (-2,1)--(-2,1.5);
\draw[dashed] (-2,1.5)--(-2,2);
\draw (-2,0) node {$\circ$};
\draw (-2,1) node {$\circ$};
\draw (-1,-1) node {$\circ$};
\draw (-1,1) node {$\circ$};
\draw (0,-1) node {$\circ$};
\draw (0,2) node {$\circ$};
\draw (1,-1) node {$\circ$};
\draw (1,2) node {$\circ$};
\draw (2,0) node {$\circ$};
\draw (2,1) node {$\circ$};
\draw[->, black!80] (0,1)--(-.7,1.7);
\draw[->, black!80] (-1,1)--(-1.7,1.7);
\draw[->, black!80] (0,2)--(-.7,2.7);
\draw[->, black!80] (-1,2)--(-1.7,2.7);
\draw[->, black!80] (-2,1)--(-2.7,1.7);
\draw[->, black!80] (1,1)--(1.7,1.7);
\draw[->, black!80] (2,1)--(2.7,1.7);
\draw[->, black!80] (1,2)--(1.7,2.7);
\draw[->, black!80] (2,2)--(2.7,2.7);
\draw[->, black!80] (1,0)--(1.7,-.7);
\draw[->, black!80] (2,0)--(2.7,-.7);
\draw[->, black!80] (1,-1)--(1.7,-1.7);
\draw[->, black!80] (2,-1)--(2.7,-1.7);
\draw[->, black!80] (0,0)--(-.7,-.7);
\draw[->, black!80] (-1,0)--(-1.7,-.7);
\draw[->, black!80] (0,-1)--(-.7,-1.7);
\draw[->, black!80] (-1,-1)--(-1.7,-1.7);
\draw[->, black!80] (-2,0)--(-2.7,-.7);
\draw[->, black!80] (-2,-1)--(-2.7,-1.7);
\draw (4) node [below] {$d$};
\draw (1) node [above] {$a$};
\draw (3) node [below] {$c$};
\draw (5) node [above] {\hspace*{.2cm}$ad$};
\draw (2) node [above] {$b$};
\draw (0,0) node {$\circ$};
\draw (0,0) node [below] {\hspace*{.2cm}$ca$};
\draw (0,0) node [above] {$x$};
\draw (-1,0) node [left] {\small$x\!\wedge\! x$};
\draw (1,1) node [below right=-5pt] {\hspace*{.1cm}$x'$};
}$$
\caption{The four-spiral semigroup $Sp_4$.}\label{F4}
\end{figure}
The elements in the same row are $\Rc$-equivalent and the elements in the same column are $\Lc$-equivalent. The arrows indicate the covers for the natural partial order and the idempotents are identified by \tikz{\draw (0,0) node {$\bullet$};}. The name of this semigroup comes from the `four-spiral' made by the placement of its idempotents. If we call $x$ the element $ca$, then $b$ is an inverse of $x$ and $d=x\wedge x$. Hence $Sp_4=LI\langle\{x\};\{(x',x'^2),(x,(x\wedge x)x)\}\rangle$.\hfill\qed\vspace*{.3cm}

Let $S=\LI\langle X;R\rangle$, $a\in S$ and $v\in \wX^+$. If $av\Rc a$, then $\varphi_v$ denotes the right translation of $\Lcc_a$ associated with $v\mu$, that is, $\varphi_v$ is the bijective mapping $\varphi_v:\Lcc_a\to\Lcc_{av}\,,\;s\mapsto sv$ which preserves the $\Rc$-related elements. If $va\Lc a$, then $\psi_v$ denotes the left translation of $\Rcc_a$ associated with $v\mu$, that is, $\psi_v$ is the bijective mapping $\psi_v:\Rcc_a\to\Rcc_{va}\,,\;s\mapsto vs$ which preserves the $\Lc$-related elements. We collect some information about the left and right translations on $S$ in the following results.

\begin{lem}\label{transl_gen}
Let $S=\LI\langle X;R\rangle$, $x\in\oX$, $s\in S(x\wedge x')$ and $v\in\wX^+$ such that $\la_v=x$. If $sv\Rc s$, then $\varphi_v:\Lcc_s\to \Lcc_{sv}$ is a right translation of $\Lcc_s$ with inverse right translation $\varphi_{v^{-1}}:\Lcc_{sv}\to\Lcc_s$.
\end{lem}

\begin{proof}
$\varphi_v$ is obviously a right translation of $\Lcc_s$ into $\Lcc_{sv}$. Further, without loss of generality, we can assume that $s$ is an idempotent; and so $s,\,(vv^{-1})\mu\in ((x\wedge x')\mu]_l$. Since $((x\wedge x')\mu)]_l$ is a left normal band, we conclude that $s$ and $s(vv^{-1})$ are $\Rc$-related idempotents of $((x\wedge x')\mu]_l$. Hence $s=s(vv^{-1})=(sv)v^{-1}$ and $\varphi_{v^{-1}}$ is the inverse right translation of $\varphi_v\,$.
\end{proof}

\begin{cor}\label{trans}
Let $S=\LI\langle X;R\rangle$, $x,y\in\oX$ and $s\in S(x\wedge x')$. Then:
\begin{itemize}
\item[$(i)$] $\varphi_x:\Lcc_s\to \Lcc_{sx}$ is a right translation of $\Lcc_s$ with inverse right translation $\varphi_{x'}:\Lcc_{sx}\to \Lcc_s$.
\item[$(ii)$] If $s(y\wedge y')\Rc s$, then $\varphi_{y\wedge y'}:\Lcc_s\to \Lcc_{s(y\wedge y')}$ is a right translation with inverse right translation $\varphi_{y\wedge x'}:\Lcc_{s(y\wedge y')}\to \Lcc_s$. 
\item[$(iii)$] If $s(x\wedge y')\Rc s$, then $\varphi_{x\wedge y'}:\Lcc_s\to \Lcc_{s(x\wedge y')}$ is a right translation with inverse right translation $\varphi_{x\wedge x'}:\Lcc_{s(x\wedge y')}\to \Lcc_s$.
\end{itemize}
\end{cor}
 
\begin{proof}
$(i)$ is obvious since $s\in S(x\wedge x')$.

$(ii)$. Clearly $\varphi_{y\wedge y'}:\Lcc_s\to \Lcc_{s(y\wedge y')}$ is a right translation of $\Lcc_s$. Let $u=(x\wedge x')(y\wedge y')=(y\wedge x')^{-1}$ and note that $s(y\wedge y')=su$ because $s\in S(x\wedge x')$. Hence $\varphi_{y\wedge y'}=\varphi_u$ and, by Lemma \ref{transl_gen}, $\varphi_{y\wedge x'}=\varphi_{u^{-1}}:\Lcc_{s(y\wedge y')}\to \Lcc_s$ is the inverse right translation of $\varphi_{y\wedge y'}$.

The proof of $(iii)$ is analogous. We just need to observe that the second mapping $\varphi_{x\wedge x'}$ is precisely $\varphi_{(x\wedge y')^{-1}}$.
\end{proof}

The two previous results have their right-left duals. We indicate next only the dual of Corollary \ref{trans}.

\begin{cor}\label{trans_dual}
Let $S=\LI\langle X;R\rangle$, $x,y\in\oX$ and $s\in (x\wedge x')S$. Then:
\begin{itemize}
\item[$(i)$] $\psi_{x'}:\Rcc_s\to \Rcc_{x's}$ is a left translation of $\Rcc_s$ with inverse left translation $\psi_{x}:\Rcc_{x's}\to \Rcc_s$.
\item[$(ii)$] If $(y\wedge y')s\Lc s$, then $\psi_{y\wedge y'}:\Rcc_s\to \Rcc_{(y\wedge y')s}$ is a left translation with inverse left translation $\psi_{x\wedge y'}:\Rcc_{(y\wedge y')s}\to \Rcc_s$. 
\item[$(iii)$] If $(y\wedge x')s\Lc s$, then $\psi_{y\wedge x'}:\Rcc_s\to \Rcc_{(y\wedge x')s}$ is a left translation with inverse left translation $\psi_{x\wedge x'}:\Rcc_{(y\wedge x')s}\to \Rcc_s$.
\end{itemize}
\end{cor}

\section{Birooted locally inverse word graph}

In this section, we introduce the birooted locally inverse word graphs. These combinatorial objects will be fundamental to the characterization of the elements of a locally inverse semigroup. However, we will focus only on general properties of these graphs. We will introduce the crucial concept of language recognized by them, and study how this concept of language behaves under homomorphic images and subgraphs of those graphs.

Birooted locally inverse word graphs are special bipartite graphs with two kinds of edges and with two distinguished vertices. Some edges will have both an orientation and a label from $\oX$, while others will have neither. We will call \emph{arrows} the oriented and labeled edges, and \emph{lines} the non-oriented and non-labeled edges. The term `edges' will be used to refer to both arrows and lines. In terms of notation, we use $\vec{\ev}$ (with a vector on top) to denote an arrow and $\oev$ (with a line on top) to denote a line. However, since our graphs will have both arrows and lines simultaneously, we will often use the notation $\ev$ to refer to both of them in general. In fact, we will use the notations $\ol{\ev}$ or $\vec{\ev}$, instead of $\ev$, only when we want to emphasize the idea that $\ev$ is a line or an arrow, respectively.

We use the term graph as synonymous of a `simple graph'. In the context of this paper, it means that our graphs will not have multiple lines connecting the same vertices and will not have multiple arrows with the same initial or starting vertex, the same label and the same final or ending vertex. Note however that we admit lines and arrows with the same endpoints, and multiple arrows from the same starting vertex to the same ending vertex but with distinct labels.

Let $\Ga=(\V,\E)$ be a graph with vertices $\V$ and edges, that is, lines and arrows, $\E$. We denote by $\lE$ and $\oE$ the sets of lines and arrows of $\Ga$, respectively. We use also the notations $V(\Ga)$ for the set of vertices, $E(\Ga)$ for the set of edges, $\ol{E}(\Ga)$ for the set of lines and $\vec{E}(\Ga)$ for the set of arrows, all with respect to the graph $\Ga$. We will consider only bipartite graphs in this paper. Thus $\V$ is partitioned into two sets $\V_l$ and $\V_r$ such that all edges have one endpoint in $\V_l$ and the other endpoint in $\V_r$. We will use also the notations $V_l(\Ga)=\V_l$ and $V_r(\Ga)=\V_r$. The vertices of $\V_l$ and $\V_r$ will be designated as left and right vertices, respectively. Thus, the \emph{side} $\sb(\av)$ of a vertex $\av$ is either left or right depending on whether $\av\in \V_l$ or $\av\in\V_r$, respectively. 

An arrow $\vec{\ev}$ starting at $\av\in\V$, labeled by $x\in\oX$ and ending at $\bv\in\V$ is represented by the triple $\vec{\ev}=(\av,x,\bv)$ with the label at the center. There is no ambiguity with this notation since our graphs do not have arrows with the same initial vertex, the same final vertex, and the same label. The label of $\vec{\ev}$ is denoted by $\l(\vec{\ev})$, while the starting vertex and the ending vertex are denoted by $\lv(\vev)$ and $\rv(\vev)$, respectively. Since our graphs are bipartite with no multiple lines with the same endpoints, there is no ambiguity in using $(\av,\bv)$, with $\av\in\V_l$ and $\bv\in\V_r$, for denoting the line $\oev$ with endpoints $\av$ and $\bv$. Thus $\lE\subseteq\V_l\times\V_r$. We write $\lv(\oev)$ and $\rv(\oev)$ to denote $\av$ and $\bv$, respectively, for $\oev=(\av,\bv)$. 

A \emph{walk} in $\Ga$ is a sequence
$$p:=\; \av_0\ev_1\av_1\ev_2\cdots\av_{n-1}\ev_n\av_n$$
of alternating vertices $\av_0,\av_1,\cdots,\av_n\in\V$ and edges $\ev_1,\ev_2, \cdots,\ev_n\in\E$ such that, for all $i\in\{1,\cdots,n\}$, $(i)$ $\av_{i-1}$ and $\av_i$ are the two endpoints of $\ev_i$ and $(ii)$ if $\ev_i$ is an arrow then $\ev_i$ starts at $\av_{i-1}$ and ends at $\av_i$. The condition $(ii)$ assures that the orientation of the arrows is respected when `walking' $p$ (from $\av_0$ to $\av_n$). The vertex $\av_0$ is the \emph{initial vertex} of $p$ while the vertex $\av_n$ is the \emph{final vertex} of $p$. The walk $p$ is designated as an \emph{$\av_0-\av_n$ walk}, and the set of all $\av_0-\av_n$ walks in $\Ga$ is denoted by $P_\Ga(\av_0,\av_n)$. If no ambiguity occurs we just write $P(\av_0,\av_n)$ for $P_\Ga(\av_0,\av_n)$. The \emph{length} of $p$ is $n$, the number of edges. Often, we will indicate only the sequence $\ev_1\ev_2\cdots\ev_n$ of edges to identify the walk $p$. If no ambiguity occurs, we may also indicate only the sequence of vertices $\av_0\av_1\cdots\av_n$ to refer to the walk $p$.

Any factor $p':=\;\av_i\ev_{i+1}\cdots\ev_j\av_j$, with $j\geq i$, of $p$ is called a \emph{subwalk} of $p$, and a \emph{decomposition} of $p$ is a sequence $p_1,\cdots,p_k$ of subwalks of $p$ such that $p=p_1\cdots p_k$ (of course, omitting the initial vertex of each $p_i$ with $i>1$ since this vertex is also the final vertex of $p_{i-1}$). The paths (walks without repetition of edges) of length 1 are called \emph{elementary paths}. Thus any subpath of $p$ of length 1 is an elementary subpath. The elementary decomposition of $p$ is the unique decomposition of $p$ into its elementary subpaths. Note that $p$ has length $n$ \iff\ its elementary decomposition has $n$ subpaths. 

The graph $\Ga$ is \emph{strongly connected} if there is an $\av-\bv$ walk for all $\av,\bv\in\V$. If $\Ga_1$ is the graph obtained from $\Ga$ by replacing each arrow by a line, then $\Ga$ is \emph{connected} if $\Ga_1$ is connected. In other words, the difference between strongly connectedness and connectedness is that we do not take into account the orientation of the arrows in the case of connectedness.

An \emph{oriented bipartite graph} is a bipartite graph where all arrows start at left vertices and end at right vertices, that is, $\oE\subseteq\V_l\times\oX\times\V_r$. The \emph{content} of a vertex $\av\in\V$ is the set $\cb(\av)$ of all letters labeling arrows of $\Ga$ with $\av$ as one of its endpoints. A \emph{locally inverse word graph} (`liw-graph' for short) is a graph $\Ga=(\V,\E)$ such that
\begin{itemize}
\item[$(i)$] $\Ga$ is a connected oriented bipartite graph;
\item[$(ii)$] $\cb(\av)\neq\emptyset$ for all $\av\in\V$;
\item[$(iii)$] if $(\av,x,\bv)\in\oE$, then there are lines $(\av,\bv_1)$ and $(\av_1,\bv)$ in $\lE$ such that $(\av_1,x',\bv_1)\in\oE$.
\end{itemize}
Note that an liw-graph is strongly connected by $(iii)$ and since it is connected. Also, each vertex of $\Ga$ is the endpoint of some arrow by $(ii)$, and together with $(iii)$, it is also the endpoint of some line. 

When illustrating an liw-graph, note that we do not need to identify which vertices are left vertices and which vertices are right vertices. This is automatically done by the arrows. Whenever the drawing of an liw-graph does not become too complex, we will try to put its vertices in two columns, left column for the left vertices and right column for the right vertices. Also, to help distinguish the lines from the arrows, we draw the lines as dashed lines. The liw-graphs will have many edges and so, the labels of some arrows will have to be placed near edges. Do not forget that only the arrows have labels.  In Figure \ref{ex_liw} we depict an example of an liw-graph to illustrate our concept.  

\begin{figure}[ht]
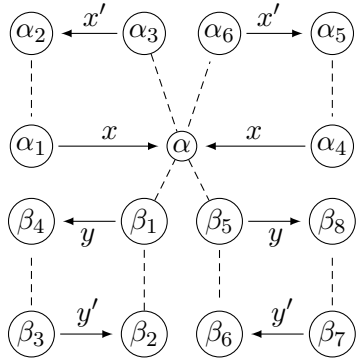

$$\tikz[shorten <=2pt, shorten >=3pt, >=latex]{
\node [circle, inner sep=1pt, draw] (1) at (3,3) {$\al$};
\node [circle, inner sep=1pt, draw] (2) at (1,4.5) {$\al_2$};
\node [circle, inner sep=1pt, draw] (3) at (1,3) {$\al_1$};
\node [circle, inner sep=1pt, draw] (4) at (2.5,4.5) {$\al_3$};
\node [circle, inner sep=1pt, draw] (5) at (5,4.5) {$\al_5$};
\node [circle, inner sep=1pt, draw] (6) at (5,3) {$\al_4$};
\node [circle, inner sep=1pt, draw] (7) at (3.5,4.5) {$\al_6$};
\node [circle, inner sep=1pt, draw] (8) at (2.5,2) {$\be_1$};
\node [circle, inner sep=1pt, draw] (9) at (2.5,0.5) {$\be_2$};
\node [circle, inner sep=1pt, draw] (10) at (1,0.5) {$\be_3$};
\node [circle, inner sep=1pt, draw] (11) at (1,2) {$\be_4$};
\node [circle, inner sep=1pt, draw] (12) at (3.5,2) {$\be_5$};
\node [circle, inner sep=1pt, draw] (13) at (3.5,.5) {$\be_6$};
\node [circle, inner sep=1pt, draw] (14) at (5,.5) {$\be_7$};
\node [circle, inner sep=1pt, draw] (15) at (5,2) {$\be_8$};
\draw[densely dashed] (4)--(1)--(7);
\draw[densely dashed] (2)--(3);
\draw[densely dashed] (5)--(6);
\draw[densely dashed] (9)--(8)--(1)--(12)--(13);
\draw[densely dashed] (10)--(11);
\draw[densely dashed] (14)--(15);
\draw[->] (3) to node [above=-1pt] {$x$} (1);
\draw[->] (4) to node [above=-1pt, pos=.4] {$x'$} (2);
\draw[->] (6) to node [above=-1pt] {$x$} (1);
\draw[->] (7) to node [above=-1pt, pos=.4] {$x'$} (5);
\draw[->] (8) to node [below=-1pt] {$y$} (11);
\draw[->] (10) to node [above=-1pt] {$y'$} (9);
\draw[->] (12) to node [below=-1pt] {$y$} (15);
\draw[->] (14) to node [above=-1pt, pos=.4] {$y'$} (13);}$$
\caption{An example of an liw-graph.}\label{ex_liw}
\end{figure}

An elementary path $q:=\,\av_0\ev_1\av_1$ is called a \emph{left} or \emph{right} elementary path depending on whether $\av_0\in\V_l$ or $\av_0\in\V_r$, respectively. Note that $\ev_1$ is a line if $q$ is a right elementary path. For the left elementary paths, we have two types: the \emph{arrow elementary paths} where $\ev_1$ is an arrow, and the \emph{line elementary paths} where $\ev_1$ is a line. To avoid ambiguity, we use the terminology `line elementary path' only for the left elementary paths. Set $\w(q)=\{\iota\}$ if $q$ is a right elementary path; $\w(q)=\{\l(\ev_1)\}\subseteq\oX$ if $q$ is an arrow elementary path; and $\w(q)=\{(x\wedge y)\in\wX\,:\; x\in\cb(\av_0)\mbox{ and } y\in \cb(\av_1)\}$ if $q$ is a line elementary path.

Let $p$ be a walk in $\Ga$ and let $p_1p_2\cdots p_k$ be its elementary decomposition. Note that the elementary subpaths $p_i$ alternate between left and right elementary subpaths. Set
$$\w(p)=\{z_1\cdots z_k\in\wX^*\,:\; z_i\in\w(p_i)\mbox{ for } 1\leq i\leq k\}\subseteq\wX^*\,.$$
Thus the $z_i$ alternate between letters from $\wX$ and the empty word. Obviously the empty words are to be omitted from the sequences $z_1\cdots z_k$. If $p$ is a trivial path with no edges, set $\w(p)=\{\iota\}$. We say that $u\in\wX^*$ is a \emph{label} for the walk $p$ if $u\in\w(p)$. By definition of $\w$ it is evident that if $q_1\cdots q_j$ is another decomposition of $p$, then 
$$\w(p)=\{u_1\cdots u_j\,:\; u_i\in\w(q_i)\mbox{ for } 1\leq i\leq j\}.$$ 

For $\av,\bv\in V(\Ga)$, set 
\begin{itemize}
\item[$(i)$] $L_{\av,\bv}(\Ga)=\cup\{\w(p)\,:\; p\in P(\av,\bv)\}$,
\item[$(ii)$] $\rvec{L_\av}(\Ga)=\cup\{\w(p)\,:\; p \mbox{ is a walk in } \Ga \mbox{ starting at } \av\}$, and
\item[$(iii)$] $\lvec{L_\bv}(\Ga)=\cup\{\w(p)\,:\; p \mbox{ is a walk in } \Ga \mbox{ ending at } \bv\}$.
\end{itemize}
The following result relates these three sets and it is an obvious consequence of the strong connectedness of $\Ga$.

\begin{lem}\label{3_lang}
For $\av,\bv\in V(\Ga)$,
\begin{itemize}
\item[$(i)$] $\rvec{L_\av}(\Ga)$ is the set of all prefixes of the words from $L_{\av,\bv}(\Ga)$;
\item[$(ii)$] $\lvec{L_\bv}(\Ga)$ is the set of all suffixes of the words from $L_{\av,\bv}(\Ga)$.
\end{itemize}
\end{lem}

A \emph{homomorphism} $\varphi:\Ga\to\Ga'$ between two liw-graphs $\Ga$ and $\Ga'$ is a pair $(\varphi_V,\varphi_E)$ of functions, $\varphi_V: V(\Ga)\to V(\Ga')$ and $\varphi_E:E(\Ga)\to E(\Ga')$, such that $(\av,\bv)\varphi_E=(\av \varphi_V,\bv\varphi_V)$ and $(\av,x,\bv)\varphi_E=(\av \varphi_V,x,\bv\varphi_V)$. In other words, a homomorphism between liw-graphs is a mapping $\varphi_V$ on the vertices that preserves line incidence and arrow incidence, orientation and labeling. Thus $\varphi_V$ must send left vertices into left vertices and right vertices into right vertices because each vertex of $\Ga$ is the endpoint of some arrow. We will use $\varphi$ for denoting both mappings $\varphi_V$ and $\varphi_E$ when no ambiguity occurs.

The homomorphism $\varphi$ is called a \emph{monomorphism} if $\varphi_V$ is one-to-one, an \emph{epimorphism} if $\varphi_V$ is onto, an \emph{$E$-surjective epimorphism} if both $\varphi_V$ and $\varphi_E$ are onto, and an \emph{isomorphism} if it is a monomorphism and an $E$-surjective epimorphism. Note that $\varphi_E$ is one-to-one if $\varphi$ is a monomorphism. Thus both $\varphi_V$ and $\varphi_E$ are bijections if $\varphi$ is an isomorphism. Further, observe that $\varphi^{-1}=(\varphi_V^{-1},\varphi_E^{-1})$ is an isomorphism (the inverse isomorphism) from $\Ga'$ to $\Ga$ if $\varphi$ is an isomorphism. Two liw-graphs are \emph{isomorphic} if there exists an isomorphism from one of them to the other.

Let $\varphi:\Ga\to\Ga'$ be an liw-graph homomorphism. Since $\varphi$ preserves line incidence and arrow incidence and orientation, if $p$ is a walk in $\Ga$, then $p\varphi$ is a walk in $\Ga'$. Since $\varphi$ preserves also the labels of the arrows, the following lemma is obvious:

\begin{lem}\label{hom_wp}
Let $\varphi:\Ga\to\Ga'$ be an liw-graph homomorphism, $p$ be a walk in $\Ga$ and $\av\in V(\Ga)$. Then $\cb(\av)\subseteq\cb(\av\varphi)$ and $\w(p)\subseteq\w(p\varphi)$.
\end{lem}

Let $\eta$ be an equivalence relation on the vertices of an liw-graph $\Ga$ that separates left from right vertices, that is, $\sb(\av)=\sb(\bv)$ for any $(\av,\bv)\in\eta$. The quotient is the graph $\Ga/\eta$ with vertices $V(\Ga/\eta)=V(\Ga)/\eta$, lines 
$$\ol{E}(\Ga/\eta)=\{(\av\eta,\bv\eta)\,:\;(\av,\bv)\in \ol{E}(\Ga)\}$$ 
and arrows 
$$\vec{E}(\Ga/\eta)=\{(\av\eta,x,\bv\eta)\,:\;(\av,x,\bv)\in \vec{E}(\Ga)\}.$$ The graph $\Ga/\eta$ is clearly an liw-graph. If $\varphi_V:V(\Ga)\to V(\Ga/\eta)$ is the mapping defined by $\av\varphi_V=\av\eta$ and $\varphi_E:E(\Ga)\to E(\Ga/\eta)$ is the mapping defined by $(\av,\bv)\varphi_E=(\av\eta,\bv\eta)$ and $(\av,x,\bv)\varphi_E=(\av\eta,x,\bv\eta)$, then $\varphi=(\varphi_V,\varphi_E)$ is an $E$-surjective epimorphism. We call $\varphi$ the \emph{natural homomorphism} from $\Ga$ onto $\Ga/\eta$. 

Conversely, if $\varphi=(\varphi_V,\varphi_E)$ is an $E$-surjective epimorphism from $\Ga$ to $\Ga'$ and $\eta$ is the equivalence relation $\ker\varphi_V$ on $V(\Ga)$, then $\varphi'_V:V(\Ga/\eta)\to V(\Ga')$ defined by $(\av\eta)\varphi'_V=\av\varphi_V$ and $\varphi'_E:E(\Ga/\eta)\to E(\Ga')$ defined by $(\av\eta,\bv\eta)\varphi'_E=(\av,\bv)\varphi_E$ and by $(\av\eta,x,\bv\eta)\varphi'_E=(\av,x,\bv)\varphi_E$ are well-defined bijections and $\varphi'=(\varphi'_V,\varphi'_E)$ is an isomorphism from $\Ga/\eta$ onto $\Ga'$. 

If $p$ is a walk in the liw-graph $\Ga$, let $p\eta$ be the corresponding path in the quotient liw-graph $\Ga/\eta$. In face of the previous observations, the following lemma is a particular instance of the last lemma:

\begin{lem}\label{quo_wp}
Let $\Ga/\eta$ be the quotient of an liw-graph $\Ga$ under the equivalence relation $\eta$ on $V(\Ga)$ that separates left from right vertices, and let $p$ be a walk in $\Ga$. Then $\w(p)\subseteq\w(p\eta)$.
\end{lem}

An \emph{liw-subgraph} of an liw-graph $\Ga$ is another liw-graph $\Ga_1$ such that $V(\Ga_1)\subseteq V(\Ga)$ and $E(\Ga_1)\subseteq E(\Ga)$. Thus if $\av\in V(\Ga_1)$ then $\cb_{\Ga_1}(\av)\subseteq\cb_\Ga(\av)$. The next lemma is also obvious:

\begin{lem}\label{sub_wp}
Let $\Ga_1$ be an liw-subgraph of the liw-graph $\Ga$ and let $p$ be a walk in $\Ga_1$. Then $\w_{\Ga_1}(p)\subseteq\w_\Ga(p)$.
\end{lem}

A \emph{birooted locally inverse word graph} (`bliw-graph' for short) is a triple $\A=(\av,\Ga,\bv)$ where $\Ga$ is an liw-graph, $\av\in V_l(\Ga)$ and $\bv\in V_r(\Ga)$. The vertices $\av$ and $\bv$ are called the \emph{roots} of $\A$, $\av$ is its \emph{left root} while $\bv$ is its \emph{right root}. We denote by $\lv(\A)$ and $\rv(\A)$ the left and right roots of $\A$ respectively.
When depicting a bliw-graph, we identify the left and right roots by a double circle or the symbol $ \tikz[baseline=-4pt,scale=0.3]{\coordinate  (1) at (0,0);\draw (1) node{$\bullet$}; \draw (1)circle (10pt);}$. Note that there is no ambiguity in using the same representation for both left and right roots since the arrows will clarify which one is the left root and which one is the right root. If $\Ga$ is the liw-graph of Figure \ref{ex_liw}, then $(\al_3,\Ga,\be_2)$ is illustrated in Figure \ref{ex_bliw}. 

\begin{figure}
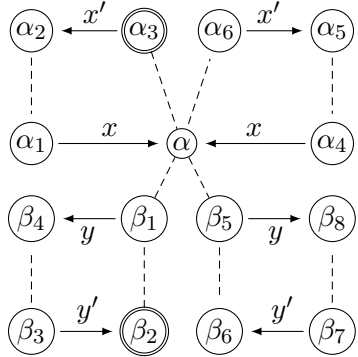

$$\tikz[shorten <=2pt, shorten >=3pt, >=latex]{
\node [circle, inner sep=1pt, draw] (1) at (3,3) {$\al$};
\node [circle, inner sep=1pt, draw] (2) at (1,4.5) {$\al_2$};
\node [circle, inner sep=1pt, draw] (3) at (1,3) {$\al_1$};
\node [circle, double, inner sep=1pt, draw] (4) at (2.5,4.5) {$\al_3$};
\node [circle, inner sep=1pt, draw] (5) at (5,4.5) {$\al_5$};
\node [circle, inner sep=1pt, draw] (6) at (5,3) {$\al_4$};
\node [circle, inner sep=1pt, draw] (7) at (3.5,4.5) {$\al_6$};
\node [circle, inner sep=1pt, draw] (8) at (2.5,2) {$\be_1$};
\node [circle, double, inner sep=1pt, draw] (9) at (2.5,0.5) {$\be_2$};
\node [circle, inner sep=1pt, draw] (10) at (1,0.5) {$\be_3$};
\node [circle, inner sep=1pt, draw] (11) at (1,2) {$\be_4$};
\node [circle, inner sep=1pt, draw] (12) at (3.5,2) {$\be_5$};
\node [circle, inner sep=1pt, draw] (13) at (3.5,.5) {$\be_6$};
\node [circle, inner sep=1pt, draw] (14) at (5,.5) {$\be_7$};
\node [circle, inner sep=1pt, draw] (15) at (5,2) {$\be_8$};
\draw[densely dashed] (4)--(1)--(7);
\draw[densely dashed] (2)--(3);
\draw[densely dashed] (5)--(6);
\draw[densely dashed] (9)--(8)--(1)--(12)--(13);
\draw[densely dashed] (10)--(11);
\draw[densely dashed] (14)--(15);
\draw[->] (3) to node [above=-1pt] {$x$} (1);
\draw[->] (4) to node [above=-1pt, pos=.4] {$x'$} (2);
\draw[->] (6) to node [above=-1pt] {$x$} (1);
\draw[->] (7) to node [above=-1pt, pos=.4] {$x'$} (5);
\draw[->] (8) to node [below=-1pt] {$y$} (11);
\draw[->] (10) to node [above=-1pt] {$y'$} (9);
\draw[->] (12) to node [below=-1pt] {$y$} (15);
\draw[->] (14) to node [above=-1pt, pos=.4] {$y'$} (13);}$$
\caption{An example of a bliw-graph.}\label{ex_bliw}
\end{figure}

The language recognized by the bliw-graph $\A=(\av,\Ga,\bv)$ is the subset
$$L(\A)=\cup\{\w(p)\,:\; p\in P(\av,\bv)\}=L_{\av,\bv}(\Ga)$$
of $\wX^+$. Note that $\iota\not\in\w(p)$ for any $p\in P(\av,\bv)$ because $\av$ is a left vertex, $\bv$ is a right vertex, and any $\av-\bv$ walk $p$ contains left elementary subpaths.

\begin{lem}\label{ext_lang}
If $\A=(\av,\Ga,\bv)$ and $\A'=(\av',\Ga',\bv')$ are two bliw-graphs such that $L(\A)\subseteq L(\A')$, then $\rvec{L_\av}(\Ga)\subseteq \rvec{L_{\av'}}(\Ga')$ and $\lvec{L_\bv}(\Ga)\subseteq \lvec{L_{\bv'}}(\Ga')$. 
\end{lem}

\begin{proof}
This lemma is an obvious consequence of Lemma \ref{3_lang}.
\end{proof}

A homomorphism $\varphi:\A\to \A'$ from a bliw-graph $\A=(\av,\Ga,\bv)$ to a bliw-graph $\A'=(\av',\Ga',\bv')$ is an liw-graph homomorphism $\varphi:\Ga\to\Ga'$ that preserves the roots, that is, $\av\varphi=\av'$ and $\bv\varphi=\bv'$. If $\varphi$ preserves only the left root [only the right root, no roots], then we say that $\varphi$ is a \emph{left} [\emph{right}, \emph{weak}] \emph{homomorphism} between two bliw-graphs. The notions of [left, right, weak] monomorphim, epimorphism, $E$-surjective epimorphism and isomorphism between bliw-graphs are now the expected ones. Thus, two bliw-graphs $\A$ and $\A'$ are [left, right, weakly] isomorphic if there exists a [left, right, weak] isomorphism $\varphi:\A\to\A'$. The quotient of a bliw-graph $\A=(\av,\Ga,\bv)$ under an equivalence relation $\eta$ on $V(\Ga)$ that separates left from right vertices is the bliw-graph $\A/\eta=(\av\eta,\Ga/\eta,\bv\eta)$. A bliw-subgraph of $\A=(\av,\Ga,\bv)$ is another bliw-graph $\A_1=(\av,\Ga_1,\bv)$ where $\Ga_1$ is an liw-subgraph of $\Ga$. The next result is a trivial consequence of Lemmas \ref{hom_wp}, \ref{quo_wp} and \ref{sub_wp}.

\begin{prop}\label{lang_bliw}
Let $\varphi:\A\to \A'$ be a homomorphism between two bliw-graphs $\A=(\av,\Ga,\bv)$ and $\A'=(\av',\Ga',\bv')$. Let $\eta$ be an equivalence relation on $V(\Ga)$ that separates left from right vertices, and let $\A_1=(\av,\Ga_1,\bv)$ be a bliw-subgraph of $\A$. Then 
$$L(\A)\subseteq L(\A')\;\mbox{ and }\; L(\A_1)\subseteq L(\A)\subseteq L(\A/\eta)\,.$$
\end{prop} 

Let $\A$ and $\A'$ be two bliw-graphs. If $\varphi:\A\to\A'$ is a homomorphism, then $L(\A)\subseteq L(\A')$ as seen in the previous proposition. In general, we cannot guarantee the converse, that is, the existence of a homomorphism $\varphi:\A\to\A'$ if $L(\A)\subseteq L(\A')$. In fact, the same is true for Stephen's theory for inverse word graphs. However, in Stephen's theory, if $\Delta$ and $\Delta'$ are `reduced' inverse word graphs, then the homomorphism $\varphi:\Delta\to\Delta'$ exists if $L(\Delta)\subseteq L(\Delta')$ (see \cite[Theorem 2.5]{stephen90}). It is natural to ask now if the same occurs for `reduced' bliw-graphs. In the next section we will see that this question has a positive answer.

\section{Reduction of bliw-graphs}\label{sec5}

In this section, we introduce the reduced bliw-graphs. We prove that the inclusion $L(\A)\subseteq L(\A')$ between reduced bliw-graphs $\A$ and $\A'$, together with another condition, is sufficient to guarantee the existence of a homomorphism $\varphi:\A\to\A'$ (Proposition \ref{conv}). Fortunately, in Section \ref{sec6}, Lemma \ref{pre_adj}, we will see that this other condition is itself a consequence of the inclusion $L(\A)\subseteq L(\A')$ for the cases we are interested in: reduced bliw-graphs associated with a given locally inverse semigroup. Here, we introduce also two reduction operations to transform bliw-graphs. We then prove that each bliw-graph can be reduced into a unique reduced bliw-graph using these two operations (Proposition \ref{rbliw}).

Let $\Ga$ be an liw-graph. A \emph{basic path} in $\Ga$ is a path $q:=\,\av_0\ev_1\av_1\ev_2\av_2$ of length 2 with one arrow. Then $q$ is either a \emph{left basic path} or a \emph{right basic path} depending on whether $\av_0$ is a left or right vertex, respectively. Thus, $\ev_1\in\ol{E}(\Ga)$ and $\ev_2\in\vec{E}(\Ga)$ if $q$ is a right basic path, and $\ev_1\in\vec{E}(\Ga)$ and $\ev_2\in\ol{E}(\Ga)$ if $q$ is a left basic path. Two right [left] basic paths are equivalent if they have the same starting [ending] vertex and their arrows have the same label. Then $\Ga$ is called \emph{deterministic} [\emph{injective}] if no two right [left] basic paths are equivalent. A particular consequence of being deterministic [injective] is that no two arrows starting [ending] at $\av$ have the same label. A \emph{reduced locally inverse word graph} (rliw-graph for short) is a deterministic and injective liw-graph. 

Obviously, a bliw-graph $\A=(\av,\Ga,\bv)$ is called reduced [deterministic, injective] if $\Ga$ is reduced [deterministic, injective]. We will write rbliw-graph as a short designation for reduced bliw-graph. The importance and usefulness of rliw-graphs and rbliw-graphs will be revealed during this paper. Our first result about them is the following:

\begin{lem}\label{r_unique}
Let $\Ga$ be an rliw-graph and let $\av\in V(\Ga)$. If $w\in \rvec{L_\av}(\Ga)$ {\normalfont{[}}$w\in \lvec{L_\av}(\Ga)${\normalfont]}, then there is unique $\bv\in V_r(\Ga)$ {\normalfont[}$\bv\in V_l(\Ga)${\normalfont]} and a unique $\av-\bv$ {\normalfont[}$\bv-\av${\normalfont]} walk $p$ such that $w\in\w(p)$. 
\end{lem}

\begin{proof}
We prove only for the $\rvec{L_\av}(\Ga)$ case since the $\lvec{L_\av}(\Ga)$ case is its dual. Let $w\in\rvec{L_\av}(\Ga)$ and $z=\hat{\la}_w$. Let $p$ and $p'$ be two walks in $\Ga$, both starting at $\av$ and ending at right vertices, such that $w\in\w(p)\cap\w(p')$. We just need to prove that $p=p'$. The uniqueness of $\bv$ is then a trivial consequence. We assume that $p\neq p'$ with the intention of getting a contradiction. Without loss of generality, we can assume further that the initial elementary subpaths $\av\ev_1\av_1$ and $\av\ev_1'\av_1'$ of $p$ and $p'$, respectively, have $\ev_1\neq\ev_1'$.

Assume first that $z=x\in\oX$. If $\av$ is a right vertex, then the initial subwalks of $p$ and $p'$ of length 2 must be equivalent right basic paths, a contradiction since $\Ga$ is reduced. If $\av$ is a left vertex, then $\ev_1$ and $\ev_1'$ are both arrows labeled by $x$ with the same starting vertex, also a contradiction.

Assume now that $z=x\wedge y\in\wX$. If $\av$ is a left vertex, then $\av_1\neq\av_1'$ and $y\in\cb(\av_1)\cap\cb(\av_1')$. Thus there are two arrows $\ev_0=(\av_0,y,\av_1)$ and $\ev_0'=(\av_0',y,\av_1')$ in $\Ga$ for some $\av_0,\av_0'\in V_l(\Ga)$; whence $\av_0\ev_0\av_1\ev_1\av$ and $\av_0'\ev_0'\av_1'\ev_1'\av$ are two distinct equivalent left basic paths in $\Ga$, once again a contradiction. If $\av$ is a right vertex, then $\av_1$ and $\av_1'$ are two distinct left vertices with $x\in\cb(\av_1)\cap\cb(\av_1')$. Thus there are two arrows $\ev_0=(\av_1,x,\av_0)$ and $\ev_0'=(\av_1',x,\av_0')$ in $\Ga$ for some $\av_0,\av_0'\in V_r(\Ga)$. Once more we get a contradiction because $\av\ev_1\av_1\ev_0\av_0$ and $\av\ev_1'\av_1'\ev_0'\av_0'$ are two distinct equivalent right basic paths in $\Ga$. Therefore $p=p'$ and we have proved this lemma.
\end{proof}

\begin{cor}\label{unique_hom}
Let $\Ga$ and $\Ga'$ be two liw-graphs such that $\Ga'$ is reduced. If $\varphi:\Ga\to \Ga'$ is a homomorphism, then $\varphi$ is completely determined by the image of a vertex. In particular, there is at most one homomorphism $\varphi:\Ga\to\Ga'$ such that $\av\varphi=\av'$ for $\av\in V(\Ga)$ and $\av'\in V(\Ga')$.
\end{cor}

\begin{proof}
Let $\varphi:\Ga\to\Ga'$ be a homomorphism and choose $\av\in V(\Ga)$. Let $\bv\in V_r(\Ga)$ and $w\in L_{\av,\bv}(\Ga)$. Then $w\in L_{\av',\bv\varphi}(\Ga')$ for $\av'=\av\varphi$ by Lemma \ref{hom_wp}. However, by Lemma \ref{r_unique}, there exists a unique $\bv'\in V_r(\Ga')$ such that $w\in L_{\av',\bv'}(\Ga')$. Hence, there is only one option for $\bv\varphi$, namely $\bv'$. We have shown that the image of a vertex automatically determines the image of all right vertices. Similarly, choosing $w\in L_{\bv,\av}(\Ga)$ for any $\bv\in V_l(\Ga)$, we conclude that the image of $\av$ determines the image of all left vertices.

The second part of this corollary is an obvious consequence of the first.
\end{proof}

The next two lemmas contain technical details needed to prove Proposition \ref{conv}.

\begin{lem}\label{lem53}
Let $\Ga$ and $\Ga'$ be two liw-graphs such that $\Ga'$ is reduced. Let $\av,\bv\in V(\Ga)$ and $\av',\bv'\in V(\Ga')$ such that $\sb(\av)=\sb(\av')$, $\sb(\bv)=\sb(\bv')$, and $L_{\av,\bv}(\Ga)\subseteq L_{\av',\bv'}(\Ga')$. For each $\ev\in \vec{E}(\Ga)$, there exists a unique $\ev'\in\vec{E}(\Ga')$ such that $\l(\ev)=\l(\ev')$, $L_{\av,\lv(\ev)}(\Ga)\subseteq L_{\av',\lv(\ev')}(\Ga')$ and $L_{\rv(\ev),\bv}(\Ga)\subseteq L_{\rv(\ev'),\bv'}(\Ga')$.
\end{lem}

\begin{proof}
Let $x=\l(\ev)$, $w_1\in L_{\av,\lv(\ev)}(\Ga)$ and $w_2\in L_{\rv(\ev),\bv}(\Ga)$. Then 
$$w=w_1xw_2\in L_{\av,\bv}(\Ga)\subseteq L_{\av',\bv'}(\Ga'),$$
and there exists $p'\in P_{\Ga'}(\av',\bv')$ such that $w\in\w(p')$. Then $p'$ has a decomposition $p_1',p_3',p_2'$ such that $w_1\in\w(p_1')$, $w_2\in\w(p_2')$ and $p_3'$ is an arrow elementary path with an arrow $\fv'\in \vec{E}(\Ga')$ labeled by $x$. 

By Lemma \ref{r_unique}, $p_1'p_3'$ is the only path of $\Ga'$ starting at $\av'$ and ending at a right vertex labeled by $w_1x$, while $p_3'p_2'$ is the only path of $\Ga'$ starting at a left vertex and ending at $\bv'$ labeled by $xw_2$. Hence, $\fv'$ is the only arrow of $\Ga'$ such that $\l(\fv')=x$, $w_1\in L_{\av',\lv(\fv')}(\Ga')$ and $w_2\in L_{\rv(\fv'),\bv'}(\Ga')$. We have shown that if the arrow $\ev'$ announced in the statement of this lemma exists, then it must be unique and equal to $\fv'$.

Let us prove now that $L_{\av,\lv(\ev)}(\Ga)\subseteq L_{\av',\lv(\fv')}(\Ga')$, and so let $u\in L_{\av,\lv(\ev)}(\Ga)$. Again $uxw_2\in L_{\av',\bv'}(\Ga')$, and there exists $q'\in P_{\Ga'}(\av',\bv')$ such that $uxw_2\in\w(q')$. By the uniqueness of $p_3'p_2'$, the path $q'$ has a decomposition $q_1',p_3',p_2'$ such that $u\in\w(q_1')$. Hence, the final vertex of $q_1'$ is $\lv(\fv')$ and $u\in L_{\av',\lv(\fv')}(\Ga')$. We have shown that $L_{\av,\lv(\ev)}(\Ga)\subseteq L_{\av',\lv(\fv')}(\Ga')$. We show that $L_{\rv(\ev),\bv}(\Ga)\subseteq L_{\rv(\fv'),\bv'}(\Ga')$ similarly.
\end{proof}

With the notation introduced in the previous lemma, the language inclusion $L_{\av,\bv}(\Ga)\subseteq L_{\av',\bv'}(\Ga')$ induces a (unique) mapping $\psi_{\oE}:\vec{E}(\Ga)\to\vec{E}(\Ga')\,,\; \ev\mapsto \ev'$ that preserves the labels. Without any further comments, we shall always use the notation $\psi_{\oE}$ for the mapping from the arrows of $\Ga$ to the arrows of $\Ga'$ induced by any such language inclusion. Note however that $\psi_{\oE}$ may not preserve incidence, that is, we cannot guaranty that $\ev\psi_{\oE}$ and $\ev_1\psi_{\oE}$ have a vertex in common if $\ev$ and $\ev_1$ are arrows in $\Ga$ with a vertex in common (see the example of Figure \ref{hom_exist} below). Nevertheless, we can prove a weaker property:

\begin{lem}\label{weak_adj}
With the notation introduced above and in the previous lemma, let $\ev,\ev_1\in\vec{E}(\Ga)$. Then:
\begin{itemize}
\item[$(i)$] If $(\lv(\ev),\rv(\ev_1))\in\ol{E}(\Ga)$, then $(\lv(\ev\psi_{\oE}),\rv(\ev_1\psi_{\oE}))\in \ol{E}(\Ga')$.
\item[$(ii)$] If $\lv(\ev)=\lv(\ev_1)$, then $(\lv(\ev\psi_{\oE}),\bv_1'),(\lv(\ev_1\psi_{\oE}),\bv_1')\in \ol{E}(\Ga')$ for some $\bv_1'\in V_r(\Ga')$. 
\item[$(iii)$] If $\rv(\ev)=\rv(\ev_1)$, then $(\av_1',\rv(\ev\psi_{\oE})),(\av_1',\rv(\ev_1\psi_{\oE}))\in \ol{E}(\Ga')$ for some $\av_1'\in V_l(\Ga')$. 
\end{itemize}
\end{lem}

\begin{proof}
$(i)$. Let $y=\l(\ev_1)$ and choose $w\in L_{\av,\lv(\ev_1)}(\Ga)$. Then 
$wy\in L_{\av,\rv(\ev_1)}(\Ga)\subseteq L_{\av,\lv(\ev)}(\Ga)$. By Lemma \ref{lem53}, $w\in L_{\av',\lv(\ev_1\psi_{\oE})}(\Ga')$ and $wy\in L_{\av',\lv(\ev\psi_{\oE})}(\Ga')$. Hence,
$$wy\in L_{\av',\rv(\ev_1\psi_{\oE})}(\Ga')\cap L_{\av',\lv(\ev\psi_{\oE})}(\Ga').$$
By Lemma \ref{r_unique}, there is only one path $p'$ in $\Ga'$ starting at $\av'$ and ending at a right vertex such that $wy\in\w(p')$. Thus $p'\in P_{\Ga'}(\av',\rv(\ev_1\psi_{\oE}))$. Furthermore, $wy\in\w(p_1')$ for some $p_1'\in P_{\Ga'}(\av',\lv(\ev\psi_{\oE}))$. Then $p_1'$ is decomposable into $p',p_2'$, where $p_2'$ is a right elementary path, due to the uniqueness of $p'$. Hence $(\lv(\ev\psi_{\oE}),\rv(\ev_1\psi_{\oE}))\in \ol{E}(\Ga')$.

The statements $(ii)$ and $(iii)$ are immediate consequences of $(i)$. For example, to prove $(ii)$, choose another arrow $\ev_2$ in $\Ga$ such that $(\lv(\ev),\rv(\ev_2))\in\ol{E}(\Ga)$ and apply $(i)$ twice, one time for $\ev$ and $\ev_2$, and the other time for $\ev_1$ and $\ev_2$. 
\end{proof}

To help understand the statement of Lemma \ref{weak_adj}, we illustrate statement $(ii)$ in Figure \ref{ex_weak}.
\begin{figure}[ht]
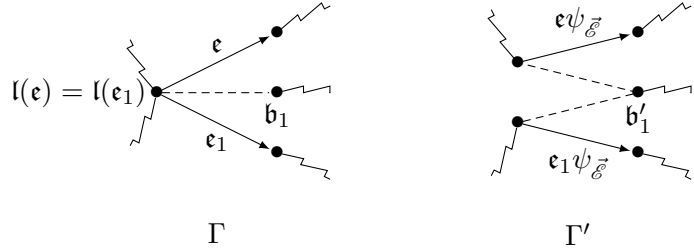

\centering $$\tikz[scale=.8, shorten <=2pt, shorten >=3pt, >=latex]{
\coordinate (1) at (1,2);
\coordinate (2) at (3,3);
\coordinate (3) at (3,1);
\coordinate (11) at (3,2);
\coordinate (4) at (7,1.5);
\coordinate (5) at (7,2.5);
\coordinate (6) at (9,1);
\coordinate (7) at (9,2);
\coordinate (8) at (9,3);
\coordinate[label=below:$\Ga$] (9) at (2,0);
\coordinate[label=below:$\Ga'$] (10) at (8,0);
\draw (1) node {$\bullet$};
\draw (1) node [left] {$\lv(\ev)=\lv(\ev_1)$};
\draw (7) node [below] {$\bv_1'$};
\draw (11) node [below] {$\bv_1$};
\draw (2) node {$\bullet$};
\draw (3) node {$\bullet$};
\draw (4) node {$\bullet$};
\draw (5) node {$\bullet$};
\draw (6) node {$\bullet$};
\draw (7) node {$\bullet$};
\draw (8) node {$\bullet$};
\draw (11) node {$\bullet$};
\draw[densely dashed] (1)--(11);
\draw[densely dashed] (4)--(7)--(5);
\draw[->] (1) to node [above] {$\ev$} (2);
\draw[->] (1) to node [below] {$\ev_1$} (3);
\draw[->] (5) to node [above] {$\ev\psi_{\oE}$} (8);
\draw[->] (4) to node [below] {$\ev_1\psi_{\oE}$} (6);
\draw[decorate,decoration=saw] (1)--(.5,3);
\draw[decorate,decoration=saw] (1)--(0.5,1);
\draw[decorate,decoration=saw] (2)--(4,3.5);
\draw[decorate,decoration=saw] (3)--(4,.5);
\draw[decorate,decoration=saw] (5)--(6.5,3.5);
\draw[decorate,decoration=saw] (4)--(6.5,.5);
\draw[decorate,decoration=saw] (6)--(10,0.5);
\draw[decorate,decoration=saw] (7)--(10,2);
\draw[decorate,decoration=saw] (11)--(4,2);
\draw[decorate,decoration=saw] (8)--(10,3.5);}$$
\caption{Illustration of Lemma \ref{weak_adj}$.(ii)$.}\label{ex_weak}
\end{figure}

Next, we give a necessary and sufficient conditions for the existence of a homomorphism $\varphi:\Ga\to\Ga'$ when $\Ga'$ is reduced. It becomes evident the importance of $\psi_{\oE}$ preserving incidence.

\begin{prop}\label{conv}
Let $\Ga$ and $\Ga'$ be two liw-graphs such that $\Ga'$ is reduced. Let $\av,\bv\in V(\Ga)$ and $\av',\bv'\in V(\Ga')$ such that $\sb(\av)=\sb(\av')$ and $\sb(\bv)=\sb(\bv')$. There exists a homomorphism $\varphi:\Ga\to\Ga'$ such that $\av\varphi=\av'$ and $\bv\varphi=\bv'$ \iff\ $L_{\av,\bv}(\Ga)\subseteq L_{\av',\bv'}(\Ga')$ and the induced mapping $\psi_{\oE}$ preserves incidence. Further, the restriction of $\varphi_E$ to $\vec{E}$ is precisely $\psi_{\oE}$.
\end{prop}

\begin{proof}
Let $\varphi:\Ga\to\Ga'$ be a homomorphism such that $\av\varphi=\av'$ and $\bv\varphi=\bv'$. Then $L_{\av_1,\bv_1}(\Ga)\subseteq L_{\av_1\varphi,\bv_1\varphi}(\Ga')$ for any $\av_1,\bv_1\in V(\Ga)$ by Lemma \ref{hom_wp}. In particular, $L_{\av,\bv}(\Ga)\subseteq L_{\av',\bv'}(\Ga')$. Moreover, if $\ev\in\vec{E}(\Ga)$, then $L_{\av,\lv(\ev)}(\Ga)\subseteq L_{\av',\lv(\ev\varphi_E)}(\Ga')$ and $L_{\rv(\ev),\bv}(\Ga)\subseteq L_{\rv(\ev\varphi_E),\bv'}(\Ga')$. Thus $\ev\varphi_E=\ev\psi_{\oE}\,$ and $\psi_{\oE}$ preserves incidence since it is the restriction of $\varphi_E$ to $\vec{E}$.

Assume now that $L_{\av,\bv}(\Ga)\subseteq L_{\av',\bv'}(\Ga')$ and that the mapping $\psi_{\oE}$ preserves incidence. For each arrow $\ev_1=(\av_1,x,\bv_1)$ of $\Ga$, define $\av_1\varphi_V=\lv(\ev_1\psi_{\oE})$ and $\bv_1\varphi_V=\rv(\ev_1\psi_{\oE})$. The mapping $\varphi_V$ is well defined because $\psi_{\oE}$ preserves incidence. Further, $\ev_1\psi_{\oE}=(\av_1\varphi_V,x,\bv_1\varphi_V)$. Set $\ev_1\varphi_E=\ev_1\psi_{\oE}$. If $\ev_2=(\av_2,\bv_2)$ is a line of $\Ga$, then $(\av_2\varphi_V,\bv_2\varphi_V)$ is a line of $\Ga'$ by Lemma \ref{weak_adj}$.(i)$. Set $\ev_2\varphi_E=(\av_2\varphi_V,\bv_2\varphi_V)$. Clearly, $\varphi=(\varphi_V,\varphi_E)$ is now a homomorphism from $\Ga$ to $\Ga'$ such that $\av\varphi=\av'$ and $\bv\varphi=\bv'$ by construction.
\end{proof}

The condition of $\psi_{\oE}$ preserving incidence cannot be omitted. The Figure \ref{hom_exist} has two rliw-graphs $\Ga$ and $\Ga'$. Note that $L_{\av,\bv}(\Ga)=L_{\av',\bv'}(\Ga')=Z^+$ for
$$Z=\{(y'\wedge y)\}\cup y'Z_1^*y$$
where
$$Z_1=\{(z_1\wedge z_2')\,:\; z_1,z_2\in\{x,y\}\}\cup x(x'\wedge x)^*x'\,.$$
However, there is no homomorphism from $\Ga$ to $\Ga'$ since $\cb(\av_1)=\{x,y\}$ but no vertex of $\Ga'$ contains both $x$ and $y$ in its content. What fails here for not existing such a homomorphism is the fact that $\psi_{\oE}$ does not preserve incidence: observe that $(\av_1,y,\bv)\psi_{\oE}=(\av_1',y,\bv')$ and $(\av_1,x,\bv_2)\psi_{\oE}=(\av_0',x,\bv_2')$.

\begin{figure}[ht]
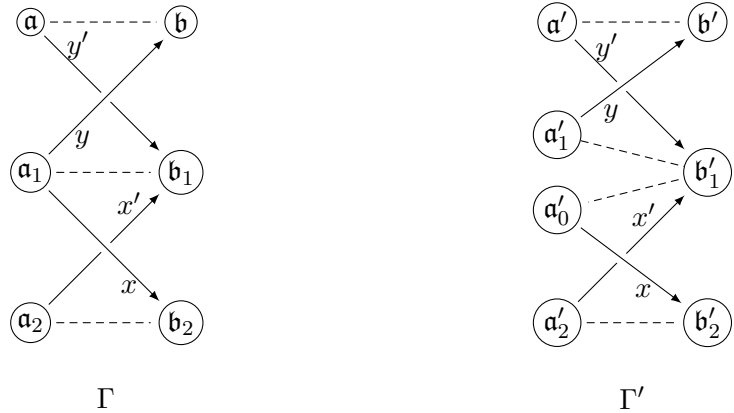

\centering $$\tikz[shorten <=2pt, shorten >=3pt, >=latex]{
\node [circle, inner sep=1pt, draw] (1) at (0,3) {$\av_1$};
\node [circle, inner sep=1pt, draw] (2) at (2,3) {$\bv_1$};
\node [circle, inner sep=1pt, draw] (3) at (0,1) {$\av_2$};
\node [circle, inner sep=1pt, draw] (4) at (2,1) {$\bv_2$};
\node [circle, inner sep=1pt, draw] (5) at (0,5) {$\av$};
\node [circle, inner sep=1pt, draw] (6) at (2,5) {$\bv$};
\node [circle, inner sep=1pt, draw] (7) at (7,3.5) {$\av_1'$};
\node [circle, inner sep=1pt, draw] (8) at (7,2.5) {$\av_0'$};
\node [circle, inner sep=1pt, draw] (9) at (9,3) {$\bv_1'$};
\node [circle, inner sep=1pt, draw] (10) at (7,1) {$\av_2'$};
\node [circle, inner sep=1pt, draw] (11) at (9,1) {$\bv_2'$};
\node [circle, inner sep=1pt, draw] (12) at (7,5) {$\av'$};
\node [circle, inner sep=1pt, draw] (13) at (9,5) {$\bv'$};
\coordinate[label=center:$\Ga$] (14) at (1,0);
\coordinate[label=center:$\Ga'$] (15) at (8,0);
\draw[densely dashed] (1)--(2);
\draw[densely dashed] (3)--(4);
\draw[densely dashed] (5)--(6);
\draw[densely dashed] (7)--(9)--(8);
\draw[densely dashed] (10)--(11);
\draw[densely dashed] (12)--(13);
\draw[->] (12) to node [above, pos=.3] {$y'$} (9);
\draw[-,line width=5pt,color=white] (7) to (13);
\draw[->] (7) to node [below, pos=.3] {$y$} (13);
\draw[->] (10) to node [above, pos=.6] {$x'$} (9);
\draw[-,line width=5pt,color=white] (8) to (11);
\draw[->] (8) to node [below, pos=.6] {$x$} (11);
\draw[->] (5) to node [above, pos=.3] {$y'$} (2);
\draw[-,line width=5pt,color=white] (1) to (6);
\draw[->] (1) to node [below, pos=.3] {$y$} (6);
\draw[->] (3) to node [above, pos=.7] {$x'$} (2);
\draw[-,line width=5pt,color=white] (1) to (4);
\draw[->] (1) to node [below, pos=.7] {$x$} (4);}$$
\caption{An example showing the necessity of $\psi_{\oE}$ preserving incidence.}\label{hom_exist}
\end{figure}

Next, we present a process to transform each bliw-graph $\A$ into a unique rbliw-graph $\B$. We begin by introducing two operations on bliw-graphs:\vspace*{.3cm}

\noindent{\bf Elementary determination}: Let $\A=(\av,\Ga,\bv)$ be a bliw-graph and let $p:=\bv_0\ol{\ev_1}\av_1\vec{{\fv}_1}\bv_1$ and $q:=\bv_0\ol{\ev_2}\av_2\vec{{\fv}_2}\bv_2$ be two equivalent right basic paths in $\Ga$. Let $\eta$ be the equivalence on $V(\Ga)$ generated by $\{(\av_1,\av_2),(\bv_1,\bv_2)\}$. The quotient bliw-graph $\A/\eta=(\av\eta,\Ga/\eta,\bv\eta)$ is called an \emph{elementary determination} of $\A$.
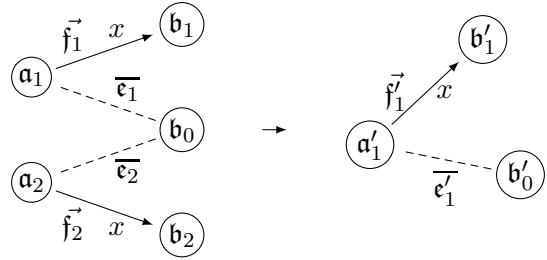
$$\tikz[shorten <=2pt, shorten >=3pt, >=latex]{
\node [circle, inner sep=1pt, draw] (1) at (2,1.4) {$\bv_0$};
\node [circle, inner sep=1pt, draw] (2) at (0,2.1) {$\av_1$};
\node [circle, inner sep=1pt, draw] (3) at (2,2.8) {$\bv_1$};
\node [circle, inner sep=1pt, draw] (4) at (0,.7) {$\av_2$};
\node [circle, inner sep=1pt, draw] (5) at (2,0) {$\bv_2$};
\draw[densely dashed] (1) to node [above, pos=.3] {$\ol{\ev_1}$} (2);
\draw[densely dashed] (1) to node [below, pos=.3] {$\ol{\ev_2}$} (4);
\draw[->] (2) to node [above, pos=.6] {$x$} node [above, pos=.2] {$\vec{{\fv}_1}$} (3);
\draw[->] (4) to node [below, pos=.6] {$x$} node [below, pos=.2] {$\vec{{\fv}_2}$} (5);
\draw[->] (3,1.4)--(3.5,1.4);
\node [circle, inner sep=1pt, draw] (6) at (6.5,.8) {$\bv_0'$};
\node [circle, inner sep=1pt, draw] (7) at (4.5,1.2) {$\av_1'$};
\node [circle, inner sep=1pt, draw] (8) at (6,2.6) {$\bv_1'$};
\draw[densely dashed] (6) to node [below] {$\ol{\ev_1'}$} (7);
\draw[->] (7) to node [right] {$x$} node [above, pos=.1] {$\vec{{\fv}_1'}$} (8);}$$

\noindent{\bf Elementary injection}: Let $\A=(\av,\Ga,\bv)$ be a bliw-graph and let $p:=\av_1\vec{{\fv}_1}\bv_1\ol{\ev_1}\av_0$ and $q:=\av_2\vec{{\fv}_2}\bv_2\ol{\ev_2}\av_0$ be two equivalent left basic paths in $\Ga$. Let $\eta$ be the equivalence on $V(\Ga)$ generated by $\{(\av_1,\av_2),(\bv_1,\bv_2)\}$. The quotient bliw-graph $\A/\eta=(\av\eta,\Ga/\eta,\bv\eta)$ is called an \emph{elementary injection} of $\A$.
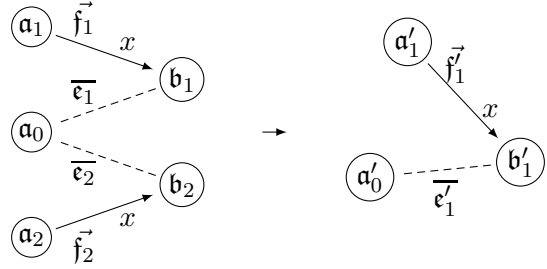
$$\tikz[shorten <=2pt, shorten >=3pt, >=latex]{
\node [circle, inner sep=1pt, draw] (1) at (0,1.4) {$\av_0$};
\node [circle, inner sep=1pt, draw] (2) at (2,2.1) {$\bv_1$};
\node [circle, inner sep=1pt, draw] (3) at (0,2.8) {$\av_1$};
\node [circle, inner sep=1pt, draw] (4) at (2,.7) {$\bv_2$};
\node [circle, inner sep=1pt, draw] (5) at (0,0) {$\av_2$};
\draw[densely dashed] (2) to node [above, pos=.7] {$\ol{\ev_1}$} (1);
\draw[densely dashed] (4) to node [below, pos=.7] {$\ol{\ev_2}$} (1);
\draw[->] (3) to node [above, pos=.7] {$x$} node [above, pos=.3] {$\vec{{\fv}_1}$} (2);
\draw[->] (5) to node [below, pos=.7] {$x$} node [below, pos=.3] {$\vec{{\fv}_2}$} (4);
\draw[->] (3,1.4)--(3.5,1.4);
\node [circle, inner sep=1pt, draw] (6) at (4.5,.8) {$\av_0'$};
\node [circle, inner sep=1pt, draw] (7) at (6.5,1) {$\bv_1'$};
\node [circle, inner sep=1pt, draw] (8) at (5,2.6) {$\av_1'$};
\draw[densely dashed] (7) to node [below] {$\ol{\ev_1'}$} (6);
\draw[->] (8) to node [right, pos=.6] {$x$} node [above, pos=.4] {$\vec{{\fv}_1'}$} (7);}$$

We will refer simultaneously to both elementary determinations and elementary injections as \emph{elementary reductions}. A \emph{reduction} is a (possibly empty) sequence of elementary reductions. Each time we apply an elementary reduction to a bliw-graph $\A$, the number of vertices decreases. Hence, each reduction is always composed by a finite number of elementary reductions, that is, the elementary reductions have what is commonly called the \emph{noetherian property}: we cannot apply elementary reductions indefinitely. 

Let $\{p_1,q_1\}$ and $\{p_2,q_2\}$ be two sets of equivalent left and/or right basic paths in $\Ga$, and let $\A/\eta_1$ and $\A/\eta_2$ be the elementary reductions associated with $\{p_1,q_1\}$ and $\{p_2,q_2\}$, respectively. Let $\eta=\eta_1\vee\eta_2$. Note that $p_2\eta_1$ and $q_2\eta_1$ are either equivalent left or right basic paths in $\A/\eta_1$ if $\eta_1\neq\eta$, or $p_2\eta_1=q_2\eta_1$ otherwise. Thus either $\A/\eta$ is the elementary reduction of $\A/\eta_1$ associated with the basic paths $p_2\eta_1$ and $q_2\eta_1$, or $\A/\eta=\A/\eta_1$ otherwise. Similarly, either $\A/\eta$ is the elementary reduction of $\A/\eta_2$ associated with the basic paths $p_1\eta_2$ and $q_1\eta_2$, or $\A/\eta=\A/\eta_2$ otherwise. Therefore, the elementary reductions have also the \emph{local confluency property}, that is, if $\A_1$ and $\A_2$ are elementary reductions of $\A$, then  there exists a bliw-graph $\A_3$ such that $\A_3$ is a reduction of both $\A_1$ and $\A_2$. 

The next result follows from a general property of noetherian locally confluent systems of rules (see \cite{newman}):

\begin{prop}\label{rbliw}
If $\A$ is a bliw-graph, then there exists a unique rbliw-graph that can be obtained from $\A$ by a reduction.
\end{prop}

The unique rbliw-graph $\B$ that can be obtained from $\A$ by a reduction is called the \emph{reduced form} of $\A$. A \emph{complete reduction} of $\A$ is a sequence of elementary reductions that led to $\B$. In Figure \ref{red_bliw} we present a complete reduction for the bliw-graph $\A=(\al_3,\Ga,\be_2)$ of Figure \ref{ex_bliw}. Note that this complete reduction folds the right side of the drawing depicting $\A$ over its left side.

\begin{figure}[ht]
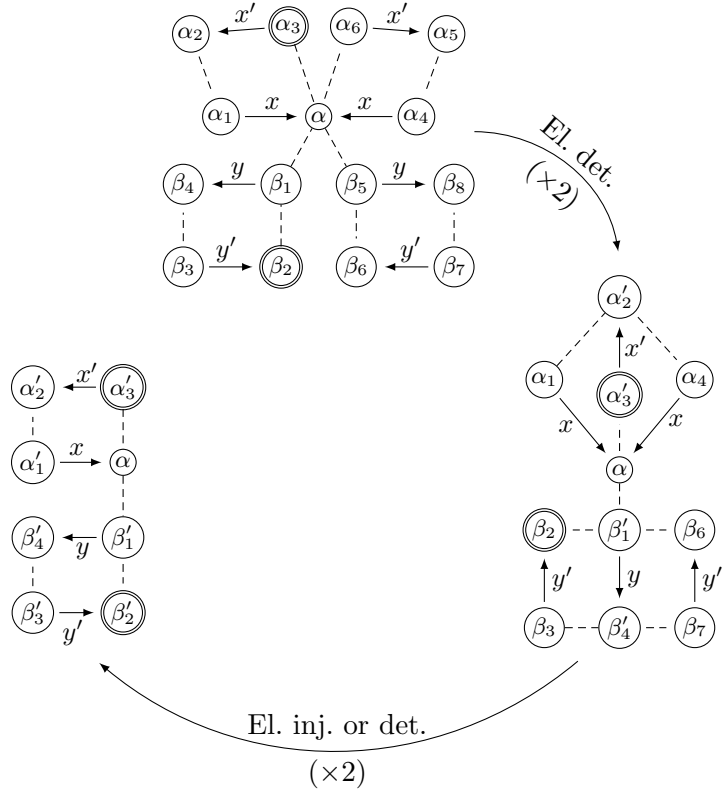

$$\tikz[shorten <=2pt, shorten >=3pt, >=latex]{
\node [circle, inner sep=1pt, draw] (1) at (7,8.2) {\footnotesize$\al$};
\node [circle, inner sep=1pt, draw] (2) at (5.3,9.3) {\footnotesize$\al_2$};
\node [circle, inner sep=1pt, draw] (3) at (5.7,8.2) {\footnotesize$\al_1$};
\node [circle, double, inner sep=1pt, draw] (4) at (6.6,9.4) {\footnotesize$\al_3$};
\node [circle, inner sep=1pt, draw] (5) at (8.7,9.3) {\footnotesize$\al_5$};
\node [circle, inner sep=1pt, draw] (6) at (8.3,8.2) {\footnotesize$\al_4$};
\node [circle, inner sep=1pt, draw] (7) at (7.4,9.4) {\footnotesize$\al_6$};
\node [circle, inner sep=1pt, draw] (8) at (6.5,7.3) {\footnotesize$\be_1$};
\node [circle, double, inner sep=1pt, draw] (9) at (6.5,6.2) {\footnotesize$\be_2$};
\node [circle, inner sep=1pt, draw] (10) at (5.2,6.2) {\footnotesize$\be_3$};
\node [circle, inner sep=1pt, draw] (11) at (5.2,7.3) {\footnotesize$\be_4$};
\node [circle, inner sep=1pt, draw] (12) at (7.5,7.3) {\footnotesize$\be_5$};
\node [circle, inner sep=1pt, draw] (13) at (7.5,6.2) {\footnotesize$\be_6$};
\node [circle, inner sep=1pt, draw] (14) at (8.8,6.2) {\footnotesize$\be_7$};
\node [circle, inner sep=1pt, draw] (15) at (8.8,7.3) {\footnotesize$\be_8$};
\draw[densely dashed] (4)--(1)--(7);
\draw[densely dashed] (2)--(3);
\draw[densely dashed] (5)--(6);
\draw[densely dashed] (9)--(8)--(1)--(12)--(13);
\draw[densely dashed] (10)--(11);
\draw[densely dashed] (14)--(15);
\draw[->] (3) to node [above=-1pt] {\small$x$} (1);
\draw[->] (4) to node [above=-1pt, pos=.4] {\small$x'$} (2);
\draw[->] (6) to node [above=-1pt] {\small$x$} (1);
\draw[->] (7) to node [above=-1pt] {\small$x'$} (5);
\draw[->] (8) to node [above=-2pt, pos=.4] {\small$y$} (11);
\draw[->] (10) to node [above=-2pt, pos=.4] {\small$y'$} (9);
\draw[->] (12) to node [above=-2pt, pos=.4] {\small$y$} (15);
\draw[->] (14) to node [above=-2pt, pos=.4] {\small$y'$} (13);
\draw[->, bend left=40] (9,8) to node [sloped, above] {El.~det.} node [sloped, below] {($\times 2$)} (11,6.3);
\node [circle, inner sep=1pt, draw] (16) at (11,3.5) {\footnotesize$\al$};
\node [circle, inner sep=1pt, draw] (17) at (11,5.8) {\footnotesize$\al_2'$};
\node [circle, inner sep=1pt, draw] (18) at (10,4.7) {\footnotesize$\al_1$};
\node [circle, double, inner sep=1pt, draw] (19) at (11,4.5) {\footnotesize$\al_3'$};
\node [circle, inner sep=1pt, draw] (20) at (12,4.7) {\footnotesize$\al_4$};
\node [circle, inner sep=1pt, draw] (21) at (11,2.7) {\footnotesize$\be_1'$};
\node [circle, double, inner sep=1pt, draw] (22) at (10,2.7) {\footnotesize$\be_2$};
\node [circle, inner sep=1pt, draw] (23) at (10,1.4) {\footnotesize$\be_3$};
\node [circle, inner sep=1pt, draw] (24) at (11,1.4) {\footnotesize$\be_4'$};
\node [circle, inner sep=1pt, draw] (25) at (12,2.7) {\footnotesize$\be_6$};
\node [circle, inner sep=1pt, draw] (26) at (12,1.4) {\footnotesize$\be_7$};
\draw[densely dashed] (18)--(17)--(20);
\draw[densely dashed] (16)--(19);
\draw[->] (18) to node [left] {\small$x$} (16);
\draw[->] (20) to node [right=-1pt, pos=.4] {\small$x$} (16);
\draw[->] (19) to node [right=-2pt] {\small$x'$} (17);
\draw[densely dashed] (16)--(21)--(22);
\draw[densely dashed] (21)--(25);
\draw[densely dashed] (23)--(24)--(26);
\draw[->] (21) to node [right=-1pt] {\small$y$} (24);
\draw[->] (23) to node [right=-1pt] {\small$y'$} (22);
\draw[->] (26) to node [right=-1pt] {\small$y'$} (25);
\draw[->, bend left=40] (10.5,1) to node [sloped, above] {El.~inj.~or det.} node [sloped, below] {($\times 2$)} (4,1);
\node [circle, inner sep=1pt, draw] (27) at (4.4,3.6) {\footnotesize$\al$};
\node [circle, inner sep=1pt, draw] (28) at (3.2,4.6) {\footnotesize$\al_2'$};
\node [circle, inner sep=1pt, draw] (29) at (3.2,3.6) {\footnotesize$\al_1'$};
\node [circle, double, inner sep=1pt, draw] (30) at (4.4,4.6) {\footnotesize$\al_3'$};
\node [circle, inner sep=1pt, draw] (31) at (4.4,2.6) {\footnotesize$\be_1'$};
\node [circle, double, inner sep=1pt, draw] (32) at (4.4,1.6) {\footnotesize$\be_2'$};
\node [circle, inner sep=1pt, draw] (33) at (3.2,1.6) {\footnotesize$\be_3'$};
\node [circle, inner sep=1pt, draw] (34) at (3.2,2.6) {\footnotesize$\be_4'$};
\draw[densely dashed] (30)--(27)--(31)--(32);
\draw[densely dashed] (29)--(28);
\draw[densely dashed] (33)--(34);
\draw[->] (29) to node [above=-1pt, pos=.4] {\small$x$} (27);
\draw[->] (30) to node [above=-2pt, pos=.3] {\small$x'$} (28);
\draw[->] (31) to node [below=-2pt, pos=.4] {\small$y$} (34);
\draw[->] (33) to node [below=-1pt, pos=.4] {\small$y'$} (32);
}$$
\caption{A complete reduction for the bliw-graph of Figure \ref{ex_bliw}.}\label{red_bliw}
\end{figure}

\begin{prop}\label{hom_red}
Let $\varphi:\A\to\A'$ be a homomorphism such that $\A'$ is reduced. Let $\B$ be the reduced form of $\A$ and let $\psi:\A\to\B$ be the natural $E$-surjective epimorphism. Then there exists a unique homomorphism $\varphi':\B\to\A'$ such that $\varphi=\psi\varphi'$.
\end{prop}

\begin{proof}
Note first that if the homomorphism $\varphi':\B\to\A'$ exists, then it must be unique due to Corollary \ref{unique_hom}.

Let $p_1:=\av\ol{\ev}_1\av_1\vec{f}_1\bv_1$ and $p_2:=\av\ol{\ev}_2\av_2\vec{f}_2\bv_2$ be two equivalent right basic paths of $\A$ such that $\l(\vec{f}_1)=\l(\vec{f}_2)=x$, and let $\A_1=\A/\eta$ be the elementary reduction of $\A$ associated with the equivalent right basic paths $p_1$ and $p_2$. Since $\A'$ is reduced, then $\av_1\varphi=\av_2\varphi$ and $\bv_1\varphi=\bv_2\varphi$. Hence, $\varphi_1:\A_1\to\A'$ defined by $(\bv\eta)\varphi_1=\bv\varphi$ is a homomorphism too. Further, $\varphi=\psi_1\varphi_1$ where $\psi_1$ is the natural homomorphism from $\A$ to $\A_1$.

A similar conclusion can be obtained if $p_1$ and $p_2$ are equivalent left basic paths instead. The proof of this proposition now follows from applying sequentially the previous conclusions to all the elementary reductions of a complete reduction of $\A$. 
\end{proof}

\section{The rliw-graph characterization of a $\Dc$-class}\label{sec6}

From now on, and if nothing is said in contrary, $S$ will denote the locally inverse semigroup $\LI\langle X;R\rangle$. Our goal in this section is to construct an rliw-graph $\Ga_e$ for each idempotent $e\in E(S)$ (Proposition \ref{rliw}), and then show that $e\Dc f$ for $e,f\in E(S)$ \iff\ $\Ga_e$ and $\Ga_f$ are isomorphic (Proposition \ref{car_D_liw}). Hence, these graphs somehow characterize the $\Dc$-classes of $S$. In the path to show that $\Ga_e$ and $\Ga_f$ are isomorphic \iff\ $e\Dc f$, we describe all isomorphisms from $\Ga_e$ to $\Ga_f$. We show that these isomorphisms are in a one-to-one correspondence with the elements of $\Rcc_e\cap\Lcc_f$ (Proposition \ref{isom_D}). As a corollary, we conclude that the group of automorphisms of $\Ga_e$ is isomorphic to $\Hcc_e$ (Corollary \ref{aut_gr}).

Let $\Sl=\{\lv_a\,:\;a\in S\}$ and $\Sr=\{\rv_a\,:\;a\in S\}$, two disjoint copies of the elements of $S$. Fix an idempotent $e\in S$ and set
$$\V_l=\{\lv_a\,:\; a\Lc e\}\subseteq \Sl\qquad\mbox{ and }\qquad \V_r=\{\rv_b\,:\; b\Rc e\}\subseteq\Sr\;,$$
two disjoint copies of $\Lcc_e$ and $\Rcc_e$, respectively. Let also
$$\lE=\{(\lv_a,\rv_b)\,:\; a\in\Lcc_e\;\mbox{ and }\;b\in \Rcc_e\cap V(a)\}.$$
By Lemma \ref{inv_lis}, if $a\in \Lcc_e\cap (x\wedge x')S$ for some $x\in\oX$, then $a$ has a unique inverse $a'$ in $\Rcc_e\cap S(x\wedge x')$. Further, if $a_1'$ is another inverse of $a$ in $\Rcc_e$, then $a'=a_1'(x\wedge x')$ by Lemma \ref{inv_yy}, and so $a'x=a_1'x$. Let
$$\oE=\{(\lv_a,x,\rv_{a'x})\,:\;a\in\Lcc_e\cap (x\wedge x')S\mbox{ and } a'\in V(a)\cap \Rcc_e\}.$$
Hence, for each $a\in\Lcc_e$ and $x\in\oX$, there is at most one element in $\oE$ of the form $(\lv_a,x,\rv_b)$ with $b\in\Rcc_e$. Let $\Ga_e$ be the bipartite graph $\Ga_e=(\V,\E)$ with vertices $\V=\V_l\cup\V_r$, partitioned into the left vertices $\V_l$ and the right vertices $\V_r$, and edges $\E=\lE\cup\oE$ where $\lE$ is the set of lines and $\oE$ is the set of arrows.

Recall the semigroups $\ol{S}_1$ and $\ol{S}_2$ introduced in Section \ref{sec3}. In Figure \ref{Ga_xx} we depict $\Ga_{x'x}$ for the semigroup $\ol{S}_1$ and $\Ga_{z'z}$ for the semigroup $\ol{S}_2$. To make it easier to construct these graphs, we include also in Figure \ref{Ga_xx} the egg-box picture of the $\Dc$-class of $x'x$ in $\ol{S}_1$ and the egg-box picture of the $\Dc$-class of $z'z$ in $\ol{S}_2$. Observe that these graphs are both reduced liw-graphs. This is not a coincidence as it is proved in the next result.

\begin{figure}[ht]
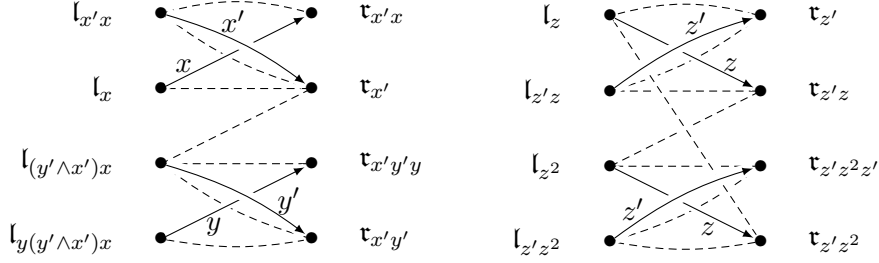

\begin{subfigure}[b]{.52\linewidth}
\centering\begin{tabular}{|c|c|c|c|}
\hline $x'x\;^*$ & $x'\;^*$ & $x'y'y$ & $x'y'$ \\ \hline
$x$ & $xx'\;^*$ & $xx'y'y$ & $xx'y'$ \\ \hline
$(y'\wedge x')x$ & $y'\wedge x'\;^*$ & $y'y\;^*$ & $y'\;^*$ \\ \hline
$\!y(y'\wedge x')x\!$ & $\!y(y'\wedge x')\!$ & $y$ & $yy'\;^*$ \\ \hline
\end{tabular}
\caption{egg-box picture of $\Dcc_{x'x}$ in $\ol{S}_1$.}
\end{subfigure}
\begin{subfigure}[b]{.47\linewidth}
\centering\begin{tabular}{|c|c|}
\hline $z'z^2$ & $z'z^2z'$ \\ [.2cm]
$z'z\,^*$ & $z'\,^*$ \\ \hline
$z^2\,^*$ & $z^2z'$ \\ [.2cm]
$z$ & $zz'\,^*$  \\ \hline 
\end{tabular}
\caption{Egg-box picture of $\Dcc_{z'z}$ in $\ol{S}_2$.}
\end{subfigure}\vspace*{.8cm}

\begin{subfigure}[b]{.52\linewidth}
\centering $$\tikz[shorten <=2pt, shorten >=3pt, >=latex]{
\coordinate[label=left:$\lv_{y(y'\wedge x')x}$\hspace*{.5cm}] (1) at (0,0);
\coordinate[label=left:$\lv_{(y'\wedge x')x}$\hspace*{.5cm}] (2) at (0,1);
\coordinate[label=left:$\lv_{x}$\hspace*{.5cm}] (3) at (0,2);
\coordinate[label=left:$\lv_{x'x}$\hspace*{.5cm}] (4) at (0,3);
\draw (1) node {$\bullet$};
\draw (2) node {$\bullet$};
\draw (3) node {$\bullet$};
\draw (4) node {$\bullet$};
\coordinate[label=right:\hspace*{.5cm}$\rv_{x'y'}$] (5) at (2,0);
\coordinate[label=right:\hspace*{.5cm}$\rv_{x'y'y}$] (6) at (2,1);
\coordinate[label=right:\hspace*{.5cm}$\rv_{x'}$] (7) at (2,2);
\coordinate[label=right:\hspace*{.5cm}$\rv_{x'x}$] (8) at (2,3);
\draw (5) node {$\bullet$};
\draw (6) node {$\bullet$};
\draw (7) node {$\bullet$};
\draw (8) node {$\bullet$};
\draw[densely dashed] (1) to [bend right=10] (5);
\draw[densely dashed] (5) to [bend left=10] (2);
\draw[densely dashed] (6)--(2)--(7)--(3);
\draw[densely dashed] (4) to [bend left=15] (8);
\draw[densely dashed] (4) to [bend right=10] (7);
\draw[-,line width=3pt,color=white] (3) to (8);
\draw[-,line width=3pt,color=white] (1) to (6);
\draw[->] (3) to node [above=-2pt, pos=.15] {$x$} (8);
\draw[->] (1) to node [below=-2pt, pos=.35] {$y$} (6);
\draw[-,line width=5pt,color=white] (4) to [bend left=10] (7);
\draw[-,line width=5pt,color=white] (2) to [bend left=10] (5);
\draw[->] (4) to [bend left=10] node [above=-1pt, pos=0.45] {$x'$} (7);
\draw[->] (2) to [bend left=10] node [above=-1pt, pos=.85] {$y'$} (5);}$$
\caption{The graph $\Ga_{x'x}$ with respect to $\ol{S}_1$.}\label{Ga_xxa}
\end{subfigure}
\begin{subfigure}[b]{.47\linewidth}
\centering $$\tikz[shorten <=2pt, shorten >=3pt, >=latex]{
\coordinate[label=left:$\lv_{z'z^2}$\hspace*{.5cm}] (1) at (0,0);
\coordinate[label=left:$\lv_{z^2}$\hspace*{.5cm}] (2) at (0,1);
\coordinate[label=left:$\lv_{z'z}$\hspace*{.5cm}] (3) at (0,2);
\coordinate[label=left:$\lv_{z}$\hspace*{.5cm}] (4) at (0,3);
\draw (1) node {$\bullet$};
\draw (2) node {$\bullet$};
\draw (3) node {$\bullet$};
\draw (4) node {$\bullet$};
\coordinate[label=right:\hspace*{.5cm}$\rv_{z'z^2}$] (5) at (2,0);
\coordinate[label=right:\hspace*{.5cm}$\rv_{z'z^2z'}$] (6) at (2,1);
\coordinate[label=right:\hspace*{.5cm}$\rv_{z'z}$] (7) at (2,2);
\coordinate[label=right:\hspace*{.5cm}$\rv_{z'}$] (8) at (2,3);
\draw (5) node {$\bullet$};
\draw (6) node {$\bullet$};
\draw (7) node {$\bullet$};
\draw (8) node {$\bullet$};
\draw[densely dashed] (1) to [bend right=10] (6);
\draw[densely dashed] (6)--(2)--(7)--(3);
\draw[densely dashed] (3) to [bend right=10] (8);
\draw[densely dashed] (8) to [bend right=15] (4);
\draw[densely dashed] (5) to [bend left=10] (1);
\draw[-,line width=5pt,color=white] (4) to (5);
\draw[densely dashed] (4)--(5);
\draw[-,line width=3pt,color=white] (4) to (7);
\draw[-,line width=3pt,color=white] (2) to (5);
\draw[->] (4) to node [above=-2pt, pos=.8] {$z$} (7);
\draw[->] (2) to node [below=-1pt, pos=.65] {$z$} (5);
\draw[-,line width=5pt,color=white] (3) to [bend left=10] (8);
\draw[-,line width=5pt,color=white] (1) to [bend left=10] (6);
\draw[->] (3) to [bend left=10] node [above=-2pt, pos=0.6] {$z'$} (8);
\draw[->] (1) to [bend left=10] node [above=-2pt, pos=.15] {$z'$} (6);}$$
\caption{The graph $\Ga_{z'z}$ with respect to $\ol{S}_2$.}\label{Ga_xxb}
\end{subfigure}
\caption{Two examples of $\Ga_e$.}\label{Ga_xx}
\end{figure}

\begin{prop}\label{rliw}
$\Ga_e$ is a reduced liw-graph.
\end{prop}

\begin{proof}
$\Ga_e$ is clearly an oriented bipartite graph. Let $a\in\Lcc_e$ and $x\in\oX$ such that $a\in (x\wedge x')S$. Then $(\lv_a,x,\rv_{a'x})\in\oE$ for $a'\in V(a)\cap \Rcc_e$, and $x\in\cb(\lv_a)$. If $b\in\Rcc_e\cap S(y'\wedge y)$ for some $y\in\oX$, then let $a$ be the unique inverse of $by'$ in $\Lcc_e\cap (y\wedge y')S$ given by Lemma \ref{inv_lis}. Note that $(\lv_a,y,\rv_b)\in\oE$ and $y\in\cb(\rv_b)$. Assume now that $(\lv_a,x,\rv_b)\in\oE$ and let $a'$ be the unique inverse of $a$ in $\Rcc_e\cap S(x\wedge x')$. Then $(\lv_a,\rv_{a'})\in\lE$ and $b=a'x$. Further, $(\lv_{x'a},\rv_b)\in\lE$ and $(\lv_{x'a},x',\rv_{a'})\in\oE$ because $x'a\in \Lcc_e\cap V(b)$ and $a'=bx'$. Thus, both conditions $(ii)$ and $(iii)$ of the definition of liw-graph are satisfied. 

We need to prove now that $\Ga_e$ is connected to conclude that $\Ga_e$ is an liw-graph. Observe that it is enough to show that $\rv_e$ is connected to each right vertex $\rv_a$ of $\Ga_e$ because each left vertex of $\Ga_e$ is connected to some right vertex.

Let $u=z_1\cdots z_n$, with $n\geq 0$ and $z_i\in\wX$ for $i=1,\cdots,n$, such that $a=eu$ ($eu=e$ if $n=0$). Set $u_i=z_1\cdots z_i$ and $a_i=eu_i$ for $i=1,\cdots,n$, and set $a_0=e$. All $a_i$ are $\Rc$-related with $e$ because $a$ is $\Rc$-related with $e$. We are done once we show that each $\rv_{a_{i-1}}$ is connected with $\rv_{a_i}$.  We need to consider two cases: $z_i=x\in\oX$ and $z_i=x\wedge y\in\wX$. In both cases, let $c_i=a_{i-1}(x\wedge x')\Rc a_i$, let $b_i$ be the inverse of $c_i$ belonging to $\Lcc_e\cap (x\wedge x')S$ and let $d_i\in S$ such that $c_id_i=a_{i-1}$. Observe that $b_i\in V(a_{i-1})$ because
$$b_ia_{i-1}b_i=b_ic_ib_i=b_i\quad\mbox{ and }\quad  a_{i-1}b_ia_{i-1}=c_ib_ic_id_i=c_id_i=a_{i-1}.$$
Hence $\ol{\ev}_i=(\lv_{b_i},\rv_{a_{i-1}})\in\lE$.

Case 1: $z_i=x\in\oX$. By definition of $b_i$, $\vec{\fv}_i=(\lv_{b_i},x,\rv_{a_i})\in\oE$ and $\rv_{a_{i-1}}$ connects to $\rv_{a_i}$ by the right basic path $p_i=\rv_{a_{i-1}}\,\ol{\ev}_i\,\lv_{b_i}\,\vec{\fv}_i\,\rv_{a_i}$.

Case 2: $z_i=x\wedge y\in\wX$. By Corollary \ref{trans}.$(iii)$, $\varphi_{x\wedge y}:\Lcc_{c_i}\to\Lcc_{a_i}$ is a right translation with inverse right translation $\varphi_{x\wedge x'}:\Lcc_{a_i}\to\Lcc_{c_i}$. Hence $c_i=a_i(x\wedge x')$. We can now easily conclude that $b_i\in V(a_i)$:
$$b_ia_ib_i=b_ia_i(x\wedge x')b_i=b_ic_ib_i=b_i\quad\mbox{ and }\quad  a_ib_ia_i=c_ib_ic_i(x\wedge y)=a_i.$$
Thus $\ol{\fv}_i=(\lv_{b_i},\rv_{a_i})\in\lE$ and $p_i=\rv_{a_{i-1}}\,\ol{\ev}_i\,\lv_{b_i}\,\ol{\fv}_i\,\rv_{a_i}$ is a walk of length 2 connecting $\rv_{a_{i-1}}$ to $\rv_{a_i}$.

Now that we have shown that $\Ga_e$ is an liw-graph, we need to prove that $\Ga_e$ is reduced. Let $\rv_b\ol{\ev}\lv_a\vec{\fv}\rv_{a'x}$ and $\rv_b\ol{\ev_1}\lv_{a_1}\vec{\fv_1}\rv_{a_1'x}$ be two equivalent right basic paths with $\l(\vec{\fv})=\l(\vec{\fv_1})=x$. Then both $a$ and $a_1$ are $\Lc$-related inverses of $b$ belonging to $((x\wedge x')\mu]_r$; and both $ab$ and $a_1b$ are $\Lc$-related idempotents belonging to $(x\wedge x')S$. Thus $ab=a_1b$ and $a=a_1$. We have shown that $\Ga_e$ is deterministic. 

Let $\lv_a\vec{\ev}\rv_{a'x}\ol{\fv}\lv_b$ and $\lv_{a_1}\vec{\ev_1}\rv_{a_1'x}\ol{\fv_1}\lv_{b}$ be two equivalent left basic paths with $\l(\vec{\ev})=\l(\vec{\ev_1})=x$. Note that we can choose $a'$ and $a_1'$ both in $S(x\wedge x')$. Then $a'x$ and $a_1'x$ are $\Rc$-related inverses of $b$ in $S(x'\wedge x)$; and so $ba'x$ and $ba_1'x$ are $\Rc$-related idempotents in $((x'\wedge x)\mu]_l$. Thus $ba'x=ba_1'x$, $a'x=a_1'x$ and $a'=a_1'$. Hence, $aa'$ and $a_1a_1'$ must be $\Lc$-related idempotents in $((x\wedge x')\mu]_r$, and consequently $aa'=a_1a_1'$ and $a=a_1$. We have shown that $\Ga_e$ is injective, and therefore $\Ga_e$ is a reduced liw-graph. 
\end{proof}

Let us look more carefully at the proof of the connectivity of $\Ga_e$ since it contains more information than what is stated in the previous proposition. We began by choosing an arbitrary $a\in\Rcc_e$ and assuming that $eu=a$ for some $u=z_1\cdots z_n$ with $n\geq 0$ and $z_i\in \wX$. Then, we defined a sequence $a_0=e$ and $a_i=ez_1\cdots z_i$, for $i=1,\cdots,n$, of elements from $\Rcc_e$, and constructed a walk $p_i$ in $\Gamma_e$ from $\rv_{a_{i-1}}$ to $\rv_{a_i}$. By construction of $p_i$, observe that $z_i\in\w(p_i)$. Thus,  $p:=\,p_1\cdots p_n$ is a walk in $\Ga_e$ with initial vertex $\rv_e$, final vertex $\rv_a$, and such that $u\in\w(p)$. By Lemma \ref{r_unique}, $p$ is the unique walk $q$ in $\Ga_e$ starting at $\rv_e$ and ending at a right vertex of $\Ga_e$ such that $u\in\w(q)$. We can now conclude that, for each $a\in\Rcc_e$, if $eu=a$ for some $u\in\wX^*$, then there exists a unique path $p$ in $\Ga_e$ starting at $\rv_e$, ending at a right vertex, and such that $u\in\w(p)$. Furthermore, $p$ must end at $\rv_a$. This conclusion is just part of a particular case of the following lemma.

\begin{lem}\label{lem62}
Let $a,b\in\Rcc_e$ and $u\in\wX^*$. Then $au=b$ \iff\ $u\in L_{\rv_a,\rv_b}(\Ga_e)$.
\end{lem}

\begin{proof}
Assume that $au=b$ and let $v\in\wX^*$ such that $ev=a$. By the conclusion made prior to this lemma, there are two walks $p$ and $q$ starting at $\rv_e$ and ending at $\rv_a$ and $\rv_b$, respectively, such that $v\in\w(p)$ and $vu\in\w(q)$. By Lemma \ref{r_unique}, $q$ has a decomposition $p,q_1$. Thus $q_1$ is a walk from $\rv_a$ to $\rv_b$ such that $u\in\w(q_1)$. We have shown that $u\in L_{\rv_a,\rv_b}(\Ga_e)$.

Conversely, assume that $u\in L_{\rv_a,\rv_b}(\Ga_e)$ and let $p$ be the (unique) $\rv_a-\rv_b$ walk such that $u\in\w(p)$. Note that $p$ has even length because it starts and ends at right vertices. Consider the decomposition of $p$ into subwalks, all of length 2, say $p_1,\cdots,p_n$. Thus $\w(p_i)\subseteq \wX$ for $i=1,\cdots,n$, and $u=z_1\cdots z_n$ with $z_i\in\w(p_i)$. Let $a_0,a_1,\cdots,a_n\in\Rcc_e$ such that $a_0=a$ and $\rv_{a_i}$ is the final vertex of the walk $p_i$ for $i=1,\cdots,n$. Then $u_i=z_1\cdots z_i\in L_{\rv_a,\rv_{a_i}}(\Ga_e)$. Also $u_0=\iota\in L_{\rv_a,\rv_a}(\Ga_e)$ since it labels the trivial path from $\rv_a$ to itself. We prove next that $au_i=a_i$ by induction, thus ending the proof of this lemma since $a_n=b$ and $u_n=u$.

Clearly $au_0=a=a_0$. So, assume that $au_{i-1}=a_{i-1}$ and let us prove that $au_i=a_i$, or equivalently, that $a_{i-1}z_i=a_i$. If $z_i=y$, then $p_i=\rv_{a_{i-1}}\ol{e_i}\lv_b\vec{f_i}\rv_{a_i}$ where $a_{i-1}\in V(b)$, $b\in \Lcc_e\cap (y\wedge y')S$ and $a_i=b'y$ for any $b'\in V(b)\cap\Rcc_e$. In particular, for $b'=a_{i-1}$, we get $a_i=a_{i-1}y$ as desired. Now, if $z_i=y\wedge y_1$, then $p_i=\rv_{a_{i-1}}\ol{e_i}\lv_b\ol{f_i}\rv_{a_i}$ where $a_{i-1}, a_i\in V(b)\cap \Rcc_e$, $b\in \Lcc_e\cap (y\wedge y')S$ and $a_i=S(y_1'\wedge y_1)$. By Lemma \ref{inv_yy}, $a_{i-1}(y\wedge y_1)=a_i$ once again as desired.
\end{proof}

Let us illustrate Lemma \ref{lem62} by considering the rliw-graph $\Ga_{x'x}$ of Figure \ref{Ga_xxa} and the elements $a=(x'x)\ol{\mu}_1\in\ol{S}_1$ and $b=(x'y')\ol{\mu}_1\in\ol{S}_1$. Our plan is to find all $u\in\widehat{X_1}^+$ such that $au=b$. We need to consider four sets of walks first. If $P_1$ is the set of all walks from $\rv_{x'x}$ to $\rv_{x'}$ avoiding crossing the vertex $\rv_{x'}$, then
$$W_1=\{\w(p):\, p\in P_1\}=\{(x'\wedge x)\}^*\cdot\{x',(x'\wedge x')\}\,.$$
If $P_2$ is the set of all walks from $\lv_{(y'\wedge x')x}$ to $\rv_{x'y'}$ avoiding crossing the vertex $\lv_{(y'\wedge x')x}$, then
$$W_2=\{\w(p):\, p\in P_2\}=\{y',(y'\wedge y')\}\cdot\{(y\wedge y')\}^*\,.$$
If $P_3$ is the set of all walks from $\rv_{x'}$ to itself avoiding crossing the vertices $\rv_{x'}$ and $\lv_{(y'\wedge x')x}$, then
$$W_3=\{\w(p):\,p\in P_3\}=\{(x\wedge x'),\iota\}\cup\{\iota,x\}\cdot\{(x'\wedge x)\}^*\cdot\{x',(x'\wedge x')\}\,.$$
If $P_4$ is the set of all walks from $\lv_{(y'\wedge x')x}$ to itself avoiding crossing the vertices $\rv_{x'}$ and $\lv_{(y'\wedge x')x}$, then 
$$W_4=\{\w(p):\,p\in P_4\}=\{(y'\wedge y),\iota\}\cup\{y',(y'\wedge y')\}\cdot\{(y\wedge y')\}^*\cdot\{y,\iota\}\,.$$
Now, observe that
$$L_{\rv_{x'x},\rv_{x'y'}}(\Ga_{x'x})=W=W_1\cdot (W_3^*\cdot W_4^*\cdot (y'\wedge x'))^*\cdot W_3^*\cdot W_4^*\cdot W_2\,.$$
Hence, by Lemma \ref{lem62}, $au=b$ \iff\ $u\in W$.

There is a right-left dual result corresponding to Lemma \ref{lem62}. However, there is a peculiarity in this dual result, the roles of $a$ and $b$ must be switched. For that reason, we include a brief proof of it.

\begin{lem}\label{lem63}
Let $a,b\in\Lcc_e$ and $u\in\wX^*$. Then $ua=b$ \iff\ $u\in L_{\lv_b,\lv_a}(\Ga_e)$.
\end{lem}

\begin{proof}
We prove this lemma only for the case where $u\in\wX$. The general case follows similarly to the proof of Lemma \ref{lem62}. Assume that $xa=b$ for $x\in\oX$. Then $\ev_1=(\lv_b,x,\rv_{b'x})\in\vec{E}(\Ga_e)$ for $b'\in V(b)\cap\Rcc_e$. Since
$$b'xab'x=b'bb'x=b'x\quad\mbox{ and }\quad ab'xa=ab'b=ae=a\,,$$
we conclude that $\ev_2=(\lv_a,\rv_{b'x})\in\ol{E}(\Ga_e)$. Thus $x\in L_{\lv_b,\lv_a}(\Ga_e)$. Conversely, if $x\in L_{\lv_b,\lv_a}(\Ga_e)$, then there exists $b'\in V(b)\cap\Rcc_e$ such that $b'x\in V(a)$ and $b\in (x\wedge x')S$. Let $c\in S$ such that $b'xc=b'$. Note that $c$ exists because $b'\Rc b'x$. Then 
$$b'xab'=b'xab'xc=b'xc=b'\quad\mbox{ and }\quad xab'xa=xa\,.$$
We have shown that $xa$ and $b$ are $\Lc$-related inverses of $b'$ belonging to $(x\wedge x')S$; whence $b=xa$ since $S$ is locally inverse.

Assume now that $(x\wedge y)a=b$ for $x,y\in\oX$. Since $b\in (x\wedge y)S$, there exists a unique inverse $b'$ of $b$ in $\Rcc_e$ such that $b'\in S(x\wedge y)$. Then $\ev_1=(\lv_b,\rv_{b'})\in\ol{E}(\Ga_e)$. Also $b'ab'=b'bb'=b$ and $ab'a=cbb'b=cb=a$ for some $c\in S$ such that $a=cb$. Thus $\ev_2=(\lv_a,\rv_{b'})\in\ol{E}(\Ga_e)$ and $x\wedge y\in L_{\lv_b,\lv_a}(\Ga_e)$. Conversely, if $x\wedge y\in L_{\lv_b,\lv_a}(\Ga_e)$, then $x\in\cb(\lv_b)$ and there exists $c\in\Rcc_e$ such that $a,b\in V(c)$ and $y\in\cb(\rv_c)$. Now
$$bc\in ((x\wedge x')\mu]_r\cap((y'\wedge y)\mu]_l=((x\wedge y)\mu]$$
and $c(x\wedge y)=c$. Hence $(x\wedge y)a\in V(c)$ because $c(x\wedge y)ac=cac=c$ and $(x\wedge y)ac(x\wedge y)a=(x\wedge y)aca=(x\wedge y)a$. Finally, we conclude that $(x\wedge y)a=b$ because they are $\Lc$-related inverses of $c$ in $(x\wedge x')S$.
\end{proof}

We use now the rliw-graph $\Ga_{z'z}$ of Figure \ref{Ga_xxb} (with respect to the semigroup $\ol{S}_2$) to illustrate Lemma \ref{lem63}. Recall the characterization of $u\ol{\mu}_2$ given in Example 2, and observe also that 
$$uz'z^2\,\ol{\mu}_2\, z \qquad\Leftrightarrow\qquad u\,\ol{\mu}_2\,z^2\;\mbox{ or }\; u\,\ol{\mu}_2\,z^2z'\,.$$
Hence $L_{\lv_z,\lv_{z'z^2}}(\Ga_{z'z})=\{u\in\widehat{X_2}^+\,:\; \la_u=z\;\;\mbox{ and }\;\; n(u)\mbox{ odd}\}$. Note that this conclusion is not easy to get directly from the rliw-graph $\Ga_{zz'}$.

In Proposition \ref{conv} we showed that, in general, if $\Ga'$ is an rliw-graph, then there exists an liw-graph homomorphism $\varphi:\Ga\to\Ga'$ such that $\av\varphi=\av'$ and $\bv\varphi=\bv'$, for $\av,\bv\in V(\Ga)$ and $\av',\bv'\in V(\Ga')$ with $\sb(\av)=\sb(\av')$ and $\sb(\bv)=\sb(\bv')$, \iff\ $L_{\av,\bv}(\Ga)\subseteq L_{\av',\bv'}(\Ga')$ and the associated mapping $\psi_{\oE}$ preserves incidence. We saw in Figure \ref{hom_exist} that we cannot omit, in general, the condition of $\psi_{\oE}$ preserving incidence. However, for the cases we are interested in, the cases where the liw-graphs are associated with a locally inverse semigroup presentation, we will see next that this latter condition is always satisfied.

\begin{lem}\label{pre_adj}
Let $e,f\in E(S)$, $\av,\bv\in V(\Ga_e)$ and $\av_1,\bv_1\in V(\Ga_f)$ such that $\sb(\av)=\sb(\av_1)$ and $\sb(\bv)=\sb(\bv_1)$. If $L_{\av,\bv}(\Ga_e)\subseteq L_{\av_1,\bv_1}(\Ga_f)$, then the induced mapping $\psi_{\oE}$ always preserves incidence.
\end{lem}

\begin{proof}
Let $\ev_1=(\lv_s,x_1,\rv_{t_1})$ and $\ev_2=(\lv_s,x_2,\rv_{t_2})$ be two arrows of $\Ga_e$ with the same starting vertex. Then $\ev_1\psi_{\oE}=(\lv_{s_1},x_1,\rv_{t_3})$ and $\ev_2\psi_{\oE}=(\lv_{s_2},x_2,\rv_{t_4})$ for some $s_1,s_2\in\Lcc_f$ and $t_3,t_4\in\Rcc_f$. By Lemma \ref{weak_adj}.$(ii)$ there exists also a right vertex $\bv\in V_r(\Ga_f)$ such that $(\lv_{s_1},\bv),\,(\lv_{s_2},\bv)\in\ol{E}(\Ga_f)$. We will prove that $s_1=s_2$. By duality, we conclude also that if $\ev_1'$ and $\ev_2'$ are two arrows with the same ending vertex, then $\ev_1'\psi_{\oE}$ and $\ev_2'\psi_{\oE}$ have the same ending vertex too; and this lemma becomes proved.

Let $s'\in V(s)\cap\Rcc_e\cap S(x_1\wedge x_1')$ and $s_1'\in V(s_1)\cap\Rcc_f\cap S(x_1\wedge x_1')$. Note that $s'$ and $s_1'$ exist because $s,s_1\in (x_1\wedge x_1')S$. Further, $ss's_1s_1'$ is an idempotent since $ss'$ and $s_1s_1'$ belong to the semilattice $((x_1\wedge x_1')\mu]$. We prove next that $ss's_1s_1'=s_1s_1'$.

Let $a\in\Lcc_e$ such that
$$\av=\left\{\begin{array}{ll}
\lv_a &\mbox{ if }\av\in V_l(\Ga_e) \\
\rv_{a'} &\mbox{ for some } a'\in V(a)\cap\Rcc_e \mbox{ if } \av\in V_r(\Ga_e)\, .
\end{array}\right.$$
Let also $u,v\in\wX^+$ such that $s'=u\mu$ and $a=v\mu$. Then
$$vu\in L_{\lv_a,\lv_s}(\Ga_e)\subseteq L_{\av,\lv_s}(\Ga_e)\subseteq L_{\av_1,\lv_{s_1}}(\Ga_f)$$
where the last inclusion follows from the definition of $\psi_{\oE}$. Note there is $a_1\in\Lcc_f$ such that $vu\in L_{\lv_{a_1},\lv_{s_1}}(\Ga_f)$. More precisely, $\av_1=\lv_{a_1}$ if $\av_1\in V_l(\Ga_f)$ or $a_1\in V(c_1)\cap\Lcc_f$ if $\av_1\in V_r(\Ga_f)$ and $\av_1=\rv_{c_1}$. Then $vus_1=a_1$ by Lemma \ref{lem63}, and
$$ss's_1s_1'\,\Lc\, s's_1s_1'=us_1s_1'\,\Lc\, vus_1s_1'=a_1s_1'\,\Lc\, s_1s_1'\,.$$
Hence $ss's_1s_1'=s_1s_1'$ since they are $\Lc$-related idempotents in the semilattice $((x_1\wedge x_1')\mu]$.

Note that $s\in (x_2\wedge x_2')S$ because $(\lv_s,x_2,\rv_{t_2})$ is an arrow of $\Ga_e$. Hence $s_1=ss's_1s_1's_1\in (x_2\wedge x_2')S$ and there is an arrow in $\Ga_f$ starting at $\lv_{s_1}$ and labeled by $x_2$. Since $\Ga_f$ is reduced, we must have $s_1=s_2$ as otherwise we would have two distinct equivalent right basic paths starting at $\bv$.
\end{proof}

Combining Proposition \ref{conv} with the previous lemma, we obtain:

\begin{cor}\label{conv_liw}
Let $e,f\in E(S)$, $\av,\bv\in V(\Ga_e)$ and $\av_1,\bv_1\in V(\Ga_f)$ such that $\sb(\av)=\sb(\av_1)$ and $\sb(\bv)=\sb(\bv_1)$. There exists a homomorphism $\varphi:\Ga_e\to\Ga_f$ such that $\av\varphi=\av_1$ and $\bv\varphi=\bv_1$ \iff\ $L_{\av,\bv}(\Ga_e)\subseteq L_{\av_1,\bv_1}(\Ga_f)$.
\end{cor}

In Figure \ref{Ga_y} we include again the rliw-graph $\Ga_{x'x}$ of Figure \ref{Ga_xxa} and also the rliw-graph $\Ga_{y'}$, both with respect to the semigroup $\ol{S}_1$. We order the vertices in $\Ga_{y'}$ so that it becomes evident that the two rliw-graphs are isomorphic. 
\begin{figure}[ht]
$$\tikz[shorten <=2pt, shorten >=3pt, >=latex]{
\coordinate[label=left:$\lv_{y(y'\wedge x')x}$\hspace*{.3cm}] (1) at (0,0);
\coordinate[label=left:$\lv_{(y'\wedge x')x}$\hspace*{.3cm}] (2) at (0,1);
\coordinate[label=left:$\lv_{x}$\hspace*{.3cm}] (3) at (0,2);
\coordinate[label=left:$\lv_{x'x}$\hspace*{.3cm}] (4) at (0,3);
\draw (1) node {$\bullet$};
\draw (2) node {$\bullet$};
\draw (3) node {$\bullet$};
\draw (4) node {$\bullet$};
\draw (1,-1) node {$\Ga_{x'x}$};
\coordinate[label=right:\hspace*{.3cm}$\rv_{x'y'}$] (5) at (2,0);
\coordinate[label=right:\hspace*{.3cm}$\rv_{x'y'y}$] (6) at (2,1);
\coordinate[label=right:\hspace*{.3cm}$\rv_{x'}$] (7) at (2,2);
\coordinate[label=right:\hspace*{.3cm}$\rv_{x'x}$] (8) at (2,3);
\draw (5) node {$\bullet$};
\draw (6) node {$\bullet$};
\draw (7) node {$\bullet$};
\draw (8) node {$\bullet$};
\draw[densely dashed] (1) to [bend right=10] (5);
\draw[densely dashed] (5) to [bend left=10] (2);
\draw[densely dashed] (6)--(2)--(7)--(3);
\draw[densely dashed] (4) to [bend left=15] (8);
\draw[densely dashed] (4) to [bend right=10] (7);
\draw[-,line width=3pt,color=white] (3) to (8);
\draw[-,line width=3pt,color=white] (1) to (6);
\draw[->] (3) to node [above=-2pt, pos=.15] {$x$} (8);
\draw[->] (1) to node [below=-2pt, pos=.35] {$y$} (6);
\draw[-,line width=5pt,color=white] (4) to [bend left=10] (7);
\draw[-,line width=5pt,color=white] (2) to [bend left=10] (5);
\draw[->] (4) to [bend left=10] node [above=-1pt, pos=0.45] {$x'$} (7);
\draw[->] (2) to [bend left=10] node [above=-1pt, pos=.85] {$y'$} (5);
\coordinate[label=left:$\lv_{yy'}$\hspace*{.3cm}] (9) at (6,0);
\coordinate[label=left:$\lv_{y'}$\hspace*{.3cm}] (10) at (6,1);
\coordinate[label=left:$\lv_{xx'y'}$\hspace*{.3cm}] (11) at (6,2);
\coordinate[label=left:$\lv_{x'y'}$\hspace*{.3cm}] (12) at (6,3);
\draw (9) node {$\bullet$};
\draw (10) node {$\bullet$};
\draw (11) node {$\bullet$};
\draw (12) node {$\bullet$};
\coordinate[label=right:\hspace*{.3cm}$\rv_{y'}$] (13) at (8,0);
\coordinate[label=right:\hspace*{.3cm}$\rv_{y'y}$] (14) at (8,1);
\coordinate[label=right:\hspace*{.3cm}$\rv_{y'\wedge x'}$] (15) at (8,2);
\coordinate[label=right:\hspace*{.3cm}$\rv_{(y'\wedge x')x}$] (16) at (8,3);
\draw (13) node {$\bullet$};
\draw (14) node {$\bullet$};
\draw (15) node {$\bullet$};
\draw (16) node {$\bullet$};
\draw (7,-1) node {$\Ga_{y'}$};
\draw[densely dashed] (9) to [bend right=10] (13);
\draw[densely dashed] (13) to [bend left=10] (10);
\draw[densely dashed] (14)--(10)--(15)--(11);
\draw[densely dashed] (12) to [bend left=15] (16);
\draw[densely dashed] (12) to [bend right=10] (15);
\draw[-,line width=3pt,color=white] (11) to (16);
\draw[-,line width=3pt,color=white] (9) to (14);
\draw[->] (11) to node [above=-2pt, pos=.15] {$x$} (16);
\draw[->] (9) to node [below=-2pt, pos=.35] {$y$} (14);
\draw[-,line width=5pt,color=white] (12) to [bend left=10] (15);
\draw[-,line width=5pt,color=white] (10) to [bend left=10] (13);
\draw[->] (12) to [bend left=10] node [above=-1pt, pos=0.45] {$x'$} (15);
\draw[->] (10) to [bend left=10] node [above=-1pt, pos=.85] {$y'$} (13);}$$
\caption{The rliw-graphs $\Ga_{x'x}$ and $\Ga_{y'}$ of $\ol{S}_1$.}\label{Ga_y}
\end{figure}
This fact is not a coincidence. Next, we show that the rliw-graph $\Ga_e$ is an invariant of the $\Dc$-class $\Dcc_e$. In other words, we show that if $f\in\Dcc_e\cap E(S)$, then $\Ga_e$ and $\Ga_f$ are always isomorphic.

\begin{prop}\label{isom_D}
Let $e,f\in E(S)$ such that $e\Dc f$. Then $\Ga_e$ and $\Ga_f$ are isomorphic. More precisely, for each $a\in \Rcc_e\cap\Lcc_f$, there exists a unique isomorphism $\varphi:\Ga_e\to\Ga_f$ such that $\lv_e\varphi =\lv_a$, and all isomorphisms from $\Ga_e$ to $\Ga_f$ are of this form.
\end{prop}

\begin{proof}
Let $a\in \Rcc_e\cap\Lcc_f$ and let $a'$ be the inverse of $a$ belonging to $\Lcc_e\cap\Rcc_f$. Observe that $se=a'$ \iff\ $sa=f$ for any $s\in S$. Thus $L_{\lv_f,\lv_a}(\Ga_f)=L_{\lv_{a'},\lv_e}(\Ga_e)$ by Lemma \ref{lem63}. By the previous corollary, there are two homomorphisms $\varphi:\Ga_e\to\Ga_f$ and $\varphi':\Ga_f\to\Ga_e$ such that $\lv_e\varphi=\lv_a$ and $\lv_a\varphi'=\lv_e$. Thus $\varphi:\Ga_e\to\Ga_f$ is an isomorphism with inverse isomorphism $\varphi'$ due to Corollary \ref{unique_hom}. Further, $\varphi$ is the only isomorphism from $\Ga_e$ to $\Ga_f$ such that $\lv_e\varphi=\lv_a$.

Finally, we must show there are no other isomorphisms from $\Ga_e$ to $\Ga_f$. So, let $\varphi:\Ga_e\to\Ga_f$ be an isomorphism such that $\lv_e\varphi=\lv_b$ and $\rv_e\varphi=\rv_{c}$ for $b\in\Lcc_f$ and $c\in\Rcc_f$. Then $L_{\rv_e,\lv_e}(\Ga_e)=L_{\rv_c,\lv_b}(\Ga_f)$, and $c\in V(b)$ because $\iota\in L_{\rv_e,\lv_e}(\Ga_e)$. Let $u,v\in\wX^*$ such that $u\mu=e$ and $v\mu=bc$. Then
$$u\in L_{\lv_e,\lv_e}(\Ga_e)\quad\mbox{ and }\quad v\in L_{\lv_b,\lv_b}(\Ga_f)$$
by Lemma \ref{lem63}. Since $L_{\lv_e,\lv_e}(\Ga_e)=L_{\lv_b,\lv_b}(\Ga_f)$, and again by Lemma \ref{lem63}, we conclude that $eb=ub=b$ and $bce=ve= e$. Thus $b\in\Rcc_e\cap\Lcc_f$ as wanted.
\end{proof}

The next result describes in more detail the isomorphisms of the previous proposition.

\begin{prop}\label{isom_descr}
Let $a\in S$, $a'\in V(a)$, $e=aa'$ and $f=a'a$, and let $\varphi_a$ be the isomorphism from $\Ga_e$ to $\Ga_f$ such that $\lv_e\varphi_a=\lv_a$ whose existence is guaranteed by the previous proposition. Then $\lv_s\varphi_a=\lv_{sa}$ for any $s\in\Lcc_e$ and $\rv_t\varphi_a=\rv_{a't}$ for any $t\in \Rcc_e$. In particular, $\rv_e\varphi_a=\rv_{a'}$.
\end{prop}

\begin{proof}
Let $s\in \Lcc_e$ and $u\in\wX^+$ such that $u\mu=s$. Then $ue=s$ and $u\in L_{\lv_s,\lv_e}(\Ga_e)$. So, $u\in L_{\lv_{s_1},\lv_a}(\Ga_f)$ for $\lv_{s_1}=\lv_s\varphi_a$, and $s_1=ua=sa$. The proof of $\rv_t\varphi_a=\rv_{a't}$ for $t\in\Rcc_e$ is similar once we show that $\rv_e\varphi_a=\rv_{a'}$. Let $\rv_e\varphi_a=\rv_c$ for some $c\in\Rcc_f$. Then $c\in V(a)$ because $\iota\in L_{\rv_e,\lv_e}(\Ga_e)=L_{\rv_c,\lv_a}(\Ga_f)$. If $u_1,u_2\in\wX^+$ are such that $u_1\mu=e$ and $u_2\mu=ac$, then 
$$u_1\in L_{\rv_e,\rv_e}(\Ga_e)\qquad\mbox{ and }\qquad u_2\in L_{\rv_c,\rv_c}(\Ga_f)$$
by Lemma \ref{lem62}. Since $L_{\rv_e,\rv_e}(\Ga_e)=L_{\rv_c,\rv_c}(\Ga_f)$, we have
$$ce=cu_1=c\qquad\mbox{ and }\qquad eac=eu_2=e\,.$$ 
Hence $c=a'$ because $c\in V(a)\cap \Lcc_e\cap\Rcc_f$.
\end{proof}

Observe that the isomorphism $\varphi_a$ is just the combination of both the right translation of $\Lcc_e$ associated with $a$ and the left translation of $\Rcc_e$ associated with $a'$.

An automorphism of an liw-graph $\Ga$ is an isomorphism from $\Ga$ into itself. Let $\Aut(\Ga)$ be the group of automorphisms of $\Ga$.

\begin{cor}\label{aut_gr}
Let $e\in E(S)$. Then $\Hcc_e$ and $\Aut(\Ga_e)$ are isomorphic groups. Further, for $b,b_1\in\Lcc_e$ and $c,c_1\in\Rcc_e$, there exists $\varphi\in\Aut (\Ga_e)$ such that
\begin{itemize}
\item[$(i)$] $\lv_b\varphi=\lv_{b_1}$ \iff\ $b\Hc b_1$;
\item[$(ii)$] $\rv_c\varphi=\rv_{c_1}$ \iff\ $c\Hc c_1$;
\item[$(iii)$] $\lv_b\varphi=\lv_{b_1}$ and $\rv_c\varphi=\rv_{c_1}$ \iff\ $bc=b_1c_1$.
\end{itemize}
\end{cor}

\begin{proof}
For each $a\in\Hcc_e$, let $\varphi_a$ be the automorphism of $\Ga_e$ such that $\lv_e\varphi_a=\lv_a$ given by Proposition \ref{isom_D}. Then the mapping $a\to\varphi_a$ is a bijection from $\Hcc_e$ to $\Aut(\Ga_e)$ also by Proposition \ref{isom_D}. Clearly $\varphi_{a_1}\varphi_{a_2}=\varphi_{a_1a_2}$ for $a_1,a_2\in \Hcc_e$ by the previous proposition. Hence, the mapping $a\to\varphi_a$ from $\Hcc_e$ to $\Aut(\Ga_e)$ is a group isomorphism. 

Let $b,b_1\in\Lcc_e$. Then $b\Hc b_1$ \iff\ there exists some $a\in\Hcc_e$ such that $ba=b_1$, that is, \iff\ there exists some $\varphi_a\in\Aut(\Ga_e)$ such that $\lv_b\varphi_a=\lv_{b_1}$. We have proved $(i)$ and note that $(ii)$ is the dual of $(i)$. Let us prove $(iii)$ now.

Suppose first that $\lv_b\varphi_a=\lv_{b_1}$ and $\rv_c\varphi_a=\rv_{c_1}$ for some $a\in\Hcc_e$. Then $b_1=ba$ and $c_1=a^{-1}c$ for $a^{-1}$ the inverse of $a$ in the group $\Hcc_e$. Hence $b_1c_1=baa^{-1}c=bec=bc$. Suppose now that $b_1c_1=bc$. Since $b,b_1\in\Lcc_e$ and $c,c_1\in\Rcc_e$, we must have $b\Hc b_1$ and $c\Hc c_1$. Thus, there are $a,a_1\in\Hcc_e$ such that $b_1=ba$ and $c_1=a_1c$. Let $s,s_1\in S$ such that $sb=e$ and $cs_1=e$. Then
$$bc=b_1c_1=baa_1c\quad\Rightarrow\quad sbcs_1=sbaa_1cs_1\quad\Rightarrow\quad e=aa_1\,,$$
and $a_1=a^{-1}$, the inverse of $a$ in $\Hcc_e$. Therefore $\lv_b\varphi_a=\lv_{b_1}$ and $\rv_c\varphi_a=\rv_{c_1}$. We have proved $(iii)$. 
\end{proof}

The conclusions of Corollary \ref{aut_gr} are corroborated by the two examples we are using. Note that $\ol{S}_1$ is an aperiodic semigroup (see Figure \ref{egg-boxa}) and only the identity mapping is an automorphism of the rliw-graph $\Ga_{x'x}$ of Figure \ref{Ga_xxa}. Now, about $\ol{S}_2$, the group $\Hcc_{z'z}$ is isomorphic to $\mathbb{Z}_2$. If we look carefully to the rliw-graph $\Ga_{z'z}$ of Figure \ref{Ga_xxb}, we see there are two automorphisms of $\Ga_{z'z}$: the identity mapping and another automorphism that sends $\lv_z$ into $\lv_{z^2}$. This latter automorphism permutes the upper part of $\Ga_{z'z}$ with its lower part.

We already know that the rliw-graphs $\Ga_e$ are an invariant of the $\Dc$-classes of $S$. The next result tells us even more. It shows us that two distinct $\Dc$-classes of $S$ have non-isomorphic rliw-graphs $\Ga_e$. Thus, the graphs $\Ga_e$ completely characterize and distinguish the $\Dc$-classes of $S$.

\begin{prop}\label{car_D_liw}
Let $e,f\in E(S)$. Then $e\Dc f$ \iff\ $\Ga_e$ and $\Ga_f$ are isomorphic. 
\end{prop}

\begin{proof}
By Proposition \ref{isom_D} we need to show only that $e\Dc f$ if $\varphi:\Ga_e\to\Ga_f$ is an isomorphism. Let $a\Lc e$ and $b\Lc f$ such that $\lv_e\varphi=\lv_b$ and $\lv_a\varphi=\lv_f$. If $u,v\in\wX^+$ are such that $u\mu=a$ and $v\mu=b$, then
$$u\in L_{\lv_a,\lv_e}(\Ga_e)=L_{\lv_f,\lv_b}(\Ga_f)\qquad\mbox{ and }\qquad v\in L_{\lv_b,\lv_f}(\Ga_f)=L_{\lv_e,\lv_a}(\Ga_e)\,.$$
Thus $ab=ub=f$ and $ba=va=e$ by Lemma \ref{lem63}, and consequently
$$aba=ae=a\qquad\mbox{ and }\qquad bab=bf=b\,.$$
Hence $b\in V(a)$ and $e=ba\Dc ab=f$ as wanted.
\end{proof}

\section{The rbliw-graph of an element with respect to a presentation}

In this section, we associate an rbliw-graph $\A_s=(\av,\Ga,\bv)$ to each element $s$ of $S$.  Of course, the rliw-graph $\Ga$ will be one of the graphs $\Ga_e$, with $e\Dc s$, discussed in the previous section. We then go into the description of the major concepts, such as the idempotents, the inverses of an element, the Green's relations, and the natural partial order, in terms of these rbliw-graphs. We also analyze how the semigroup product and the sandwich operation $\wedge$ on $S$ translate into operations on these rbliw-graphs.

Before defining $\A_s$, let us prove first the following result, which is crucial to guarantee that $\A_s$ is well defined.

\begin{prop}\label{iso_bliw}
Let $e,f\in E(S)$ and $s,t,s_1,t_1\in S$ such that $s\Lc e\Rc t$ and $s_1\Lc f\Rc t_1$. Then $\A=(\lv_s,\Ga_e,\rv_t)$ and $\A_1=(\lv_{s_1},\Ga_f,\rv_{t_1})$ are isomorphic \iff\ $e\Dc f$ and $st=s_1t_1$.
\end{prop}

\begin{proof}
Assume first that $\varphi:\A\to\A_1$ is an isomorphism. Then $\varphi:\Ga_e\to\Ga_f$ is an rliw-graph isomorphism such that $\lv_s\varphi=\lv_{s_1}$ and $\rv_t\varphi=\rv_{t_1}$. Thus $e\Dc f$ by Proposition \ref{car_D_liw}. By Proposition \ref{isom_descr}, there are $a\in\Rcc_e\cap\Lcc_f$ and $a'\in V(a)\cap\Lcc_e\cap\Rcc_f$ such that $\lv_b\varphi=\lv_{ba}$ for any $b\in\Lcc_e$ and $\rv_{b_1}\varphi=\rv_{a'b_1}$ for any $b_1\in\Rcc_e$. In particular, $s_1=sa$ and $t_1=a't$. Hence $s_1t_1=saa't=set=st$.

Assume now $e\Dc f$ and $st=s_1t_1$. Then $s\Rc s_1$ and $t\Lc t_1$. Let $a\in\Rcc_e\cap\Lcc_f$ such that $sa=s_1$ and consider the isomorphism $\varphi_a:\Ga_e\to\Ga_f$ given by Proposition \ref{isom_D}. Then $\lv_s\varphi_a=\lv_{s_1}$ by Proposition \ref{isom_descr}. Let $a'$ be the inverse of $a$ in $\Lcc_e\cap\Rcc_f$. Then $a't\Hc t_1$ and $s_1a't=saa't=st=s_1t_1$. Hence $t_1=a't$ and $\rv_t\varphi_a=\rv_{t_1}$ again by Proposition \ref{isom_descr}. We have shown that $\varphi_a$ is an rbliw-graph isomorphism from $\A$ to $\A_1$.
\end{proof}

For each $s\in S$, choose $e\in\Dcc_s$, $a\in\Lcc_e$ and $b\in\Rcc_e$ such that $s=ab$. Define $\A_s=(\lv_a,\Ga_e,\rv_b)$. The previous proposition tells us that the definition of $\A_s$ is, up to isomorphism, independent of the choice of $e$, $a$ and $b$, that is, if we choose different $f\in E(S)$, $a_1\in\Lcc_f$ and $b_1\in\Rcc_f$ such that $a_1b_1=s$, then $(\lv_a,\Ga_e,\rv_b)$ and $(\lv_{a_1},\Ga_f,\rv_{b_1})$ are isomorphic. Further, the previous proposition also tells us that $\A_s$ and $\A_t$ are non-isomorphic if $s\neq t$.

\begin{cor}\label{ig_bliw}
The following conditions are equivalent for $s,t\in S$:
\begin{itemize}
\item[$(i)$] $s=t$;
\item[$(ii)$] $\A_s$ and $\A_t$ are isomorphic;
\item[$(iii)$] $L(\A_s)=L(\A_t)$.
\end{itemize}
\end{cor}

\begin{proof}
$(i)$ and $(ii)$ are equivalent as observed above, and $(ii)$ implies $(iii)$ obviously. From Corollary \ref{conv_liw} it is also immediate that $(iii)$ implies $(ii)$.
\end{proof}

We will work with the rbliw-graphs $\A_s$ up to isomorphism and we will never be stuck to a particular representation $(\lv_a,\Ga_e,\rv_b)$ for $\A_s$. In every particular situation we will choose a representation for $\A_s$ that best fits our purposes. Usually, but not always, if $e\in E(S)$, then we choose the representation $(\lv_e,\Ga_e,\rv_e)$ for $\A_e$; and if $s\Rc e$, then we choose the representation $(\lv_e,\Ga_e,\rv_s)$ for $\A_s$.

Next, we characterize the Green's relations on $S$ using the rbliw-graphs $\A_s$.

\begin{prop}\label{Gclas}
If $s,t\in S$, then:
\begin{itemize}
\item[$(i)$] $s\Rc t$ \iff\ $\A_s$ and $\A_t$ are left isomorphic.
\item[$(ii)$] $s\Lc t$ \iff\ $\A_s$ and $\A_t$ are right isomorphic.
\item[$(iii)$] $s\Hc t$ \iff\ $\A_s$ and $\A_t$ are left and right isomorphic.
\item[$(iv)$] $s\Dc t$ \iff\ $\A_s$ and $\A_t$ are weakly isomorphic.
\item[$(v)$] $s\Jc t$ \iff\ there exist weak homomorphisms $\varphi:\A_s\to\A_t$ and $\psi:\A_t\to\A_s$.
\end{itemize}
\end{prop}

\begin{proof}
$(ii)$ is the dual of $(i)$,  $(iii)$ follows from $(i)$ and $(ii)$ together, and $(iv)$ follows from Proposition \ref{car_D_liw}. Next, we prove $(i)$. So, let $s\Rc t$ and let $e\in\Rcc_s\cap E(S)$. Then $\A_s=(\lv_e,\Ga_e,\rv_s)$ and $\A_t=(\lv_e,\Ga_e,\rv_t)$, and $\A_s$ and $\A_t$ are clearly left isomorphic. Conversely, let $\varphi:\A_s\to\A_t$ be a left isomorphism and let $e\in\Rcc_s\cap E(S)$ and $f\in\Rcc_t\cap E(S)$. Then $\A_s=(\lv_e,\Ga_e,\rv_s)$ and $\A_t=(\lv_f,\Ga_f,\rv_t)$. Note that $\varphi:\A_s\to (\lv_f,\Ga_f,\rv_{t_1})$ is an isomorphism for $t_1\in\Rcc_t$ such that $\rv_s\varphi=\rv_{t_1}$. By proposition \ref{iso_bliw} we must have $s=es=ft_1=t_1\Rc t$.

Finally, let us prove $(v)$. We start by assuming there are weak homomorphisms $\varphi:\A_s\to\A_t$ and $\psi:\A_t\to\A_s$. Let $e\in E(S)\cap\Rcc_s$, $f\in E(S)\cap\Rcc_t$ and set
$$\A_s=(\lv_e,\Ga_e,\rv_s)\qquad\mbox{ and }\qquad \A_t=(\lv_f,\Ga_f,\rv_t)\,.$$ 
Let $a,b\in\Rcc_f$ such that $\rv_a=\rv_e\varphi$ and $\rv_b=\rv_s\varphi$. Consider also $u\in\wX^+$ such that $u\mu=s$. Then $u\in L_{\rv_e,\rv_s}(\Ga_e)\subseteq L_{\rv_a,\rv_b}(\Ga_f)$. If $u_1\in L_{\rv_f,\rv_a}(\Ga_f)$ and $u_2\in L_{\rv_b,\rv_t}(\Ga_f)$, then $u_1uu_2\in L_{\rv_f,\rv_t}(\Ga_f)$ and 
$$t=f(u_1uu_2)=fu_1su_2\,.$$
Thus $t\in SsS$. Using $\psi$ now, we conclude by symmetry that $s\in StS$, and consequently $t\Jc s$.

Assume now that $s\Jc t$, and let $e\in E(S)\cap\Rcc_s$ and $f\in E(S)\cap\Rcc_t$. Thus $t=t_1st_2$ for some $t_1,t_2\in S$. Further, we can assume that $t_1\Rc t$, and so $t\Rc t_1s\Rc t_1e$. If $u\in L_{\rv_e,\rv_s}(\Ga_e)$, then $t_1eu=t_1s$. Hence
$$L_{\rv_e,\rv_s}(\Ga_e)\subseteq L_{\rv_{t_1e},\rv_{t_1s}}(\Ga_f)$$
and there is a homomorphism $\varphi:\Ga_e\to\Ga_f$ such that $\rv_e\varphi=\rv_{t_1e}$ and $\rv_s\varphi=\rv_{t_1s}$. Note now that $\varphi$ is a weak homomorphism from $\A_s=(\lv_e,\Ga_e,\rv_s)$ to $\A_t=(\lv_f,\Ga_f,\rv_t)$. By symmetry, there exists also a weak homomorphism $\psi:\A_t\to\A_s$.
\end{proof}

The next result is now a trivial consequence of the previous proposition using Corollary \ref{conv_liw}.

\begin{cor}\label{Gclas_lang}
Let $\A_s=(\av,\Ga,\bv)$ and $\A_t=(\av_1,\Ga_1,\bv_1)$ for $s,t\in S$. Then:
\begin{itemize}
\item[$(i)$] $s\Rc t$ \iff\ $L_{\av,\av}(\Ga)=L_{\av_1,\av_1}(\Ga_1)$.
\item[$(ii)$] $s\Lc t$ \iff\ $L_{\bv,\bv}(\Ga)=L_{\bv_1,\bv_1}(\Ga_1)$.
\item[$(iii)$] $s\Hc t$ \iff\ $L_{\av,\av}(\Ga)=L_{\av_1,\av_1}(\Ga_1)$ and $L_{\bv,\bv}(\Ga)=L_{\bv_1,\bv_1}(\Ga_1)$.
\item[$(iv)$] $s\Dc t$ \iff\ there is $\av'\in V_l(\Ga)$ such that $L_{\av',\av'}(\Ga)=L_{\av_1,\av_1}(\Ga_1)$.
\item[$(v)$] $s\Jc t$ \iff\ there are $\av'\in V_l(\Ga)$ and $\av_1'\in V_l(\Ga_1)$ such that $L_{\av',\av'}(\Ga)\subseteq L_{\av_1,\av_1}(\Ga_1)$ and $L_{\av_1',\av_1'}(\Ga_1)\subseteq L_{\av,\av}(\Ga)$.
\end{itemize}
\end{cor}

The idempotents and the inverses of elements of $S$ can also be characterized using the rbliw-graphs $\A_s$:

\begin{prop}\label{id_inv}
Let $s,t\in S$. Then:
\begin{itemize}
\item[$(i)$] $s\in E(S)$ \iff\ $(\lv(\A_s),\rv(\A_s))\in\ol{E}(\A_s)$, and \iff\ $x\wedge y\in L(\A_s)$ for some $x,y\in\oX$.
\item[$(ii)$] $t\in V(s)$ \iff\ there is a weak isomorphism $\varphi:\A_s\to\A_t$ such that $(\lv(\A_s)\varphi,\rv(\A_t))$ and $(\lv(\A_t),\rv(\A_s)\varphi)$ are edges of $\A_t$.
\end{itemize}
\end{prop}

\begin{proof}
$(i)$. If $s=e\in E(S)$, then $\A_s=(\lv_e,\Ga_e,\rv_e)$ and $(\lv_e,\rv_e)$ is an edge of $\Ga_e$ by construction of $\Ga_e$ and since $e$ is an inverse of itself. Conversely, let $(\lv(\A_s),\rv(\A_s))$ be an edge of $\A_s$, and let $e\in E(S)\cap\Rcc_s$. Then $\A_s=(\lv_e,\Ga_e,\rv_s)$, and $e$ is an inverse of $s$ since $(\lv_e,\rv_s)$ is an edge of $\Ga_e$. Hence $s^2=ses=s$ and $s\in E(S)$. It is obvious from the definition of $\w(\cdot)$ that $(\lv(\A_s),\rv(\A_s))\in\ol{E}(\A_s)$ \iff\ $x\wedge y\in L(\A_s)$ for some $x,y\in\oX$, namely $x\in\cb(\lv(\A_s))$ and $y\in\cb(\rv(\A_s))$.

$(ii)$. Let $t\in V(s)$. Then $e=st$ and $f=ts$ are idempotents and
$$\A_e=(\lv_e,\Ga_e,\rv_e),\;\;\A_s=(\lv_e,\Ga_e,\rv_s),\;\; \A_t=(\lv_t,\Ga_e,\rv_e)\;\mbox{ and }\; \A_f=(\lv_t,\Ga_e,\rv_s)\,.$$
Further $(\lv_e,\rv_e)$ and $(\lv_t,\rv_s)$ are edges of $\Ga_e$ because $e$ and $f$ are idempotents respectively. If $\varphi$ is the identity automorphism of $\Ga_e$, then $\varphi:\A_s\to\A_t$ is a weak isomorphism such that $(\lv(\A_s)\varphi,\rv(\A_t))$ and $(\lv(\A_t),\rv(\A_s)\varphi)$ are edges of $\Ga_e$, as desired.

Conversely, let $\varphi:\A_s\to\A_t$ be a weak isomorphism such that 
$$(\lv(\A_s)\varphi,\rv(\A_t))\,,\;(\lv(\A_t),\rv(\A_s)\varphi)\in\ol{E}(\A_t)\,.$$
Choose $e\in E(S)\cap\Rcc_t$ and set $\A_t=(\lv_e,\Ga_e,\rv_t)$. Let $a\in\Lcc_e$ and $b\in\Rcc_e$ such that $\lv_a=\lv(\A_s)\varphi$ and $\rv_b=\rv(\A_s)\varphi$. Then $\varphi:\A_s\to (\lv_a,\Ga_e,\rv_b)$ is an isomorphism, and $s=ab$ by Corollary \ref{ig_bliw}. Note that
$$\A_{at}=(\lv_a,\Ga_e,\rv_t)\qquad\mbox{ and }\qquad \A_{b}=(\lv_e,\Ga_e,\rv_b)\,.$$
By hypothesis, $(\lv_a,\rv_t)\in \ol{E}(\A_{at})$ and $(\lv_e,\rv_b)\in\ol{E}(\A_b)$, and so $at,b\in E(S)$ by $(i)$. Further 
$$s\Lc b\Rc t\Lc at\Rc s\quad\mbox{ and }\quad st=abt=at\,,$$
whence $t\in V(s)$.
\end{proof}

Let us look now to the relations $\omega$, $\wr$ and $\wl$ on $E(S)$.

\begin{prop}\label{wr_liw}
The following conditions are equivalent for $e,f\in E(S)$:
\begin{itemize}
\item[$(i)$]  $f\wr e$ {\normalfont [}$f\wl e${\normalfont ]};
\item[$(ii)$] There is a left {\normalfont [}right{\normalfont ]} homomorphism $\varphi:\A_e\to\A_f$;
\item[$(iii)$] $L_{\lv_e,\lv_e}(\Ga_e)\subseteq L_{\lv_f,\lv_f}(\Ga_f)\;$ {\normalfont [}\,$L_{\rv_e,\rv_e}(\Ga_e)\subseteq L_{\rv_f,\rv_f}(\Ga_f)\,${\normalfont ]}.
\end{itemize}
\end{prop}

\begin{proof}
We prove this result only for the $\wr$ case since the $\wl$ case is the dual. Further, the equivalence between $(ii)$ and $(iii)$ follows from Corollary \ref{conv_liw}. If $L_{\lv_e,\lv_e}(\Ga_e)\subseteq L_{\lv_f,\lv_f}(\Ga_f)$ and $u\mu=e$ for some $u\in\wX^+$, then $u\in L_{\lv_f,\lv_f}(\Ga_f)$ and $ef=uf=f$.  If $f\wr e$, then $vf=f$ whenever $ve=e$ for $v\in\wX^*$, and $L_{\lv_e,\lv_e}(\Ga_e)\subseteq L_{\lv_f,\lv_f}(\Ga_f)$. Thus $(i)$ and $(iii)$ are equivalent.
\end{proof}

\begin{cor}\label{w_liw}
The following conditions are equivalent for $e,f\in E(S)$:
\begin{itemize}
\item[$(i)$]  $f\,\omega\,e$;
\item[$(ii)$] There is a homomorphism $\varphi:\A_e\to\A_f$;
\item[$(iii)$] $L_{\lv_e,\rv_e}(\Ga_e)\subseteq L_{\lv_f,\rv_f}(\Ga_f)$.
\end{itemize}
\end{cor}

\begin{proof}
Once again, $(ii)$ and $(iii)$ are equivalent due to Corollary \ref{conv_liw}. Also $(ii)$ implies $(i)$ by Proposition \ref{wr_liw}, using both $\wr$ and $\wl$. Assume $f\,\omega\,e$ and let $\varphi:\Ga_e\to\Ga_f$ be the homomorphism such that $\lv_e\varphi=\lv_f$ (Proposition \ref{wr_liw} again). Let $b\in\Rcc_f$ and $x,y\in\oX$ such that $\rv_e\varphi=\rv_b$ and 
$$e\in ((x\wedge x')\mu]_r\cap ((y\wedge y')\mu]_l=((x\wedge y')\mu].$$
Then $x\wedge y'\in L_{\lv_e,\rv_e}(\Ga_e)\subseteq L_{\lv_f,\rv_b}(\Ga_f)$. Since $f(x\wedge y')=fe(x\wedge y')=f$, we must have $b=f$ as otherwise $\Ga_f$ would not be injective ($y'\in\cb(\rv_b)\cap\cb(\rv_f)$ and $(\lv_f,\rv_b)$ and $(\lv_f,\rv_f)$ are two edges of $\Ga_f$). Hence $\rv_e\varphi=\rv_f$ as wanted.
\end{proof}

Now that we have a characterization of the relations $\wr$, $\wl$ and $\omega$ on $E(S)$, let us describe the natural partial order and the quasiorders $\ler$, $\lel$, $\leh$ and $\leq_{\Jc}$ on $S$ using the rbliw-graphs $\A_s$.

\begin{prop}\label{Gdes}
If $s,t\in S$, then:
\begin{itemize}
\item[$(i)$] $t\leq s$ \iff\ there is a homomorphism $\varphi:\A_s\to\A_t$.
\item[$(ii)$] $t\ler s$ \iff\ there is a left homomorphism $\varphi:\A_s\to\A_t$.
\item[$(iii)$] $t\lel s$ \iff\ there is a right homomorphism $\varphi:\A_s\to\A_t$.
\item[$(iv)$] $t\leh s$ \iff\ there are a left homomorphism $\varphi:\A_s\to\A_t$ and a right homomorphism $\varphi':\A_s\to\A_t$.
\item[$(v)$] $t\leq_{\Jc} s$ \iff\ there is a weak homomorphism $\varphi:\A_s\to\A_t$.
\end{itemize}
\end{prop}

\begin{proof}
$(i)$. Assume first that $t\leq s$ and let $e\in\Rcc_s\cap E(S)$. It is well known that there exists an $f\in (e]\cap\Rcc_t$ such that $fs=t$. Then 
$$\A_s=(\lv_e,\Ga_e,\rv_s)\qquad\mbox{ and }\qquad \A_t=(\lv_f,\Ga_f,\rv_t)\,$$
and, by Corollary \ref{w_liw}, there is a homomorphism $\varphi:(\lv_e,\Ga_e,\rv_e)\to(\lv_f,\Ga_f, \rv_f)$. Let $t_1\in \Rcc_f$ such that $\rv_s\varphi=\rv_{t_1}$ and let $u\in\wX^+$ such that $u\mu=s$. Then
$$u\in L_{\rv_e,\rv_s}(\Ga_e)\subseteq L_{\rv_f,\rv_{t_1}}(\Ga_f)$$
and $t_1=fu=fs=t$. Hence, $\varphi$ is also a homomorphism from $\A_s$ to $\A_t$.

Assume now that $\varphi:\A_s\to\A_t$ is a homomorphism. Let $e\in\Rcc_s\cap E(S)$ and $f\in\Rcc_t\cap E(S)$. Then $\A_s=(\lv_e,\Ga_e,\rv_s)$ and $\A_t=(\lv_f,\Ga_f,\rv_t)$, and so $\lv_e\varphi=\lv_f$ and $\rv_s\varphi=\rv_t$. Let $g\in\Rcc_f$ such that $\rv_e\varphi=\rv_g$. Then $\varphi$ is a homomorphism from $\A_e=(\lv_e,\Ga_e,\rv_e)$ to $\A_g=(\lv_f,\Ga_f,\rv_g)$, and $g\in E(S)$ since $(\lv_f,\rv_g)=(\lv_e,\rv_e)\varphi_E\in\ol{E}(\Ga_f)$. If $u\mu=s$ for some $u\in\wX^+$, then $u\in L_{\rv_e,\rv_s}(\Ga_e)\subseteq L_{\rv_g,\rv_t}(\Ga_f)$ and $gs=gu=t$. Similarly, there is a $g_1\in E(S)\cap\Lcc_t$ such that $sg_1=t$. Hence $t\leq s$.

$(ii)$ follows from $(i)$ and Proposition \ref{Gclas}$.(i)$ using the following observation: 
$$t\ler s\quad\Leftrightarrow\quad t_1\leq s\mbox{ for some } t_1\Rc t.$$
$(iii)$ is the dual of $(ii)$, and $(iv)$ is just $(ii)$ and $(iii)$ put together. If we look carefully to the proof of Proposition \ref{Gclas}$.(v)$, we realize that we have proved, in fact, the statement $(v)$ of the present result (and used it for both $s\leq_{\Jc} t$ and $t\leq_{\Jc} s$).
\end{proof}

Once more, due to Corollary \ref{conv_liw}, the following result is trivial.

\begin{cor}\label{Gdes_lang}
Let $\A_s=(\av,\Ga,\bv)$ and $\A_t=(\av_1,\Ga_1,\bv_1)$ for $s,t\in S$. Then:
\begin{itemize}
\item[$(i)$] $t\leq s$ \iff\ $L(\A_s)\subseteq L(\A_t)$.
\item[$(ii)$] $t\ler s$ \iff\ $L_{\av,\av}(\Ga)\subseteq L_{\av_1,\av_1}(\Ga_1)$.
\item[$(iii)$] $t\lel s$ \iff\ $L_{\bv,\bv}(\Ga)\subseteq L_{\bv_1,\bv_1}(\Ga_1)$.
\item[$(iv)$] $t\leh s$ \iff\ $L_{\av,\av}(\Ga)\subseteq L_{\av_1,\av_1}(\Ga_1)$ and $L_{\bv,\bv}(\Ga)\subseteq L_{\bv_1,\bv_1}(\Ga_1)$.
\item[$(v)$] $t\leq_{\Jc} s$ \iff\ there is $\av_1'\in V_l(\Ga_1)$ such that $L_{\av,\av}(\Ga)\subseteq L_{\av_1',\av_1'}(\Ga_1)$.
\end{itemize}
\end{cor}

Before we continue, let us see how the theory developed so far applies to an example, the example of the four-spiral semigroup $Sp_4$.\vspace*{.3cm}

\noindent{\bf Example of the four-spiral semigroup}: As observed in Example 3, the four-spiral semigroup $Sp_4$ is given by the presentation 
$$Sp_4=LI\langle\{x\};\{(x',x'^2),(x,(x\wedge x)x)\}\rangle.$$ 
We begin constructing the rliw-graph $\Ga_{x'x}$. For that purpose, we represent the (unique) $\Dc$-class of $Sp_4$ in Figure \ref{SpD}, indicating also the elements of $\Rcc_{x'x}$ and of $\Lcc_{x'x}$. We abbreviate and write only $z$ for the element $x'(x\wedge x)$. An inspection of this figure immediately tells us that $V(x^i)=\{z^i,z^{i-1}x'\}$ and $V(x'x^{i+1})=\{z^i,z^ix'\}$ for $i\in\mathbb{N}$, where $z^0$ is the empty word, and $V(x'x)=\{x'x,x'\}$. Thus $\ol{E}(\Ga_{x'x})$ is the set
$$\{(\lv_{x^i},\rv_{z^i}),(\lv_{x^i},\rv_{z^{i-1}x'}),(\lv_{x'x^{i+1}},\rv_{z^i}),(\lv_{x'x^{i+1}},\rv_{z^ix'}),(\lv_{x'x},\rv_{x'x}),(\lv_{x'x},\rv_{x'})\}.$$
Further, since $zx=x'x$ and $z^{i+1}x=z^i$ for $i\in\mathbb{N}$, then $\vec{E}(\Ga_{x'x})$ is the set
$$\{(\lv_{x^{i+1}},x,\rv_{z^i}),(\lv_x,x,\rv_{x'x}),(\lv_{x'x^{i+1}},x',\rv_{z^ix'}),(\lv_{x'x},x',\rv_{x'})\}.$$
We depict the rliw-graph $\Ga_{x'x}$ in Figure \ref{Sp_xx}.

\begin{figure}[ht]
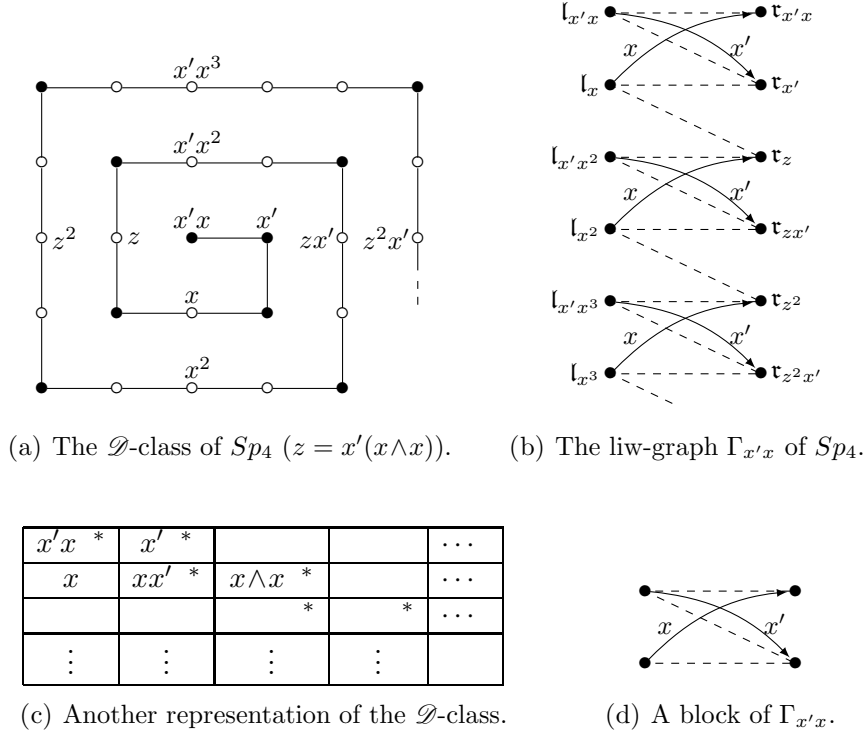

\begin{subfigure}[b]{.5\linewidth}
$$\tikz[scale=1, shorten <=2pt, shorten >=2pt, >=latex]{
\coordinate (1) at (0,1);
\coordinate (2) at (1,1);
\coordinate (3) at (1,0);
\coordinate (4) at (-1,0);
\coordinate (5) at (-1,2);
\coordinate (6) at (2,2);
\coordinate (7) at (2,-1);
\coordinate (8) at (-2,-1);
\draw (1) node {$\bullet$};
\draw (2) node {$\bullet$};
\draw (3) node {$\bullet$};
\draw (4) node {$\bullet$};
\draw (5) node {$\bullet$};
\draw (6) node {$\bullet$};
\draw (7) node {$\bullet$};
\draw (8) node {$\bullet$};
\draw (-2,3) node {$\bullet$};
\draw (3,3) node {$\bullet$};
\draw (1)--(2);
\draw (2)--(3);
\draw (3)--(0,0);
\draw (0,0)--(4);
\draw (4)--(-1,1);
\draw (-1,1)--(5);
\draw (5)--(0,2);
\draw (0,2)--(1,2);
\draw (1,2)--(6);
\draw (6)--(2,1);
\draw (2,1)--(2,0);
\draw (2,0)--(2,-1);
\draw (2,-1)--(7);
\draw (7)--(1,-1);
\draw (1,-1)--(0,-1);
\draw (0,-1)--(-1,-1);
\draw (-1,-1)--(8);
\draw (-2,0)--(8);
\draw (-2,0)--(-2,1);
\draw (-2,2)--(-2,1);
\draw (-2,2)--(-2,3);
\draw (-2,3)--(-1,3);
\draw (0,3)--(-1,3);
\draw (0,3)--(1,3);
\draw (2,3)--(1,3);
\draw (2,3)--(3,3);
\draw (3,2)--(3,3);
\draw (3,2)--(3,1);
\draw (3,.5)--(3,1);
\draw[dashed] (3,.5)--(3,0);
\draw (-2,2) node {$\circ$};
\draw (-1,3) node {$\circ$};
\draw (0,3) node {$\circ$};
\draw (1,3) node {$\circ$};
\draw (2,3) node {$\circ$};
\draw (3,2) node {$\circ$};
\draw (3,1) node {$\circ$};
\draw (-2,0) node {$\circ$};
\draw (-2,1) node {$\circ$};
\draw (-1,-1) node {$\circ$};
\draw (-1,1) node {$\circ$};
\draw (0,-1) node {$\circ$};
\draw (0,2) node {$\circ$};
\draw (1,-1) node {$\circ$};
\draw (1,2) node {$\circ$};
\draw (2,0) node {$\circ$};
\draw (2,1) node {$\circ$};
\draw (0,0) node {$\circ$};
\draw (1) node [above] {\small$x'x$};
\draw (0,2) node [above] {\hspace*{.15cm}\small$x'x^2$};
\draw (0,3) node [above] {\hspace*{.15cm}\small$x'x^3$};
\draw (2) node [above] {\small$x'$};
\draw (0,0) node [above] {\small$x$};
\draw (-1,1) node [right] {\small$z$};
\draw (-2,1) node [right] {\small$z^2$};
\draw (2,1) node [left=-1pt] {\small$zx'$};
\draw (3,1) node [left=-1pt] {\small$z^2x'$};
\draw (0,-1) node [above] {\hspace*{.15cm}\small$x^2$};
}$$
\caption{The $\Dc$-class of $Sp_4$ ($z=x'(x\!\wedge\! x)$).}\label{SpD}
\end{subfigure}
\begin{subfigure}[b]{.44\linewidth}
$$\tikz[scale=.8, shorten <=2pt, shorten >=3pt, >=latex]{
\coordinate[label=right:\small$\rv_{x'x}$] (1) at (2.5,7.2);
\coordinate[label=left:\small$\lv_{x'x}$] (2) at (0,7.2);
\coordinate[label=right:\small$\rv_{x'}$] (3) at (2.5,6);
\coordinate[label=left:\small$\lv_{x}$] (4) at (0,6);
\coordinate[label=right:\small$\rv_{z}$] (5) at (2.5,4.8);
\coordinate[label=left:\small$\lv_{x'x^2}$] (6) at (0,4.8);
\coordinate[label=right:\small$\rv_{zx'}$] (7) at (2.5,3.6);
\coordinate[label=left:\small$\lv_{x^2}$] (8) at (0,3.6);
\coordinate[label=right:\small$\rv_{z^2}$] (9) at (2.5,2.4);
\coordinate[label=left:\small$\lv_{x'x^3}$] (10) at (0,2.4);
\coordinate[label=right:\small$\rv_{z^2x'}$] (11) at (2.5,1.2);
\coordinate[label=left:\small$\lv_{x^3}$] (12) at (0,1.2);
\draw (1) node {$\bullet$};
\draw (2) node {$\bullet$};
\draw (3) node {$\bullet$};
\draw (4) node {$\bullet$};
\draw (5) node {$\bullet$};
\draw (6) node {$\bullet$};
\draw (7) node {$\bullet$};
\draw (8) node {$\bullet$};
\draw (9) node {$\bullet$};
\draw (10) node {$\bullet$};
\draw (11) node {$\bullet$};
\draw (12) node {$\bullet$};
\draw[dashed] (1)--(2)--(3)--(4)--(5)--(6)--(7)--(8)--(9)--(10)--(11)--(12)--(1.25,.6);
\draw[->, bend left=20] (2) to node [above, pos=.85] {\small$x'$} (3);
\draw[->, bend left=20] (4) to node [above, pos=.15] {\small$x$} (1);
\draw[->, bend left=20] (6) to node [above, pos=.85] {\small$x'$} (7);
\draw[->, bend left=20] (8) to node [above, pos=.15] {\small$x$} (5);
\draw[->, bend left=20] (10) to node [above, pos=.85] {\small$x'$} (11);
\draw[->, bend left=20] (12) to node [above, pos=.15] {\small$x$} (9);
}$$
\caption{The liw-graph $\Ga_{x'x}$ of $Sp_4$.}\label{Sp_xx}
\end{subfigure}\vspace*{.8cm}

\begin{subfigure}[b]{.57\linewidth}
\centering
\begin{tabular}{|c|c|c|c|c|}
\hline 
$x'x\;\;^*$ & $x'\;\;^*$ & & &$\cdots\hspace*{.2cm}$ \\ \hline
$x$& $xx'\;\;^*$ & $x\!\wedge\!x\;\;^*$ & &$\cdots\hspace*{.2cm}$  \\ \hline
&&\hspace*{1cm}$^*$& \hspace*{.8cm}$^*$&$\cdots\hspace*{.2cm}$ \\ \hline
$\vdots$ &$\vdots$ &$\vdots$ &$\vdots$ & \\ \hline
\end{tabular}
\caption{Another representation of the $\Dc$-class.}\label{Sp_stair}
\end{subfigure}
\begin{subfigure}[b]{.37\linewidth}
$$\tikz[scale=.8, shorten <=2pt, shorten >=3pt, >=latex]{
\coordinate (1) at (2.5,1.2);
\coordinate (2) at (0,1.2);
\coordinate (3) at (2.5,0);
\coordinate (4) at (0,0);
\draw (1) node {$\bullet$};
\draw (2) node {$\bullet$};
\draw (3) node {$\bullet$};
\draw (4) node {$\bullet$};
\draw[dashed] (1)--(2)--(3)--(4);
\draw[->, bend left=20] (2) to node [above, pos=.85] {\small$x'$} (3);
\draw[->, bend left=20] (4) to node [above, pos=.15] {\small$x$} (1);
}$$
\caption{A block of $\Ga_{x'x}$.}\label{Sp_block}
\end{subfigure}
\caption{The four-spiral semigroup $Sp_4$.}\label{Sp4}
\end{figure}

Each of the left vertices $\lv_a$ of $\Ga_{x'x}$ represents one of the $\Rc$-classes of $Sp_4$, namely the $\Rc$-class $\Rcc_a$, while each of the right vertices $\rv_b$ of $\Ga_{x'x}$ represents one of the $\Lc$-classes of $Sp_4$, namely the $\Lc$-class $\Lcc_b$. Note that only the identity mapping is an automorphism of $\Ga_{x'x}$ since $\rv_{x'x}$ is the only vertex of degree 2 (number of edges). This observation corroborates the fact that the $\Hc$-classes of $Sp_4$ are trivial.

The lines of $\Ga_{x'x}$ also form an infinite descending `zigzag' path starting at $\rv_{x'x}$ and passing through all the vertices. This reflects the fact that we can reorder the rows and the columns of the egg-box picture of $Sp_4$ so that the idempotents form an infinite descending `stair' (see Figure \ref{Sp_stair}).

Although $\Ga_{x'x}$ has no non-trivial automorphism, it has plenty of endomorphisms. Note that $\Ga_{x'x}$ is an infinite descending chain of copies of the `block' depicted in Figure \ref{Sp_block}. The endomorphisms of $\Ga_{x'x}$ are then obtained by `sliding down' these blocks. Thus, the set of endomorphisms of $\Ga_{x'x}$, with the usual composition of mappings, is isomorphic to the free monogenic monoid. By Proposition \ref{Gdes}, for each $a\in Sp_4$, $(a]_{\leq}$ is an infinite descending chain for the natural partial order $\leq$. This is a well-known fact about $Sp_4$ which was already expressed in Figure \ref{F4} by its arrows. \hfill\qed\vspace*{.3cm}

Finally, let us return to the general case of a locally inverse semigroup $S$ given by a presentation and pay some attention to the two operations defined on $S$, namely the semigroup product and the sandwich operation. We begin by analyzing the sandwich operation $\wedge$.

Given two bliw-graphs $\A_1=(\av_1,\Ga_1,\bv_1)$ and $\A_2=(\av_2,\Ga_2,\bv_2)$, set
$$\A_1\owedge \A_2=(\av_1,\Ga_1\cup\{(\av_1,\bv_2)\}\cup\Ga_2,\bv_2)\,,$$ 
that is, take the union of $\Ga_1$ and $\Ga_2$, add the edge $(\lv(\A_1),\rv(\A_2))$ for the new graph to become connected, and keep the left root of $\Ga_1$ and the right root of $\Ga_2$. Clearly $\A_1\owedge\A_2$ is a bliw-graph too. If $\A_1$ and $\A_2$ are reduced, set also $\A_1\wedge\A_2$ as the reduced form of $\A_1\owedge\A_2$. Let $\phi:\A_1\owedge\A_2\to\A_1\wedge\A_2$ be the natural $E$-surjective epimorphism. Although $\A_1$ and $\A_2$ are reduced, we cannot guarantee that the left homomorphism $\phi_{|\A_1}:\A_1\to\A_1\wedge\A_2$ and the right homomorphism $\phi_{|\A_2}:\A_2\to\A_1\wedge\A_2$ are monomorphims. The complete reduction of $\A_1\owedge\A_2$ into $\A_1\wedge\A_2$ may identify some vertices (and edges) of $\A_1$ or of $\A_2$ even when $\A_1$ and $\A_2$ are reduced. We will write only $\phi$ also to refer to both restrictions $\phi_{|\A_1}$ and $\phi_{|\A_2}$.

Any homomorphism $\ol{\psi}:\A_1\owedge\A_2\to\B$ is a combination of a left homomorphism $\ol{\psi}_1=\ol{\psi}_{|\A_1}:\A_1\to\B$ with a right homomorphism $\ol{\psi}_2=\ol{\psi}_{|\A_2}:\A_2\to\B$ that satisfy
$$(\lv(\A_1)\ol{\psi}_1,\rv(\A_2)\ol{\psi}_2)\in \ol{E}(\B)\,.$$
Once again, if no ambiguity occurs, we will write $\ol{\psi}$ also to refer to both $\ol{\psi}_1$ and $\ol{\psi}_2$. If $\B$ is reduced, then $\ol{\psi}$ induces a (unique) homomorphism $\psi:\A_1\wedge\A_2\to\B$ such that $\ol{\psi}=\phi\psi$ by Proposition \ref{hom_red}.

Conversely, if $\psi:\A_1\wedge\A_2\to \B$ is a homomorphism, then we denote by $\ol{\psi}$ the homomorphism $\phi\psi:\A_1\owedge\A_2\to \B$.
Hence, we can look at $\psi$ as a combination of a left homomorphism $\psi_1=\psi_{|\A_1\phi}$ with a right homomorphism $\psi_2=\psi_{|\A_2\phi}$ that coincide on $\A_1\phi\cap\A_2\phi$ and satisfy 
$$(\lv(\A_1)\phi\psi_1,\rv(\A_2)\phi\psi_2)\in\ol{E}(\B)\,.$$ 

Let $s,t\in S$. Although $\A_s\wedge\A_t$ and $\A_{s\wedge t}$ are not isomorphic in general, there is a sort of `Universal Property' between them that gives us some insight into the operation $\wedge$ on $S$. 

\begin{prop}\label{wedge_car}
Let $s,t,a,b\in S$. There is a homomorphism $\varphi:\A_s\wedge\A_t\to\A_a$ \iff\  $a\in (s\wedge t]$. In particular, $a\in E(S)$ and there is a homomorphism $\hat{\varphi}:\A_s\wedge\A_t\to\A_{s\wedge t}$. Further, for any homomorphism $\varphi:\A_s\wedge\A_t\to\A_b$, there exists a unique homomorphism $\psi:\A_{s\wedge t}\to\A_b$ such that $\varphi=\hat{\varphi}\psi$.
\end{prop}

\begin{proof}
Assume first there is a homomorphism $\varphi:\A_s\wedge\A_t\to\A_a$. Note that $(\lv(\A_a),\rv(\A_a))\in\ol{E}(\A_a)$ because $\lv(\A_a)=\lv(\A_s)\ol{\varphi}$, $\rv(\A_a)=\rv(\A_t)\ol{\varphi}$ and $(\lv(\A_s)\ol{\varphi},\rv(\A_t)\ol{\varphi})\in\ol{E}(\A_a)$. By Proposition \ref{id_inv}$.(i)$ we must have $a\in E(S)$. Further, by Proposition \ref{wr_liw}, $a\in (ss']_r$ for any $s'\in V(s)$ since $\ol{\varphi}_{|\A_s}:\A_s\to\A_a$ is a left homomorphism. Similarly, $a\in (t't]_l$ for any $t'\in V(t)$, and so $a\in (s\wedge t]$ because $S$ is a locally inverse semigroup.

Assume now that $a\in (s\wedge t]$. Then $a\in (ss']_r$ for any $s'\in V(s)$ and there is a left homomorphism $\varphi_1:\A_s\to\A_a$ again by Proposition \ref{wr_liw}. Similarly, there is also a right homomorphism $\varphi_2:\A_t\to\A_a$. Note now that
$$(\lv(\A_s)\varphi_1,\rv(\A_t)\varphi_2)=(\lv(\A_a),\rv(\A_a))\in\ol{E}(\A_a)$$
because $a$ is an idempotent. Hence, $\varphi_1$ and $\varphi_2$ combined give us a homomorphism $\ol{\varphi}:\A_s\owedge\A_t\to \A_a$. Using Proposition \ref{hom_red}, we obtain a homomorphism $\varphi:\A_s\wedge\A_t\to\A_a$ such that $\ol{\varphi}=\phi\varphi$. 

In particular, for $a=s\wedge t$, we conclude there is a homomorphism $\hat{\varphi}:\A_s\wedge\A_t\to \A_{s\wedge t}$. If $\varphi:\A_s\wedge\A_t\to\A_b$ is another homomorphism, then $b\,\omega\,(s\wedge t)$ and there exists a unique homomorphism $\psi:\A_{s\wedge t}\to\A_b$ by Corollaries \ref{unique_hom} and \ref{w_liw}. Clearly $\varphi=\hat{\varphi}\psi$ again by Corollary \ref{unique_hom} applied to the homomorphisms from $\A_s\wedge\A_t$ to $\A_b$. 
\end{proof}

As for the operation $\wedge$ on $S$, we will introduce an operation $\cdot$ between rbliw-graphs that is related with the product of $S$. The rbliw-graphs $\A_s\cdot\A_t$ and $\A_{st}$ will not be isomorphic for $s,t\in S$ in general. However, the analogue of Proposition \ref{wedge_car} for the operation $\cdot$ can be shown. We split it into the next three results, but first let us introduce the operation $\cdot$ on rbliw-graphs.

Let $\A_1=(\av_1,\Ga_1,\bv_1)$ and $\A_2=(\av_2,\Ga_2,\bv_2)$ be two bliw-graphs, and set
$$\A_1\odot\A_2=(\av_1,\Ga_1\cup\{(\av_2,\bv_1)\}\cup\Ga_2,\bv_2)\,.$$ 
Then $\A_1\odot\A_2$ is a bliw-graph, and the difference between $\A_1\odot\A_2$ and $\A_1\owedge\A_2$ is only on the new edge added. We add $(\lv(\A_2),\rv(\A_1))$ for the $\A_1\odot\A_2$ case and $(\lv(\A_1),\rv(\A_2))$ for the $\A_1\owedge\A_2$ case. Therefore, by construction, the identity mapping from $\A_1\odot\A_2$ to $\A_2\owedge\A_1$ is a weak isomorphism: $\A_1\odot\A_2$ and $\A_2\owedge\A_1$ have the same underlying liw-graph but with a different choice of roots.

Let $\A_1\cdot\A_2$ be the reduced form of $\A_1\odot\A_2$ when $\A_1$ and $\A_2$ are reduced. Then $\A_1\cdot\A_2$ and $\A_2\wedge\A_1$ have the same underlying rliw-graph and the identity mapping between them is a weak isomorphism. Further, the natural $E$-surjective epimorphism $\phi:\A_2\owedge\A_1\to\A_2\wedge\A_1$ is also the natural $E$-surjective epimorphism from $\A_1\odot\A_2$ to $\A_1\cdot\A_2$. Similar to the case of the operations $\owedge$ and $\wedge$:
\begin{itemize}
\item[$(i)$] Any homomorphism $\ol{\psi}:\A_1\odot\A_2\to\B$ is a combination of a left homomorphism $\ol{\psi}_1=\ol{\psi}_{|\A_1}:\A_1\to\B$ with a right homomorphism $\ol{\psi}_2=\ol{\psi}_{|\A_2}:\A_2\to\B$ that satisfy
$$(\lv(\A_2)\ol{\psi}_2,\rv(\A_1)\ol{\psi}_1)\in \ol{E}(\B)\,.$$
\item[$(ii)$] If $\B$ is reduced, then any homomorphism $\ol{\psi}:\A_1\odot\A_2\to \B$ induces a unique homomorphism $\psi:\A_1\cdot\A_2\to \B$ such that $\ol{\psi}=\phi\psi$.
\item[$(iii)$] If $\psi:\A_1\cdot\A_2\to \B$ is a homomorphism, then $\ol{\psi}$ denotes the homomorphism $\phi\psi$ from $\A_1\odot\A_2$ to $\B$.
\end{itemize} 

It is well known that $\,st\Rc s(t\wedge s)\Lc (t\wedge s)\Rc (t\wedge s)t\Lc st$ and $st=s(t\wedge s)t$ for any $s,t\in S$. Hence
$$\A_{t\wedge s}=(\lv_{(t\wedge s)},\Ga_{t\wedge s},\rv_{(t\wedge s)}) \quad\mbox{ and }\quad \A_{st}=(\lv_{s(t\wedge s)},\Ga_{t\wedge s},\rv_{(t\wedge s)t})\,.$$
Thus, the identity mapping on $\Ga_{t\wedge s}$ induces a weak isomorphism from $\A_{st}$ to $\A_{t\wedge s}$. It is now obvious that the homomorphism $\hat{\varphi}:\A_t\wedge\A_s\to\A_{t\wedge s}$ given by Proposition \ref{wedge_car} is also a weak homomorphism from $\A_s\cdot\A_t$ to $\A_{st}$. In the next result we prove that $\hat{\varphi}:\A_s\cdot\A_t\to \A_{st}$ is, in fact, a homomorphism.

\begin{prop}\label{dot_hom}
Let $s,t\in S$. Then $\hat{\varphi}:\A_s\cdot\A_t\to\A_{st}$ is a homomorphism.
\end{prop}

\begin{proof}
Choose idempotents $e$ and $f$ such that $e\Lc s$ and $f\Rc t$. Then 
$\A_s=(\lv_s,\Ga_e,\rv_e)$ and $\A_t=(\lv_f,\Ga_f,\rv_t)$. Further, if $\Ga'=\Ga_e\cup\{(\lv_f,\rv_e)\}\cup\Ga_f$, then $\A_s\odot\A_t=(\lv_s,\Ga',\rv_t)$ and $\A_t\owedge\A_s=(\lv_f,\Ga',\rv_e)$. Let $u,v\in S$ and $x,y,x_1,y_1\in\oX$ such that $s=u\mu$, $t=v\mu$, $f\in ((x\wedge y)\mu]$ and $e\in ((x_1\wedge y_1)\mu]$. Note that $s=(u(x_1\wedge y_1))\mu$ and $t=((x\wedge y)v)\mu$ since $s\in S(x_1\wedge y_1)$ and $t\in (x\wedge y)S$. By Lemmas \ref{lem62} and \ref{lem63},
$$v\in  L_{\rv_f,\rv_t}(\Ga_f)\subseteq L_{\rv_f,\rv_t}(\Ga')\qquad\mbox{ and }\qquad u\in L_{\lv_s,\lv_e}(\Ga_e)\subseteq L_{\lv_s,\lv_e}(\Ga')\,,$$ 
and so
$$(x\wedge y)v\in L_{\rv_e,\rv_t}(\Ga')\qquad\mbox{ and }\qquad u(x_1\wedge y_1)\in L_{\lv_s,\lv_f}(\Ga')\,.$$
If $\Ga$ is the reduced form of $\Ga'$, then
$\A_s\cdot\A_t=(\av_1,\Ga,\bv_1)$ and $\A_t\wedge\A_s=(\av_2,\Ga,\bv_2)$
for $\av_1=\lv_s\phi$, $\bv_1=\rv_t\phi$, $\av_2=\lv_f\phi$ and $\bv_2=\rv_e\phi$. Thus 
$$(x\wedge y)v\in L_{\bv_2,\bv_1}(\Ga) \qquad\mbox{ and }\qquad u(x_1\wedge y_1)\in L_{\av_1,\av_2}(\Ga)\,.$$

Set $\A_{t\wedge s}=(\lv_{t\wedge s},\Ga_{t\wedge s},\rv_{t\wedge s})$. Then $\lv_{t\wedge s}=\av_2\hat{\varphi}$ and $\rv_{t\wedge s}=\bv_2\hat{\varphi}$. Let $a\in\Lcc_{t\wedge s}$ and $b\in\Rcc_{t\wedge s}$ such that $\lv_a=\av_1\hat{\varphi}$ and $\rv_b=\bv_1\hat{\varphi}$. Then $\hat{\varphi}$ is a homomorphism from $\A_s\cdot\A_t$ to $(\lv_a,\Ga_{t\wedge s},\rv_b)$. Further $(x\wedge y)v\in L_{\rv_{t\wedge s},\rv_b}(\Ga_{t\wedge s})$ and $u(x_1\wedge y_1)\in L_{\lv_a,\lv_{t\wedge s}}(\Ga_{t\wedge s})$. Hence 
$$b=(t\wedge s)(x\wedge y)v=(t\wedge s)t\quad\mbox{ and }\quad a=u(x_1\wedge y_1)(t\wedge s)=s(t\wedge s)\,.$$
Thus $ab=s(t\wedge s)t=st$ and $\A_{st}=(\lv_a,\Ga_{t\wedge s},\rv_b)$ by Proposition \ref{iso_bliw}. We have shown that $\hat{\varphi}$ is a homomorphism from $\A_s\cdot\A_t$ to $\A_{st}$.
\end{proof}

\begin{prop}\label{dot_uni}
Let $\varphi:\A_s\cdot\A_t\to\A_a$ be a homomorphism for $s,t,a\in S$. Then there exists a unique homomorphism $\psi:\A_{st}\to\A_a$ such that $\varphi=\hat{\varphi}\psi$.
\end{prop}

\begin{proof}
By Corollary \ref{unique_hom} it is enough to show the existence of $\psi$.  Let $\A_s=(\lv_s,\Ga_e,\rv_e)$ and $\A_t=(\lv_f,\Ga_f,\rv_t)$ for idempotents $e\in\Lcc_s$ and $f\in\Rcc_t$. Let $\A_a=(\av,\Ga,\bv)$. Thus $\av=\lv_s\ol{\varphi}$ and $\bv=\rv_t\ol{\varphi}$. If $\av_1=\lv_f\ol{\varphi}$ and $\bv_1=\rv_e\ol{\varphi}$, then $(\av_1,\Ga,\bv_1)=\A_b$ for some $b\in S$ such that $a\Dc b$. Further, $\varphi$ is a homomorphism from $\A_t\wedge\A_s$ to $\A_b$. Let $\psi$ be the homomorphism from $\A_{t\wedge s}$ to $\A_b$ such that $\varphi=\hat{\varphi}\psi$ given by Proposition \ref{wedge_car}. Since $\lv(\A_{st})=(\lv(\A_s\cdot\A_t))\hat{\varphi}=\lv_s\ol{\hat{\varphi}}$ and $\rv(\A_{st})=(\rv(\A_s\cdot\A_t))\hat{\varphi}=\rv_t\ol{\hat{\varphi}}$, we have
$$(\lv(\A_{st})\psi=(\lv_s)\ol{\hat{\varphi}}\psi=\lv_s\ol{\varphi}=\av\quad \mbox{ and }\quad (\rv(\A_{st})\psi=(\rv_t)\ol{\hat{\varphi}}\psi=(\rv_t)\ol{\varphi}=\bv\,.$$
Thus $\psi$ is also a homomorphism from $\A_{st}$ to $\A_a$, and we have proved this proposition.
\end{proof}

\begin{cor}\label{dot_car}
For $s,t,a\in S$, we have $a\leq st$ \iff\ there is a homomorphism $\varphi:\A_s\cdot\A_t\to\A_a$. 
\end{cor}

\begin{proof}
By Proposition \ref{Gdes}$.(i)$, $a\leq st$ \iff\ there is a homomorphism $\psi:\A_{st}\to\A_a$. By the last two results, there is a homomorphism $\psi:\A_{st}\to\A_a$ \iff\ there is a homomorphism $\varphi:\A_s\cdot\A_t\to\A_a$.
\end{proof}

The following corollary deals with a special product case.

\begin{cor}\label{dot_cor}
Let $s,t\in S$. 
\begin{itemize}
\item[$(i)$] $st=t$ \iff\ there exists a left homomorphism $\varphi:\A_s\to\A_t$ such that $(\lv(\A_t),\rv(\A_s)\varphi)\in\ol{E}(\A_t)$.
\item[$(ii)$] $st=s$ \iff\ there exists a right homomorphism $\varphi:\A_t\to\A_s$ such that $(\lv(\A_t)\varphi,\rv(\A_s))\in\ol{E}(\A_s)$.
\end{itemize}
\end{cor}

\begin{proof}
$(i)$. If $st=t$, then there is a homomorphism $\psi:\A_s\cdot\A_t\to\A_t$
by Corollary \ref{dot_car}. Note further that $\varphi=\ol{\psi}_{|\A_s}$ is a left homomorphism while $\ol{\psi}_{|\A_t}$ is the identity mapping. In particular, 
$$(\lv(\A_t),\rv(\A_s)\varphi)=(\lv(\A_t)\ol{\psi},\rv(\A_s)\ol{\psi})=(\lv(\A_t),\rv(\A_s))\ol{\psi}\in\ol{E}(\A_t)\,.$$

Assume now that there is a left homomorphism $\varphi:\A_s\to\A_t$ such that $(\lv(\A_t),\rv(\A_s)\varphi)\in\ol{E}(\A_t)$, and define the mapping $\ol{\varphi_1}:\A_s\odot\A_t\to\A_t$ by setting $\ol{\varphi_1}_{|\A_s}=\varphi$ and letting $\ol{\varphi_1}_{|\A_t}$ be the identity mapping. Clearly, $\ol{\varphi_1}$ is a homomorphism since $(\lv(\A_t),\rv(\A_s)\varphi)\in\ol{E}(\A_t)$, whence $\varphi_1:\A_s\cdot\A_t\to\A_t$ is a homomorphism too. Let $\psi$ be the right homomorphism $\phi_{|\A_t}\hat{\varphi}:\A_t\to\A_{st}$. Let also $\psi':\A_{st}\to\A_t$ be the homomorphism such that $\varphi_1=\hat{\varphi}\psi'$ given by Proposition \ref{dot_uni}. Then
$$\psi\psi':\A_t\to\A_t \qquad\,\mbox{ and }\qquad \psi'\psi:\A_{st}\to\A_{st}$$ are right homomorphisms too. By Corollary \ref{unique_hom}, $\psi\psi'$ and $\psi'\psi$ must be the identity automorphisms of $\A_t$ and $\A_{st}$, respectively. Since $\psi'$ is a homomorphism, we conclude that $\psi'$ is an isomorphism with inverse isomorphism $\psi$. Finally, by Corollary \ref{ig_bliw}, we must have $st=t$.

The proof of $(ii)$ is similar.
\end{proof}

We end this section with a remark. Let $S_1$ be a locally inverse semigroup and let $X$ be a subset of $S_1$. For each $x\in X$ choose an inverse $x'\in V(x)$. It may happen that we choose the same inverse for two distinct elements of $X$.  Let $X'$ be the multiset (it may contain multiple copies of the same element) of all those inverses such that $x\to x'$ becomes a bijection from $X$ to $X'$. Assume that $S_1$ is generated by $X\cup X'$ as a type $\langle 2,2\rangle$ algebra. Consider now $X$ and $X'$ as sets of formal letters. Then $S_1$ is a homomorphic image of $\wX^+$ and there is a presentation $P=\langle X; R\rangle$ such that $S_1=\LI\langle X;R\rangle$. Thus, we can use the techniques developed in this paper and analyze the structure of $S_1$ using rbliw-graphs whose arrows are labeled by letters from $X\cup X'$.

\section{Some special classes of locally inverse semigroups}

In this last section, we study the structure of the rbliw-graphs for special classes of locally inverse semigroups. Some of these classes include normal bands, normal bands of groups, left [right] generalized inverse semigroups, generalized inverse semigroups, $E$-solid locally inverse semigroups, strict regular semigroups and idempotent generated locally inverse semigroups.

We begin by relating some $\Dc$-class structural properties with structural properties about rbliw-graphs. Let $\ol{\Ga}$ be the graph obtained from an liw-graph $\Ga$ by deleting all arrows. By Proposition \ref{id_inv}, the elements of a $\Dc$-class $\Dcc_e$ of $S$ are all idempotents \iff\ the bipartite graph $\ol{\Ga}_e$ is complete. Thus we have the following trivial result:

\begin{prop}
Let $S$ be a locally inverse semigroup $S$ and $e\in E(S)$. Then $\Dcc_e$ is a rectangular band \iff\ $\ol{\Ga}_e$ is a complete bipartite graph. Thus, $S$ is a normal band \iff\ $\ol{\Ga}_e$ is a complete bipartite graph for all $e\in E(S)$.
\end{prop}

We denote by $d_l(\av)$ the number of lines of $\Ga$ with $\av$ as one of its endpoints. Note that if $a\Lc e\Rc b$ for $e\in E(S)$, then $d_l(\lv_a)$ and $d_l(\rv_b)$ give the number of idempotents in $\Rcc_a$ and $\Lcc_b$ respectively. The following proposition and its corollary are now obvious.

\begin{prop}
Let $D$ be a $\Dc$-class of $S$ and let $e\in D$ be an idempotent. Then each $\Rc$-class of $D$ has only one idempotent \iff\ $d_l(\av)=1$ for all $\av\in V_l(\Ga_e)$; and each $\Lc$-class of $D$ has only one idempotent \iff\ $d_l(\av)=1$ for all $\av\in V_r(\Ga_e)$.
\end{prop}

A \emph{left [right] generalized inverse semigroup} is a regular semigroup whose idempotents constitute a left [right] normal band. They can be described also as the locally inverse semigroups whose $\Rc$-classes [$\Lc$-classes] have only one idempotent. Thus, inverse semigroups are the semigroups which are simultaneously left and right generalized inverse semigroups. A \emph{generalized inverse semigroup} is a regular semigroup whose idempotents constitute a normal band. These latter semigroups cannot be described by conditions on the number of idempotents in their $\Rc$-classes and/or $\Lc$-classes.

\begin{cor}
The locally inverse semigroup $S$ is
\begin{itemize}
\item[$(i)$] a left generalized inverse semigroup \iff\ $d_l(\av)=1$ for all $e\in S$ and $\av\in V_l(\Ga_e)$;
\item[$(ii)$] a right generalized inverse semigroup \iff\ $d_l(\av)=1$ for all $e\in S$ and $\av\in V_r(\Ga_e)$;
\item[$(iii)$] an inverse semigroup \iff\ $d_l(\av)=1$ for all $e\in S$ and $\av\in V(\Ga_e)$.
\end{itemize}
\end{cor}

Let $\Ga$ be an liw-graph. The graph $\ol{\Ga}$ is not connected in general. The maximal connected subgraphs of $\ol{\Ga}$ are called the \emph{line-connected components} of $\Ga$. Two vertices of $\Ga$ are \emph{line-connected} if they belong to the same line-connected component. We denote by $C(\Ga)$ the set of all line-connected components of $\Ga$, and by $\Omega_\Ga(\av)$ (or simply by $\Omega(\av)$ if no ambiguity occurs) the line-connected component of $\Ga$ containing $\av\in V(\Ga)$. Thus $S$ is a left [right] generalized inverse semigroup \iff\ $\Omega(\av)$ has only one right [left] vertex, for all $e\in E(S)$ and $\av\in\Ga_e$; and $S$ is an inverse semigroup \iff\ $\Omega(\av)$ is the trivial connected bipartite graph with only two vertices and one line, for all $e\in E(S)$ and $\av\in\Ga_e$.

Before we characterize the generalized inverse semigroups using the line-connected components, we should remark the following about inverse semigroups. If $S$ is, in fact, an inverse semigroup and $s\in S$, then we can represent $\A_s$ by $\A_s=(\lv_e,\Ga_e,\rv_s)$ where $e$ is the unique idempotent in the $\Rc$-class of $s\in S$. If we contract each line-connected component of $\Ga_e$ into a single vertex, we obtain precisely the \emph{Sch\"{u}tzenberger graph} of the $\Rc$-class $\Rcc_s$ (see \cite{stephen90} for the definition of these latter graphs). 

\begin{prop}\label{gen_inv}
The locally inverse semigroup $S$ is a generalized inverse semigroup \iff\ $\Omega(\av)$ is a complete bipartite graph for all $e\in E(S)$ and $\av\in\Ga_e$.
\end{prop}

\begin{proof}
Assume that $S$ is a generalized inverse semigroup and let $e\in E(S)$ and $\av\in\Ga_e$. Let $a_1,a_2\in \Lcc_e$ and $b_1,b_2\in\Rcc_e$ such that 
$$p\,:=\;\lv_{a_1}\,\ev_1\,\rv_{b_1}\,\fv_1\,\lv_{a_2}\,\ev_2\,\rv_{b_2}$$
is a path in $\Omega(\av)$. Then $a_1,a_2\in V(b_1)$ and $b_2\in V(a_2)$. Hence $a_1b_2=a_1eb_2=a_1b_1a_2b_2$ is the product of two idempotents, and so it is an idempotent since $E(S)$ is a normal band. Thus there is an edge with endpoints $\lv_{a_1}$ and $\rv_{b_2}$ in $\Ga_e$ by Proposition \ref{id_inv}. From the previous observation, it follows trivially by graph arguments that the bipartite graph $\Omega(\av)$ is complete.

Assume now that the bipartite graphs $\Omega(\av)$ are complete for all $e\in E(S)$ and $\av\in\Ga_e$. Let $e,f\in E(S)$. Then 
$$ef\;\Rc\; e_1\;\Lc\; (f\wedge e)\;\Rc\; f_1\;\Lc\; ef$$
for $e_1=e(f\wedge e)\in E(S)$ and $f_1=(f\wedge e)f\in E(S)$, and $ef=e_1f_1$. Consider the representation $(\lv_{e_1},\Ga_{f\wedge e},\rv_{f_1})$ of $\A_{ef}$ and note that $\lv_{e_1},\rv_{f_1}\in \Omega(\lv_{f\wedge e})$. Since $\Omega(\lv_{f\wedge e})$ is complete, there is an edge connecting $\lv_{e_1}$ and $\rv_{f_1}$ in $\Ga_{f\wedge e}$. Hence $ef$ is an idempotent. We have shown that $E(S)$ is a subsemigroup of $S$. It is well known that the set of idempotents of a locally inverse semigroup is a subsemigroup precisely when $S$ is a generalized inverse semigroup.
\end{proof}

The \emph{core} of a semigroup $S_1$ is the subsemigroup $C(S_1)$ generated by all idempotents. It is well known that $s\in C(S_1)$ \iff\ there are idempotents $e_1,\cdots,e_n,f_1,\cdots,f_{n-1}\in E(S_1)$ such that
$$s\Rc e_1\Lc f_1\Rc e_2\Lc\cdots \Lc f_{n-1}\Rc e_n\Lc s\;\;\mbox{ and }\;\; s=e_1f_1e_2\cdots f_{n-1}e_n\,.$$
Further, the core of a locally inverse [regular] semigroup is known to be locally inverse [regular]. The next result establishes the relationship between the core of $S$ and the line-connected components of $\Ga_e$ for $e\in E(S)$. 

\begin{prop}\label{core}
Let $(\lv_a,\Ga_e,\rv_b)$ be a representation of $\A_s$ for some $s\in S$. Then $s\in C(S)$ \iff\ $\lv_a$ and $\rv_b$ are line-connected in $\Ga_e$. 
\end{prop} 

\begin{proof}
Assume that $s\in C(S)$ and let $e_1,e_2, \cdots, e_n$ and $f_1,\cdots,f_{n-1}$ be idempotents of $S$ such that $s\Rc e_1\Lc f_1\Rc e_2\Lc\cdots\Lc f_{n-1}\Rc e_n\Lc s$ and $s=e_1f_1e_2\cdots f_{n-1}e_n\,$. Then there are $a_1,\cdots,a_n\in\Lcc_e$ and $b_1,\cdots,b_n\in\Rcc_e$ such that
$$a_1=a,\quad a_ib_i=e_i,\quad a_{i+1}b_i=f_i\quad\mbox{ and }\quad b_ia_i=e=b_ia_{i+1}.$$
Note that $(\lv_{a_i},\Ga_e,\rv_{b_i})$ and $(\lv_{a_{i+1}},\Ga_e,\rv_{b_i})$ are representations of $\A_{e_i}$ and $\A_{f_i}$, respectively, and that $\ev_{i}=(\lv_{a_i},\rv_{b_i})$ and $\fv_i=(\lv_{a_{i+1}},\rv_{b_i})$ are edges of $\Ga_e$. Further 
$$ab=s=e_1f_1e_2\cdots f_{n-1}e_n=a_1b_1a_2b_1a_2b_2\cdots a_nb_{n-1}a_nb_n=a_1eb_n=ab_n$$
and $b=b_n$. Hence, there is a walk in $\Ga_e$ from $\lv_a$ and $\rv_b$ with no arrows, and $\lv_a$ and $\rv_b$ belong to the same line-connected component of $\Ga_e$.

Assume now that $\lv_a$ and $\rv_b$ belong to the same line-connected component of $\Ga_e$. Let
$$p:=\;\lv_{a_1}\,\ev_1\,\rv_{b_1}\,\fv_1\,\lv_{a_2}\,\ev_2\,\rv_{b_2}\cdots \lv_{a_n}\,\ev_n\,\rv_{b_n}$$
be a walk from $\lv_a=\lv_{a_1}$ to $\rv_b=\rv_{b_n}$ with no arrows. Thus $\ev_1,\cdots,\ev_n$ and $\fv_1,\cdots,\fv_{n-1}$ are edges of $\Ga_e$, and $a_i,a_{i+1}\in V(b_i)$. Then $e_i=a_ib_i$ and $f_i=a_{i+1}b_i$ are idempotents, and $b_ia_i=e=b_ia_{i+1}$. Hence
$$e_1f_1e_2\cdots f_{n-1}e_n=a_1b_1a_2b_1a_2b_2\cdots a_nb_{n-1}a_nb_n=a_1eb_n=ab=s\,,$$
and $s\in C(S)$.
\end{proof}

A characterization for the idempotent generated locally inverse semigroups using liw-graphs is now obvious.

\begin{cor}\label{id_gen}
The locally inverse semigroup $S$ is idempotent generated \iff\ $\ol{\Ga}_e$ is connected for each $e\in E(S)$.
\end{cor}

An liw-graph $\Ga$ is \emph{line-transitive} if for each pair of lines $(\ev,\fv)$ of $\Ga$, there exists $\varphi\in \Aut(\Ga)$ such that $\ev\varphi=\fv$. Thus, any line-transitive liw-graph is also \emph{left vertex-transitive} [\emph{right vertex-transitive}], that is, for each pair $(\av,\bv)$ of left [right] vertices of $\Ga$, there exists $\varphi\in \Aut(\Ga)$ such that $\av\varphi=\bv$. Note also that if $\Ga$ is reduced, then any one of these three transitivity properties implies the \emph{arrow-transitivity} property, that is, for each pair $(\ev,\fv)$ of arrows of $\Ga$ with the same label, there exists $\varphi\in \Aut(\Ga)$ such that $\ev\varphi=\fv$.

The \emph{line-distance} between two vertices $\av$ and $\bv$ of $\Ga$ is their distance in $\ol{\Ga}$, that is, is the minimal length between all paths in $\ol{\Ga}$ connecting those vertices. If there is no path between $\av$ and $\bv$ in $\ol{\Ga}$, then their line-distance is infinite. Note that if $\Ga$ is reduced and the line-distance between $\av$ and $\bv$ is 2, then these vertices have the same side but there is no $\varphi\in\Aut(\Ga)$ such that $\av\varphi=\bv$. Hence, any left and right vertex-transitive rliw-graph has no adjacent lines. So, for rliw-graphs, edge-transitivity is equivalent to left and right vertex-transitivity put together.

\begin{prop}
Let $e\in E(S)$. Then
\begin{itemize}
\item[$(i)$] $\Dcc_e$ is a group \iff\ $\Ga_e$ is line-transitive.
\item[$(ii)$] $\Dcc_e$ is a left group (that is, a direct product of a group with a left zero semigroup) \iff\ $\Ga_e$ is right vertex-transitive.
\item[$(iii)$] $\Dcc_e$ is a right group (that is, a direct product of a group with a right zero semigroup) \iff\ $\Ga_e$ is left vertex-transitive.
\end{itemize}
\end{prop}

\begin{proof}
Note that $(ii)$ follows from Proposition \ref{Gclas}$.(ii)$ since $\Dcc_e$ is a left group \iff\ it has only one $\Lc$-class. Similarly, $(iii)$ follows from Proposition \ref{Gclas}$.(i)$. Finally, $(i)$ follows from $(ii)$ and $(iii)$ together. 
\end{proof}

\emph{Clifford semigroups}, also known as semilattices of groups, have many different characterizations. One of them is as regular semigroups with $\Hc=\Dc$. Thus each $\Dc$-class of a Clifford semigroup is a subgroup. Similarly, \emph{left} [\emph{right}] \emph{Clifford semigroups} can be defined as regular semigroups with $\Lc=\Dc$ [$\Rc=\Dc$]. The following result is now an obvious consequence of the previous proposition.

\begin{cor}
The locally inverse semigroup $S$ is 
\begin{itemize}
\item[$(i)$] a Clifford semigroup \iff\ $\Ga_e$ is line-transitive for all $e\in E(S)$.
\item[$(ii)$] a left [right] Clifford semigroup \iff\ $\Ga_e$ is right [left] vertex-transitive for all $e\in E(S)$.
\end{itemize}
\end{cor}

The analysis of the arrow-transitive case is also interesting. Next, we prove that all liw-graphs $\Ga_e$ of $S$ are arrow-transitive \iff\ $S$ is a \emph{strict regular semigroup}, that is, $S$ is a regular subsemigroup of direct products of completely simple and completely 0-simple semigroups. Alternatively, and also convenient to us, strict regular semigroups can be defined as regular semigroups whose local submonoids are Clifford semigroups. We begin with the following lemma:

\begin{lem}\label{atrans}
Let $e\in E(S)$. Then $\Ga_e$ is arrow-transitive \iff\ $\Dcc_e\cap xx'Sxx'$ is either the empty set or a subgroup of $S$ for each $x\in\oX$.  
\end{lem}

\begin{proof}
Assume first that $\Ga_e$ is arrow-transitive and that $\Dcc_e\cap xx'Sxx'\neq\emptyset$ for some $x\in\oX$. Let $s,t\in\Dcc_e\cap xx'Sxx'$. Consider two representations $(\lv_a,\Ga_e,\rv_b)$ and $(\lv_{a_1},\Ga_e,\rv_{b_1})$ of $\A_s$ and $\A_t$ respectively. Hence $a_1\Lc a\Lc e$ and $a_1,a\in xx'S$. By construction of $\Ga_e$, there must exist two arrows $\ev=(\lv_a,x,\rv_c)$ and $\fv=(\lv_{a_1},x,\rv_{c_1})$ in $\Ga_e$. Now, since $\Ga_e$ is arrow-transitive, there exists $\varphi\in\Aut(\Ga_e)$ such that $\ev\varphi=\fv$. Thus $\varphi$ is a left isomorphism from $\A_s$ onto $\A_t$, and $s\Rc t$ by Proposition \ref{Gclas}$.(i)$. Similarly, using the fact that $b_1\Rc b\Rc e$ and $b_1,b\in Sxx'$, we can show that $s\Lc t$. Consequently, $\Dcc_e\cap xx'Sxx'$ is precisely an $\Hc$-class of $\Dcc_e$. Finally, since $xx'Sxx'$ is an inverse submonoid of $S$, $\Dcc_e\cap xx'Sxx'$ must be a maximal subgroup of $S$.

Assume now that $\Dcc_e\cap xx'Sxx'$ is either the empty set or a subgroup of $S$ for each $x\in\oX$. Let $\ev=(\lv_a,x,\rv_b)$ and $\fv=(\lv_{a_1},x,\rv_{b_1})$ be two arrows of $\Ga_e$ labeled by $x\in\oX$. Then $b=a'x$ for some $a'\in V(a)\cap \Rcc_e\cap Sxx'$ and $b_1=a_1'x$ for some $a_1'\in V(a_1)\cap \Rcc_e\cap Sxx'$. So, $aa',a_1a_1'\in E(\Dcc_e)\cap xx'Sxx'$ and, consequently, $aa'=a_1a_1'$ because $\Dcc_e\cap xx'Sxx'$ is a subgroup by assumption. Hence $ab=a_1b_1$, and $(\lv_a,\Ga_e,\rv_b)$ is isomorphic to $(\lv_{a_1},\Ga_e,\rv_{b_1})$ by Proposition \ref{iso_bliw}. We have shown that there exist $\varphi\in\Aut(\Ga_e)$ such that $\ev\varphi=\fv$. Therefore $\Ga_e$ is arrow-transitive.
\end{proof}

\begin{prop}
The locally inverse semigroup $S$ is a strict regular semigroup \iff\ $\Ga_e$ is arrow-transitive for all $e\in E(S)$.
\end{prop}

\begin{proof}
Assume first that $S$ is a strict regular semigroup, and let $e\in E(S)$ and $x\in\oX$. Since $xx'Sxx'$ is a Clifford semigroup, $\Dcc_e\cap xx'Sxx'$ is either empty or an $\Hc$-class of $\Dcc_e$. Hence $\Ga_e$ is arrow-transitive by Lemma \ref{atrans}.

Assume now that $\Ga_e$ is arrow transitive for all $e\in E(S)$. By Lemma \ref{atrans}, for each $x\in\oX$, $\Dcc_e\cap xx'Sxx'$ is either empty or a subgroup for all $e\in E(S)$. So, each $\Dc$-class of $xx'Sxx'$ is a group and $xx'Sxx'$ is a Clifford semigroup. Let now $f$ be an idempotent of $S$ and let $x\in\oX$ such that $f\omega^r xx'$. Note that $\varphi:fSf\to fSfxx',\,s\mapsto sxx'$ is an isomorphism. Hence $fSfxx'$ is a regular subsemigroup of $xx'Sxx'$ and, consequently, it is a Clifford semigroup too. We have shown that all local submonoids $fSf$ of $S$ are Clifford semigroups. So, $S$ is a strict regular semigroup.
\end{proof}

No automorphism of an rliw-graph $\Ga$ can send a vertex $\av$ into another vertex at line-distance 2 from $\av$ as, otherwise, this would contradict the fact of $\Ga$ being reduced. Hence, we will consider also the following weak version of vertex-transitivity: an liw-graph is \emph{almost vertex-transitive} if for each pair $(\av,\bv)$ of vertices of $\Ga$, there exists $\varphi\in\Aut(\Ga)$ such that $\av\varphi$ is at line-distance at most 2 from $\bv$. This condition is clearly equivalent to say that for each pair of vertices $(\av,\bv)$ of different sides, there exists $\varphi\in\Aut(\Ga)$ such that $\av\varphi$ and $\bv$ are connected by a line. The \emph{orbit} $O(\av)$ of a vertex $\av\in\Ga$ is the set
$$O(\av)=\{\av'\,:\;\av'=\av\varphi\mbox{ for some }\varphi\in\Aut(\Ga)\}\,.$$
The orbit $O(\av)$ is \emph{full} if each vertex $\bv\in\Ga$ with $\sb(\bv)\neq \sb(\av)$ is connected by a line to some vertex of $O(\av)$. Observe that $\Ga$ is almost vertex-transitive \iff\ all its orbits are \emph{full}.

\begin{prop}
Let $a\in S$ and $e,f\in E(S)$ such that $e\Lc a\Rc f$. Then
\begin{itemize}
\item[$(i)$] $\Rcc_a$ is a subsemigroup of $S$ \iff\ the orbit of $\lv_a$ in $\Ga_e$ is full.
\item[$(ii)$] $\Lcc_a$ is a subsemigroup of $S$ \iff\ the orbit of $\rv_a$ in $\Ga_f$ is full.
\end{itemize}
\end{prop}

\begin{proof}
We prove only $(i)$ since $(ii)$ is its dual. Assume first that $\Rcc_a$ is a subsemigroup of $S$ and let $\rv_b$ be a right vertex of $\Ga_e$. Since each $\Hc$-class of $\Rcc_a$ is a group, there exists $b'\in V(b)\cap \Hcc_a$. Hence $b'b\Hc ab$ and there is a left-isomorphism $\varphi$ from $\A_{ab}=(\lv_a,\Ga_e,\rv_b)$ onto $\A_{b'b}=(\lv_{b'},\Ga_e,\rv_b)$. Thus $\lv_{b'}\in O(\lv_a)$ and $(\lv_{b'},\rv_b)\in\ol{E}(\Ga_e)$ because $b'b\in E(S)$. We have shown that the orbit of $\lv_a$ in $\Ga_e$ is full. 

Assume now that $O(\lv_a)$ is full in $\Ga_e$ and let $s\in \Rcc_a$. Choose $b\in\Rcc_e$ such that $s=ab$. Hence $(\lv_a,\Ga_e,\rv_b)$ is a representation of $\A_s$. Since $O(\lv_a)$ is full, there exists $\varphi\in\Aut(\Ga_e)$ such that $(\lv_{a_1},\rv_b)$ is a line of $\Ga_e$ for $\lv_{a_1}=\lv_a\varphi$. Thus $\A_{a_1b}=(\lv_{a_1},\Ga_e,\rv_b)$ is left isomorphic to $(\lv_a,\Ga_e,\rv_b)$, and $a_1b\Rc ab$ by Proposition \ref{Gclas}$.(i)$. Consequently $a_1b\Hc ab$ because $a_1b,ab\in\Lcc_b$. Finally, observe that $a_1b\in E(S)$ since $(\lv_{a_1},\rv_b)$ is a line in $\Ga_e$ (Proposition \ref{id_inv}). We have shown that each $\Hc$-class of $\Rcc_a$ has an idempotent. Therefore, $\Rcc_a$ is a subsemigroup of $S$.
\end{proof}

The next two results are now obvious consequences of this last proposition.

\begin{cor}\label{ccs}
Let $e\in E(S)$. Then $\Dcc_e$ is a completely simple subsemigroup of $S$ \iff\ $\Ga_e$ is almost vertex-transitive.
\end{cor}

\begin{cor}
The locally inverse semigroup $S$ is completely regular (that is, a normal band of groups) \iff\ $\Ga_e$ is almost vertex-transitive for all $e\in E(S)$.
\end{cor}

A $\Dc$-class $D$ of $S$ is a rectangular group if $D$ is a completely simple subsemigroup of $S$ and the product of any two idempotents of $D$ is another idempotent of $D$. If we look carefully to the proof of Proposition \ref{gen_inv}, we conclude that the product of any two idempotents of a completely simple $\Dc$-class $\Dcc_e$ is another idempotent \iff\ the line-connected components of $\Ga_e$ are complete bipartite graphs. The next corollary follows now from Corollary \ref{ccs}.

\begin{cor}
The $\Dc$-class $\Dcc_e$ is a rectangular group \iff\ $\Ga_e$ is almost vertex-transitive and its line-connected components are all (isomorphic) complete bipartite graphs.
\end{cor}

We say that the line-connected components of $\Ga_e$ are almost vertex-transitive if for any two line-connected vertices $\av$ and $\bv$ of $\Ga_e$ with different sides, there exists $\varphi\in\Aut(\Ga_e)$ such that $\av\varphi$ is connected to $\bv$ by a line. A locally inverse semigroup $S$ is \emph{$E$-solid} if all $\Hc$-classes of its core $C(S)$ are groups. In fact, for $S$ to be $E$-solid, it is enough to show that $\Hcc_{fg}$ is a group for all $e,f,g\in E(S)$ such that $f\Lc e\Rc g$. The next result relates $E$-solid locally inverse semigroups with liw-graphs with almost vertex-transitive line-connected components.

\begin{prop}
The locally inverse semigroup $S$ is $E$-solid \iff\  each line-connected components of $\Ga_e$ is almost vertex-transitive, for all $e\in E(S)$.
\end{prop}

\begin{proof}
Let $S$ be an $E$-solid locally inverse semigroup. Consider $e\in E(S)$ and $a,b\in\Dcc_e$ such that $a\Lc e\Rc b$ and $\lv_a$ is line-connected to $\rv_b$ in $\Ga_e$. Then $(\lv_a,\Ga_e,\rv_b)$ is a representation of $\A_{ab}$ and $ab\in C(S)$ by Proposition \ref{core}. Since $S$ is an $E$-solid locally inverse semigroup, $\Hcc_{ab}$ contains an idempotent $f$ and there exists $b'\in V(b)\cap \Hcc_a$. Hence $\A_f=(\lv_{b'},\Ga_e,\rv_b)$ is left-isomorphic to $\A_{ab}$ and $\rv_b$ is adjacent to $\lv_{b'}\in O(\lv_a)$. We can now conclude that the line-connected components of $\Ga_e$ are almost vertex-transitive.

Assume now that, for all $e\in E(S)$, the line-connected components of $\Ga_e$ are almost vertex-transitive. Let $e$, $f$ and $g$ be three idempotents of $S$ such that $f\Lc e\Rc g$.  We can represent $\A_{fg}$ as $(\lv_f,\Ga_e,\rv_g)$ and
$$\lv_f\rv_e\lv_e\rv_g$$
is a walk in $\Ga_e$ from $\lv_f$ to $\rv_g$. Thus, there exists $\varphi\in\Aut(\Ga_e)$ such that $(\lv_a,\rv_g)\in\ol{E}(\Ga_e)$ for $\lv_a=\lv_f\varphi$ since the line-connected components of $\Ga_e$ are almost vertex-transitive. So $\A_{ag}=(\lv_a,\Ga_e,\rv_g)$ is left-isomorphic to $\A_{fg}$ and $ag\Rc fg$. But both $ag$ and $fg$ belong to $\Lcc_g$. Hence $ag\Hc fg$. Finally, $ag\in E(S)$ because $(\lv_a,\rv_g)\in\ol{E}(\Ga_e)$. We have shown that $\Hcc_{fg}$ contains an idempotent and, therefore, $S$ is an $E$-solid locally inverse semigroup.
\end{proof}

We end this paper by considering the case where each vertex is the endpoint of exactly one arrow.

A presentation $P=\langle X;R\rangle$ is called \emph{$X$-straight} if $\la_u=\la_v$ and $\tau_u=\tau_v$ for all $(u,v)\in R$. Note that $\la_u=\la_v$ and $\tau_u=\tau_v$ for all $(u,v)\in\varepsilon$ where $\varepsilon$ is the Auinger's congruence on $\wX^+$. Then $\la_u=\la_v$ and $\tau_u=\tau_v$ for all $(u,v)\in\mu$ \iff\ $P$ is an $X$-straight presentation. If
$$S_{x,y}=\{u\in\wX^+\,:\;\la_u=x\mbox{ and } \tau_u=y\}$$
for $x,y\in\oX$, then $S$ is the disjoint union of the subsemigroups $S_{x,y}/\mu$ \iff\ $P$ is $X$-straight. By the description of the maximal subsemilattices of the bifree locally inverse semigroup $\wX^+/\varepsilon$ given in \cite{ol17}, it is not hard to verify that $E(S_{x,y}/\mu)$ are the maximal subsemilattices of $S$ if $P$ is $X$-straight.

A locally inverse semigroup is called \emph{straight} if its maximal subsemilattices are disjoint. Hence $S$ is straight if $P$ is an $X$-straight presentation. However, the converse is not true: if $S$ is an inverse semigroup, then $S$ is obviously straight but the presentation $P$ is not $X$-straight. 

\begin{prop}
The sets $\cb(\av)$ are singleton sets for all $e\in E(S)$ and $\av\in V(\Ga_e)$  \iff\ $P$ is $X$-straight.
\end{prop}

\begin{proof}
Assume that $\cb(\av)$ are singleton sets for all $e\in E(S)$ and $\av\in V(\Ga_e)$, and let $(u,v)\in R$. Let $x=\la_u$, $y=\la_v$ and $e\in E(S)$ such that $e\Rc u\mu=v\mu$. Then $e\in (x\wedge x')S\cap (y\wedge y')S$, and $(\lv_e,x,\rv_{ex})$ and $(\lv_e,y,\rv_{ey})$ are two arrows of $\Ga_e$. Hence $x=y$. Similarly, we show that $\tau_u=\tau_v$, and so $P$ is $X$-straight. 

Assume now that $P$ is $X$-straight, let $e\in E(S)$ and choose $a\Lc e$ and $x_1,x_2\in\cb(\lv_a)$. Then there are arrows $\ev_1=(\lv_a,x_1,\rv_{b_1})$ and $\ev_2=(\lv_a,x_2,\rv_{b_2})$ in $\Ga_e$, and $a\in (x_1\wedge x_1')S\cap (x_2\wedge x_2')S$. Hence $a=u\mu=v\mu$ for some $u\in x_1\wX^*$ and $v\in x_2\wX^*$, and 
$$a\in (S_{x_1,y_1}/\mu)\cap (S_{x_2,y_2}/\mu)$$
for $y_1=\tau_u$ and $y_2=\tau_v$. Thus $x_1=x_2$ because $P$ is $X$-straight, and so $\cb(\lv_a)$ is a singleton set. We can conclude that $\cb(\rv_b)$ is a singleton set for all $b\Rc e$ similarly.
\end{proof}

\vspace*{.5cm}

\noindent{\bf Acknowledgments}: This work was partially supported by CMUP (UID/ MAT/00144/2019), which is funded by FCT with national (MCTES) and European structural funds through the programs FEDER, under the partnership agreement PT2020.

\end{document}